\documentclass[a4paper,fleqn]{cas-sc}

\usepackage{graphicx}%
\usepackage{multirow}%
\usepackage{amsmath,amssymb,amsfonts}%
\usepackage{amsthm}%
\usepackage{mathrsfs}%
\usepackage[title]{appendix}%
\usepackage{xcolor}%
\usepackage{textcomp}%
\usepackage{manyfoot}%
\usepackage{booktabs}%
\usepackage{algorithm}%
\usepackage{algorithmicx}%
\usepackage{algpseudocode}%
\usepackage{listings}%

\usepackage{subfigure}

\newcommand{\p}{\partial}
\newcommand{\bb}{\boldsymbol}
\newcommand{\bu}{\textbf{u}}
\newcommand{\bbf}{\boldsymbol{F}}

\newtheorem{mydef}{Definition}[subsection]

\newtheorem{thm}{Theorem}[section]

\newtheorem{lemma}[thm]{Lemma}

\newcommand{\bp}{\bb{P}}
\newcommand{\B}{\bb{B}}
\newcommand{\bC}{\textbf{C}}
\newcommand{\bF}{\textbf{F}}
\newcommand{\bD}{\textbf{D}}

\newcommand{\pll}{p_{\parallel}}
\newcommand{\per}{p_{\perp}}
\newcommand{\bhat}{\boldsymbol{b}}
\newcommand{\con}{\textbf{U}}

\newcommand{\sgn}{\text{sgn}}

\newcommand{\iph}{{i+\frac{1}{2}}}

\newcommand{\rev}[1]{{\textcolor{magenta}{#1}}}

\def\tsc#1{\csdef{#1}{\textsc{\lowercase{#1}}\xspace}}
\tsc{WGM}
\tsc{QE}
\tsc{EP}
\tsc{PMS}
\tsc{BEC}
\tsc{DE}

\begin{document}
\let\WriteBookmarks\relax
\def\floatpagepagefraction{1}
\def\textpagefraction{.001}
\shorttitle{CGL Equations: Eigensystem Analysis and Applications to 1-D Test Problems}
\shortauthors{Singh et~al.}


\title [mode = title]{Chew, Goldberger \& Low Equations: Eigensystem Analysis and Applications to One-Dimensional Test Problems}                      

\author[1]{Chetan Singh}[]
\cormark[1]
\ead{maz218518@iitd.ac.in}

\credit{Formal analysis, Investigation, Methodology, Software, Validation, Writing – original draft}

\affiliation[1]{organization={Department of Mathematics, Indian Institute of Technology Delhi}, 
                country={India}}

\author[2]{Deepak Bhoriya}[]
\ead{dbhoriy2@nd.edu}

\credit{Formal analysis, Investigation, Methodology, Software, Writing – original draft}

\affiliation[2]{organization={Physics Department, University of Notre Dame},
	country={USA}}
	
\author[1]{Anshu Yadav}[]
\ead{anshuyadav132013@gmail.com}

\credit{Formal analysis, Investigation }

\author[1]{Harish Kumar}[]
\ead{hkumar@iitd.ac.in}

\credit{Conceptualization, Methodology, Supervision, Writing – original draft, Writing – review \& editing }

\author[2,3]{Dinshaw S. Balsara}[]
\ead{dbalsara@nd.edu}

\credit{Conceptualization, Investigation, Methodology, Software, Writing – original draft}

\affiliation[3]{organization={ACMS, University of Notre Dame},
	country={USA}}

\cortext[cor1]{Corresponding author.}

\begin{abstract}
Chew, Goldberger \& Low (CGL) equations describe one of the simplest plasma flow models that allow anisotropic pressure, i.e., pressure is modeled using a symmetric tensor described by two scalar pressure components, one parallel to the magnetic field, another perpendicular to the magnetic field. The system of equations is a non-conservative hyperbolic system. In this work, we analyze the eigensystem of the CGL equations. We present the eigenvalues and the complete set of right eigenvectors. We also prove the linear degeneracy of some of the characteristic fields. Using the eigensystem for CGL equations, we propose \rev{HLL and HLLI Riemann solvers} for the CGL system. Furthermore, we present the  AFD-WENO schemes up to the seventh order in one dimension and demonstrate the performance of the schemes on several one-dimensional test cases.
\end{abstract}


\maketitle

\section{Introduction}
Equations of magnetohydrodynamics are one of the simplest plasma flow models. The model treats plasma like a single fluid having the same density, velocity, and pressure. The model is indeed suitable for many plasma dynamics applications. However, in for many applications, the additional effects are essential, and they need to be incorporated in the model, e.g., particle collisions~\cite{leboeuf1979magnetohydrodynamic}, Hall current~\cite{mininni2003role}, electric current dissipation~\cite{montgomery1988minimum}, heat conduction~\cite{ju2013asymptotic} and pressure anisotropy~\cite{meng2012classical,huang2019six,Hirabayashi2016new}. 

In low-density magnetized plasma, where particle gyration and field-aligned motion are not connected by collisions, pressure anisotropy naturally develops. For example, in space and astrophysical phenomena~\cite{quataert2003radiatively,ichimaru1977bimodal,narayan1994advection} collisionless plasmas are considered, which are characterized by extremely hot and dilute gas, causing the mean free path of charged particles to exceed the system's scale size. The solar wind~\cite{chandran2011incorporating,hollweg1976collisionless,zouganelis2004transonic} and the magnetosphere of earth~\cite{burch2016magnetospheric,karimabadi2014link,dumin2002corotation,heinemann1999role} are also examples of collisionless plasmas flows. Due to this, the pressure needs to be modeled using a symmetric positive tensor. 

Several fluid and plasma flow models have considered the tensorial description of the pressure in plasma, ~\cite{sen2018entropy,meena2017positivity,meena2019robust,hakim2008extended,Hirabayashi2016new}. One of the simplest plasma flow models, which does not enforce {\em local thermodynamic equilibrium}, is proposed by Chew, Goldberger, and Low ~\cite{chew1956boltzmann}. The model is also known as the {\em double-adiabatic model} and is valid for single-fluid collisionless plasma with a strong magnetic field. In this model, pressure is modeled by a magnetic field-rotated symmetric pressure tensor (see \cite{hunana2019brief} for the detailed discussion), which is defined by two scalar components: parallel pressure component and perpendicular pressure component. 

The CGL equations are a set of hyperbolic PDEs containing non-conservative products ~\cite{singh2024entropy,hunana2019brief,huang2019six}. Due to this, developing higher-order numerical methods is highly nontrivial. One key difficulty is choosing the appropriate path, which is often unknown~\cite{dal1995definition}. In addition, the numerical solutions depend on the numerical viscosity~\cite{abgrall2010comment} of the method. In practice, a linear path is often considered to design approximate Riemann solvers~\cite{meng2012classical}. More recently, in \cite{bhoriya2024}, authors have designed higher-order finite volume WENO-based schemes. In another work \cite{singh2024entropy}, authors have presented higher-order finite-difference entropy stable schemes for the CGL system.

Another difficulty in developing numerical methods for CGL equations is that the current literature does not present a complete set of eigenvectors for the CGL system. In this work, we present a complete set of eigenvectors for the CGL system. We also analyze linearly degenerate characteristic fields. To demonstrate the application of the presented eigensystem, we use AFD-WENO schemes for one-dimensional CGL equations based on the recent work in \cite{balsara2023efficient1,balsara2024efficientnc}. This also involves the design of \rev{HLL and HLLI Riemann solvers} for CGL equations following \cite{dumbser2016new}. These schemes are tested using several one-dimensional test cases. 

The rest of the article is organized as follows: In Section \ref{sec:cgl}, we present the CGL equations. We also present a complete set of eigenvalues and eigenvectors of the system and discuss the linear degeneracy of some characteristic fields. In Section \ref{sec:afd_weno}, following \cite{balsara2024efficient,balsara2024efficientnc}, we present the AFD-WENO schemes for the CGL equations. \rev{We also present the HLL and HLLI  Riemann solvers for the CGL equations}. The numerical results are presented in Section \ref{sec:num}. We conclude in Section \ref{sec:con}.
\section{Chew, Goldberger \& Low (CGL) equations for plasma flows}
\label{sec:cgl}
Following \cite{chew1956boltzmann,meng2012classical,hunana2019brief,singh2024entropy,bhoriya2024}, the CGL equations for the variables $\con = (\rho, \rho\bu, \pll-\per, E, \B)^\top$, can be written as 
\begin{align} 
	&\frac{ \p\rho}{\p t}+\nabla\cdot(\rho \bu)=0,&\label{1}\\
	&\frac{\p \rho \bu}{\p t}+\nabla\cdot\left[\rho \bu\bu+\per \textbf{I}+(\pll-\per)\bhat\bhat-\frac{1}{4\pi}\left(\B\B-\frac{|\B|^2}{2}\textbf{I}\right)\right]=0,&\label{2}\\
	&\frac{\p (\pll-\per)}{\p t}+\nabla\cdot\left((\pll-\per) \bu\right)+\textcolor{red}{(2\pll+\per)\bhat\cdot\nabla \bu\cdot\bhat - \per\nabla\cdot\bu}=	\frac{\per-\pll}{\tau},&\label{3}\\
	&\frac{\p E}{\p t}+\nabla\cdot\bigg[\bu\left(E + \per+ \frac{|\B|^2}{8 \pi}\right) +\bu\cdot\left((\pll-\per)\bhat\bhat-\frac{\B\B}{4 \pi}\right)\bigg]=0,&\label{4}\\
	&\frac{\p \B}{\p t}+\nabla \times [-(\bu \times \B)]=0.\label{5}
\end{align}
Here $\rho$ is the density, $\bu = (u_x, u_y, u_z)^\top$ is the velocity vector, and $\B=(B_x, B_y, B_z)^\top$ is the magnetic field vector. The direction of the magnetic field vector is denoted by unit vector $\bhat$, which  is given by,
\begin{align*}
	\bhat=\frac{\B}{|\B|}=b_{x}\hat{i}+b_{y}\hat{j}+b_{z}\hat{k}.
\end{align*}
The symmetric pressure tensor $\bp$ is given by
\begin{align*}
	\bp&=\pll\bhat\bhat+\per(\textbf{I}-\bhat\bhat) \\
	&=\pll\begin{pmatrix}
		b_{x}b_{x} & b_x b_y & b_{x}b_{z}\\
		b_{y}b_{x} & b_{y}b_{y} & b_{y}b_{z}\\
		b_{z}b_{x} & b_{z}b_{y} & b_{z}b_{z}
	\end{pmatrix}+\per\begin{pmatrix}
		1-b_{x}b_{x} & -b_x b_y & -b_{x}b_{z}\\
		-b_{y}b_{x} & 1-b_{y}b_{y} & -b_{y}b_{z}\\
		-b_{z}b_{x} &  -b_{z}b_{y} & 1-b_{z}b_{z}
	\end{pmatrix}.
\end{align*}
The quantities $\pll$ and $\per$ are scalar components of the pressure tensor given by $\pll=\bp:\bhat\bhat$ and $\per=\bp:\textbf{I}-\bhat\bhat$. \rev{Here, we have ignored the effects of Finite Larmor Radius (FLR) corrections, non-gyrotropic stress tensor,
and gyroviscous stress tensor.} The average scalar pressure $p$ is given by $p=\frac{2 \per+\pll}{3}$. The system of equations is closed by assuming the ideal equation of state, 
\begin{align}
	\label{eq:EOS}
	E = \frac{\rho |\bu|^2}{2}+\frac{|\B|^2}{8 \pi}+\frac{3 p}{2}.
\end{align}
The system ~\eqref{1}-\eqref{5} contains a non-conservative product in eqn.~\eqref{3} (the non-conservative terms are colored in red). Furthermore, we have considered an equation for $\pll-\per$ as this allows us to evolve the anisotropy in the pressure components. To compute the {\em MHD} solutions (numerical solution close to isotropic MHD limit $\pll\approx\per$), we have also considered an additional source term in \eqref{3}, which is used when we want to compute {\em isotropic} (MHD) limit. 

In one dimension, similar to the case of MHD equations, the magnetic field component $B_x$ can be considered constant. Consequently, the system \eqref{1}-\eqref{5} can be written as,
\begin{equation}\label{7}
	\frac{\p \con}{\p t} + \frac{\p \bF}{\p x} + \bC(\con)\frac{\p \con}{\p x} = \textbf{S}(\con),
\end{equation}
where,
\begin{align*}
	\con=
	\begin{pmatrix}
		\rho\\
		\rho u_{x}\\
		\rho u_{y}\\
		\rho u_{z}\\
		\pll-\per\\
		E\\
		B_{y}\\
		B_{z}\\
	\end{pmatrix},~~
	\bF =
	\begin{pmatrix}
		\rho u_{x}\\
		\rho u_{x}^{2}+\per-\frac{1}{4\pi}B_{x}^{2} +\frac{|\B|^{2}}{8\pi} +(\pll-\per)b_{x}^{2}\\
		\rho u_{x}u_{y}-\frac{1}{4\pi}B_{x}B_{y}+(\pll-\per)b_{x}b_{y}\\
		\rho u_{x}u_{z}-\frac{1}{4\pi}B_{x}B_{z}+(\pll-\per)b_{x}b_{z}\\
		(\pll-\per) u_{x}\\
		\left(E+ \per+ \frac{|\B|^2}{8 \pi}\right)u_{x} -\frac{B_{x}}{4\pi}(\B\cdot\bu) + (\pll-\per)(\bu\cdot\bhat) b_{x}\\
		u_{x}B_{y} - u_{y}B_{x}\\
		u_{x}B_{z} - u_{z}B_{x}\\ 
	\end{pmatrix},
\end{align*}
\begin{align*}
	\textbf{C}(\con)=
	\begin{pmatrix}
		0 & 0 & 0 & 0 & 0 & 0 & 0 & 0 \\
		0 & 0 & 0 & 0 & 0 & 0 & 0 & 0 \\
		0 & 0 & 0 & 0 & 0 & 0 & 0 & 0 \\
		0 & 0 & 0 & 0 & 0 & 0 & 0 & 0 \\
		\frac{\per u_x}{\rho}-\frac{(2\pll+\per)(\bhat\cdot\bu)b_x}{\rho} & \frac{(2\pll+\per)b^2_{x}}{\rho}-\frac{\per}{\rho} & \frac{(2\pll+\per)b_{x}b_y}{\rho} & \frac{(2\pll+\per)b_{x}b_z}{\rho}
		& 0 & 0 & 0 & 0 \\
		0 & 0 & 0 & 0 & 0 & 0 & 0 & 0 \\
		0 & 0 & 0 & 0 & 0 & 0 & 0 & 0 \\
		0 & 0 & 0 & 0 & 0 & 0 & 0 & 0 \\
	\end{pmatrix},
\end{align*}
and
\begin{align}
	\label{eq:source}
	\textbf{S}(\con)= \left(0, 0,	0,	0,\frac{\per-\pll}{\tau}, 0, 0,	0\right)^\top .
\end{align}

\subsection{Eigenvalues and Eigenvectors}
\label{subsec:eigen}
Let us define the primitive variables as $$\textbf{W} = \left\{\rho, u_{x}, u_{y}, u_{z}, \pll, \per, B_{y}, B_{z}\right\}.$$ Then  Eqn.~\eqref{7} can be written in quasilinear form,
\begin{equation}\label{9}
	\frac{\p \textbf{W}}{\p t} + \textbf{A}\frac{\p \textbf{W}}{\p x} = \textbf{S}
\end{equation}
where the matrix $\textbf{A} = \frac{\p\textbf{W}}{\p\con}\frac{\p \bF}{\p\textbf{W}} +\frac{\p\textbf{W}}{\p\con}\bC(\con)\frac{\p \con}{\p\textbf{W}}$ is given by,
\begin{align*}
	\begin{pmatrix}
		u_x & \rho & 0 & 0 & 0 & 0 & 0 & 0\\
		0 & u_x & 0 & 0 & \frac{b_x^2}{\rho} & \frac{(1-b_x^2)}{\rho} & \frac{B_y}{4 \pi \rho}-\frac{2 b_x^2 b_y\Delta p}{\rho |\B|} & \frac{B_z}{4 \pi \rho}-\frac{2 b_x^2    b_z\Delta p}{\rho |\B|} \\
		0 & 0 & u_x & 0 & \frac{b_x b_y}{\rho} & -\frac{b_x b_y}{\rho}  & \bigg(\frac{b_x(1-2 b_y^2)\Delta p}{\rho |\B|}-\frac{B_x}{4\pi\rho}\bigg) & \frac{-2 b_x b_y b_z\Delta p}{\rho |\B|}\\
		0 & 0 & 0 & u_x & \frac{b_x b_z}{\rho} & -\frac{b_x b_z}{\rho} & \frac{-2 b_x b_y b_z\Delta p}{\rho |\B|} & \left(\frac{b_x(1-2 b_z^2)\Delta p}{\rho |\B|}-\frac{B_x}{4\pi\rho}\right)\\
		0 & p_\parallel(1+2 b_x^2) & 2p_\parallel b_x b_y & 2p_\parallel b_x b_z & u_x & 0 & 0 & 0 \\
		0 & p_\bot(2-b_x^2) & -p_\bot b_x b_y & -p_\bot b_x b_z & 0 & u_x  & 0 & 0\\
		0 & B_y & -B_x & 0 & 0 & 0  & u_x & 0\\
		0 & B_z & 0 & -B_x & 0 & 0  & 0 & u_x
	\end{pmatrix},
\end{align*}
with, $\Delta p=(\pll-\per)$. Following~\cite{kato1966propagation,meng2012classical,singh2024entropy}, the characteristic equation for the matrix $\textbf{A}$ is given by
\begin{align*}
	(\lambda-u_x)^2\xi_1(\lambda) \xi_2(\lambda) =0,
\end{align*}
where, 

$$\xi_1(\lambda)=(\lambda-u_x)^2-\left(\frac{B_x^2}{4\pi\rho}-\frac{b_x^2(\pll-\per)}{\rho}\right),$$ 
and
\begin{align*}
	\xi_2(\lambda)=(\lambda-u_x)^4-&\left(\frac{|\B|^2}{4\pi\rho}+\frac{2\per}{\rho}+\frac{b_x^2(2\pll-\per)}{\rho}\right)(\lambda-u_x)^2 \\
	&+ \left(\frac{3\pll B_x^2}{4\pi \rho^2}-\frac{3 b_x^4 \pll(\pll-\per)}{\rho^2} + \frac{b_x^2(1-b_x^2)(6\pll-\per)}{\rho^2}\right).
\end{align*}
Two roots  of the characteristic equations  are $\lambda=u_x,~u_x$ corresponding to the \textit{entropy wave} and the \textit{pressure anisotropy wave}. The roots of $\xi_1(\lambda)$ represent a pair of \textit{Alfv\'{e}n waves},
\[\lambda=u_x\pm c_a~~~~\text{where,}~c_a= \sqrt{{\frac{B_x^2}{4 \pi\rho}-\frac{(\Delta p)b_x^2}{\rho}}}.\]

The equation $\xi_2(\lambda)=0$ has four roots that are related to \textit{magnetosonic waves} and are give by,
\[\lambda=u_x\pm c_f,~u_x\pm c_s,\]
where,
\begin{align*}
	c_{f,s}= &\frac{1}{\sqrt{2\rho}}\Bigg[\frac{{|\B|}^2}{4\pi}+2p_\bot +b_x^2(2p_\parallel-p_\bot) \pm \Bigg\{\left(\frac{{|\B|}^2}{4 \pi}+2p_\bot +b_x^2(2p_\parallel-p_\bot)\right)^2\\
	&+4\left(p_\bot^2 b_x^2(1-b_x^2)-3p_\parallel p_\bot b_x^2(2-b_x^2)+3p_\parallel^2 b_x^4-\frac{3B_x^2 p_\parallel}{ 4 \pi}\right)\Bigg\}^{\frac{1}{2}}\Bigg]^{\frac{1}{2}} . 
\end{align*}
 Each sign in $c_{f,s}$ is associated with two waves that propagate symmetrically in opposite directions relative to the bulk flow speed $u_x$.The positive sign corresponds to the {\em fast magnetosonic waves}, and the negative sign corresponds to {\em slow magnetosonic waves}. The complete set of eigenvalues of the matrix $\textbf{A}$ are given by,
\[\bb{\Lambda}=\left\{u_x,~ u_x,~ u_x\pm c_a,~ u_x\pm c_f,~ u_x\pm c_s\right\}.\]

\rev{Let us consider the following notation,
\begin{align*}
\alpha_1&=|\B|^4 - 16 B_x^2 \pi p_\parallel,\\
\Upsilon_1 &= \sqrt{\alpha_1^2 + 
	8 \alpha_1(B_x^2 + 2 (B_y^2+B_z^2)) \pi  p_\perp + 
	16 (B_x^4 + 8 B_x^2 (B_y^2+B_z^2) + 4 (B_y^2+B_z^2)^2) \pi^2 p_\perp^2},&\\
	&\text{and}&\\
\Upsilon_2 &= |\B|^4+4 \pi B_x^2 \left(2 p_\parallel+ p_\perp\right)+\left(B_y^2+B_z^2\right) \left(8 \pi p_\perp\right).&
\end{align*}
A simple calculation shows that,
\begin{align}
c_s~=~  \frac{ \sqrt{\Upsilon_2-\Upsilon_1}}{2|\B|\left(\sqrt{2\pi\rho}\right)}  ~\text{   and    }~
  c_f~=~  \frac{ \sqrt{\Upsilon_2+\Upsilon_1}}{2|\B|\left(\sqrt{2\pi\rho}\right)}.  \label{eq:vector_con_1}
\end{align}}
Hence, we need $\Upsilon_2\pm\Upsilon_1\ge0$ for the system to be hyperbolic. Following \cite{kato1966propagation}, we consider pressure limits
$$P_m=\frac{p_{\bot}^2}{6p_{\bot}+\frac{3|B|^2}{4\pi}},\text{    and   }P_M=\frac{|B|^2}{4\pi}+p_{\bot}$$
and the solution set
\begin{equation}\label{10}
	\Omega = \left\{ \con \in \mathbb{R}^8 | \rho>0,\pll>0,\per>0 \text{ and } P_{m} \leq p_{\parallel} \leq P_{M}.  \right\}
\end{equation}
Unlike the MHD system, in the case of CGL, slow waves may not be slower than the Alfv\'en waves. Hence, following \cite{kato1966propagation}, we need to divide the solution set $\Omega$ in three parts, as follows,

\begin{enumerate}
	\item[(i)] $P_{m}\le p_\parallel \le\frac{P_{M}}{4},\hspace{1.4cm} \text{ if }\hspace{1.6cm}c_s\le c_{a}\le c_f,$
	\item[(ii)] $\frac{P_{M}}{4}\le p_\parallel \le\frac{1}{4}P_{M}+\frac{3P_{m}}{4},\hspace{0.2cm} \text{ if }\hspace{1.5cm} c_s\le c_{a}< c_f,$
	\item[(iii)] $\frac{P_{M}}{4}+\frac{3P_{m}}{4}\le p_\parallel \le P_{M},\hspace{0.4cm} \text{ if }\hspace{1.5cm}c_a\le c_{s}< c_f.$
\end{enumerate}
Under the above assumptions, the CGL system is hyperbolic. Furthermore, the complete set of right eigenvectors is given as follows:

The right eigenvectors corresponding to the entropy wave $u_x$ and the pressure anisotropy wave $u_x$ are denoted by $R^e_{u_{x}}$ and $R^{p}_{u_{x}}$ are given below: 
\begin{gather*}
R^e_{u_{x}}=\begin{pmatrix}
		1\\
		0\\
		0\\
		0\\
		0\\
		0\\
		0\\
		0\\
	\end{pmatrix}\quad
	\text{ and }\quad R^{p}_{u_{x}}=
	\begin{pmatrix}
		0\\
		0\\
		0\\
		0\\
		(1-b_x^2)\Delta p\\
		b_x^2\Delta p-\frac{|\B|^2}{4\pi }\\
		B_y\\
		B_z\\
	\end{pmatrix}.
\end{gather*}
The right eigenvectors corresponding to the eigenvalues $u_x\pm c_a$ are given by
\begin{gather*}R_{u_x\pm c_a}=
	\begin{pmatrix}
		0\\
		0\\
		\pm{b_z}\sgn(B_x)\sqrt{\left(\frac{|\B|^2}{4 \pi \rho}-\frac{\Delta p}{\rho }\right)}\\
		\mp{b_y}\sgn(B_x)\sqrt{\left(\frac{|\B|^2}{4 \pi \rho}-\frac{\Delta p}{\rho }\right)}\\
		0\\
		0\\
		-{B_z}\\
		{B_y}\\
	\end{pmatrix}.
\end{gather*}
\rev{The right eigenvectors corresponding to the eigenvalues $u_x\pm c_f$ and $u_x\pm c_s$ are denoted by $R_{u_x\pm c_f}$ and $R_{u_x\pm c_s}$, respectively, and are given as,
\begin{gather*}                          
	R_{u_x\pm c_f}=    \begin{pmatrix}
		\alpha_2 b_x(\Upsilon_1+\Upsilon_3)\rho \\
		\pm\frac{\alpha_3 b_x}{2\sqrt{2}}(\Upsilon_1+\Upsilon_3)\sqrt{\frac{\Upsilon_2+\Upsilon_1}{ \pi \rho}}\\
		\mp\frac{\alpha_3 b_y}{2\sqrt{2} }(\Upsilon_4-\Upsilon_1)\sqrt{\frac{\Upsilon_1+\Upsilon_2}{\pi \rho}}\\
		\mp\frac{\alpha_3 b_z}{2\sqrt{2} }(\Upsilon_4-\Upsilon_1)\sqrt{\frac{\Upsilon_1+\Upsilon_2}{ \pi \rho}}\\
		-\alpha_2 p_\parallel b_x(\Upsilon_5-3\Upsilon_1)\\
		-\alpha_2 p_\perp b_x(\Upsilon_6-\Upsilon_1)\\
		B_x B_y\\
		B_x B_z
	\end{pmatrix}\quad
	\text{ and }\quad                           
	R_{u_x\pm c_s}=     \begin{pmatrix}  
		\alpha_2 b_x(-\Upsilon_1+\Upsilon_3)\rho \\
		\pm\frac{\alpha_3 b_x}{2\sqrt{2}}(\Upsilon_3-\Upsilon_1)\sqrt{\frac{\Upsilon_2-\Upsilon_1}{\pi \rho}}\\
		\mp\frac{\alpha_3 b_y}{2\sqrt{2} }(\Upsilon_4+\Upsilon_1)\sqrt{\frac{\Upsilon_2-\Upsilon_1}{ \pi \rho}}\\
		\mp\frac{\alpha_3 b_z}{2\sqrt{2} }(\Upsilon_4+\Upsilon_1)\sqrt{\frac{\Upsilon_2-\Upsilon_1}{ \pi \rho}}\\
		-\alpha_2 p_\parallel b_x(\Upsilon_5+3\Upsilon_1)\\
		-\alpha_2 p_\perp b_x(\Upsilon_6+\Upsilon_1)\\
		B_x{B_y}\\
		B_x{B_z}         
	\end{pmatrix}, 
\end{gather*}
where $\alpha_1,~\Upsilon_1,~\Upsilon_2$ are defined earlier and $\alpha_2=\frac{1}{8 \pi p_\bot |\B|}, ~\alpha_3=\frac{1}{8 \pi p_\bot |\B|^2},$ 
\begin{align*}
\Upsilon_3 &= 4 B_x^2 \pi (4 p_\parallel-p_\perp)-|\B|^4,&\\
\Upsilon_4&=|\B|^4-4 \pi B_x^2 \left(4 p_\parallel-3 p_\perp\right)+\left(B_y^2+B_z^2\right) \left(8 \pi p_\perp\right),&\\
\Upsilon_5 &= 4 \pi p_\perp \left(3 B_x^2+4 \left(B_y^2+B_z^2\right)\right)+3|\B|^4-48 B_x^2 \pi p_\parallel,&\\
\text{  and  }&&\\
\Upsilon_6&=|\B|^4-4 \pi B_x^2 \left(4 p_\parallel- p_\perp\right)-\left(B_y^2+B_z^2\right) \left(8 \pi p_\perp\right).
\end{align*}
} 

The characteristic speed $\lambda_i(\textbf{W})$ of a hyperbolic system defines a characteristic field known as the $\lambda_i$-characteristic field. Let us recall the following definition:

\begin{mydef}
A $\lambda_i$-characteristic ﬁeld is said to be linearly degenerate if
\[\nabla\lambda_i(\textbf{W})\cdot R_i(\textbf{W})=0,~~~~~\forall~\textbf{W}\in\Omega.\]
\end{mydef}
We have the following results:
\begin{lemma}
For the CGL system, the characteristic field corresponding to the entropy wave $u_x$, the pressure anisotropy wave $u_x$, and the \textit{Alfv\'{e}n waves} $u_x\pm c_a$ are linearly degenerate.
\end{lemma}
\begin{proof}
For entropy wave and pressure anisotropy wave,
\[\nabla u_x = \{0,1,0,0,0,0,0,0\}.\] Now its easy to see that,
\[\nabla u_x\cdot R^e_{u_{x}} = 0 \text{      and     }\nabla u_x\cdot R^p_{u_{x}} = 0.\]
For \textit{Alfv\'{e}n waves} $u_x\pm c_a$,
\[\nabla (u_x\pm c_a) = \left\{\mp\frac{b_x}{2\rho}\alpha_f,1,0,0,\mp\frac{b_x}{2\rho\alpha_f},\pm\frac{b_x}{2\rho\alpha_f},-\frac{b_x b_y (\Delta p)}{\rho|\B|\alpha_f},-\frac{b_x b_z (\Delta p)}{\rho|\B|\alpha_f}\right\}\]
where $\alpha_f=\sqrt{\left(\frac{|\B|^2}{4 \pi \rho}-\frac{\Delta p}{\rho }\right)}$.

Now, a simple calculation shows that,
\[\nabla (u_x\pm c_a)\cdot R_{u_x\pm c_a} = 0.\]
\end{proof}
 In \cite{huang2019six}, the authors present eigenvalues of semi-relativistic MHD involving anisotropic pressure. The expressions are analogous to that of the CGL system, provided we neglect the influence of electron pressure and relativistic effects.

In the next section, we will present the AFD-WENO schemes for the CGL equations.
\section{AFD-WENO Schemes for  CGL equations}
\label{sec:afd_weno}
Let us consider the one-dimensional CGL system \eqref{7} without source terms,
\begin{equation}\label{eq:CGL_1d}
\frac{\p \con}{\p t} + \frac{\p \bF}{\p x} + \bC(\con)\frac{\p \con}{\p x} = 0.
\end{equation}
To present the AFD-WENO scheme, let us consider the mesh presented in Figure \eqref{image_1}. For each $i^{th}$-zone,  $x_{i-\frac{1}{2}}$ and $x_{i+\frac{1}{2}}$ denotes the  zone boundaries and  $x_i = \frac{x_{i-\frac{1}{2}}+x_{i+\frac{1}{2}}}{2}$ is the zone center. In addition, the size of the zone is given by $\Delta x_i=x_{i+\frac{1}{2}} - x_{i-\frac{1}{2}}$. Here, we will assume a uniform mesh size. We want to evolve the mesh function $\con_i$, which is defined pointwise at the zone centers and represents the point value of the conservative variable $\con$. 
\begin{figure}[!htbp]
\begin{center}
	\subfigure{\includegraphics[width=4.0in, height=3.0in]{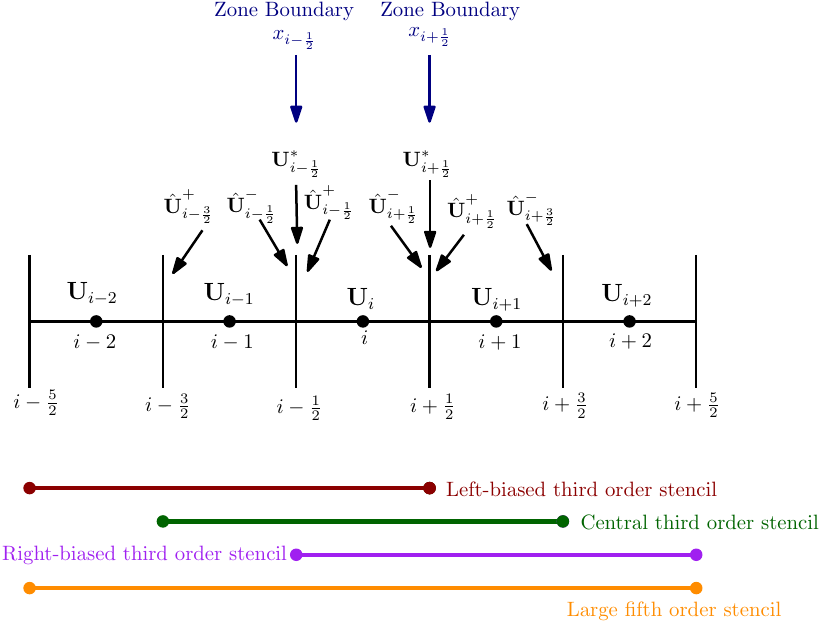}}
	\caption{Part of the mesh around zone $i$} 
	\label{image_1}    
\end{center}
\end{figure}
Using the high-order WENO-AO interpolation (see ~\cite{balsara2024efficient,balsara2016efficient,balsara2009efficient,balsara2020efficient,balsara2023efficient}.)  for the point values $\con_i$, we can now reconstruct the values,  $\hat{\con}^{+}_{i-\frac{1}{2}}$ and $\hat{\con}^{-}_{i+\frac{1}{2}}$ (we use characteristic variables-based reconstruction), which are the interpolated values at the left and right sides of the zone $i$ (see Figure \ref{image_1}). We note that $\hat{\con}^{-}_{i+\frac{1}{2}}$ is located at the left side of the zone boundary $x_{i+\frac{1}{2}}$ and $\hat{\con}^{+}_{i-\frac{1}{2}}$ is located at the right side of the zone boundary $x_{i-\frac{1}{2}}$. We carry out this reconstruction process in each zone.

Following ~\cite{balsara2024efficientnc}, the AFD-WENO scheme for \eqref{eq:CGL_1d} will have the following form:
\begin{align}\label{eq:AFD_NC}
\p_t\con_i \cong & -\frac{1}{\Delta x}\{\bD^{\ast-}(\hat{\con}^-_{i+\frac{1}{2}},\hat{\con}^+_{i+\frac{1}{2}})+\bD^{\ast+}(\hat{\con}^-_{i-\frac{1}{2}},\hat{\con}^+_{i-\frac{1}{2}})\}\nonumber\\
&-\rev{\frac{1}{\Delta x}\left\{\bF_{C}\left(\hat{\con}^-_{i+\frac{1}{2}}\right)-\bF_{C}\left(\hat{\con}^+_{i-\frac{1}{2}}\right)\right\}}-\frac{1}{\Delta x}\bC(\con_i)\left\{\hat{\con}^-_{i+\frac{1}{2}}-\hat{\con}^+_{i-\frac{1}{2}}\right\}\nonumber\\
&-\frac{1}{\Delta x}\bC(\con_i)\left[\textcolor{red}{-\frac{(\Delta x)^2}{24}[\p^2_x\hat{\con}]_{i+\frac{1}{2}}} \textcolor{blue}{+ \frac{7(\Delta x)^4}{5760}[\p^4_x\hat{\con}]_{i+\frac{1}{2}}} \textcolor{violet}{-\frac{31(\Delta x)^6}{967680}[\p^6_x\hat{\con}]_{i+\frac{1}{2}}}\right]\nonumber\\
&-\frac{1}{\Delta x}\bC(\con_i)\left[-\left[\textcolor{red}{-\frac{(\Delta x)^2}{24}[\p^2_x\hat{\con}]_{i-\frac{1}{2}}} \textcolor{blue}{+ \frac{7(\Delta x)^4}{5760}[\p^4_x\hat{\con}]_{i-\frac{1}{2}}} \textcolor{violet}{-\frac{31(\Delta x)^6}{967680}[\p^6_x\hat{\con}]_{i-\frac{1}{2}}}\right]\right]\nonumber\\
&-\frac{1}{\Delta x}\left[\textcolor{red}{-\frac{(\Delta x)^2}{24}[\p^2_x\hat{\bF}_{C}]_{i+\frac{1}{2}}} \textcolor{blue}{+ \frac{7(\Delta x)^4}{5760}[\p^4_x\hat{\bF}_C]_{i+\frac{1}{2}}} \textcolor{violet}{-\frac{31(\Delta x)^6}{967680}[\p^6_x\hat{\bF}_C]_{i+\frac{1}{2}}}\right]\nonumber\\
&-\frac{1}{\Delta x}\left[-\left[\textcolor{red}{-\frac{(\Delta x)^2}{24}[\p^2_x\hat{\bF}_C]_{i-\frac{1}{2}}} \textcolor{blue}{+ \frac{7(\Delta x)^4}{5760}[\p^4_x\hat{\bF}_C]_{i-\frac{1}{2}}} \textcolor{violet}{-\frac{31(\Delta x)^6}{967680}[\p^6_x\hat{\bF}_C]_{i-\frac{1}{2}}}\right]\right].
\end{align}

\rev{Here, the fluctuations $\bD^{\ast-}$ and $\bD^{\ast+}$ are based on the state $\con^*$ at each zone boundary resolved using the HLL or HLLI Riemann solver presented in the section \ref{subsec:HLL} and \ref{subsec:HLLI}.}

To complete the description of the scheme, we need to describe the $[\p^2_x\hat{\bbf}_{C}]_{i+\frac{1}{2}},~[\p^4_x\hat{\bbf}_{C}]_{i+\frac{1}{2}}$ and $[\p^6_x\hat{\bbf}_{C}]_{i+\frac{1}{2}}$ at the zone boundary $x_\iph$. We follow the procedure used in \cite{balsara2024efficient} and \cite{balsara2024efficientnc}.
\begin{figure}[!htbp]
\begin{center}
	\subfigure{\includegraphics[width=5.0in, height=3.0in]{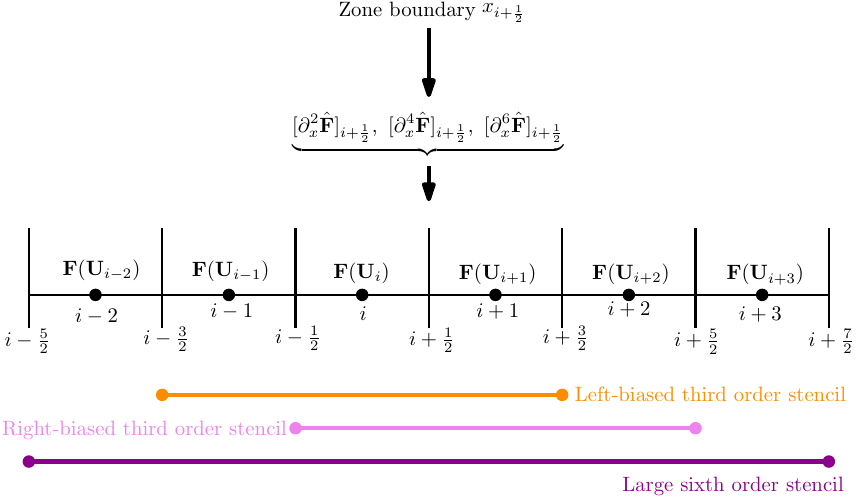}}
	\caption{part of the mesh around zone boundary $i+\frac{1}{2}$}
	\label{image_2}
\end{center}
\end{figure}
At each zone center, we compute the values $\bbf(\con_i)$. For a third-order AFD-WENO scheme, $[\p^2_x\hat{\bbf}_{C}]_{i+\frac{1}{2}}$ is needed at the zone boundary $x_{i+\frac{1}{2}}$ with third-order of accuracy. We use non-linear hybridization  to compute $[\p^2_x\hat{\bbf}_{C}]_{i+\frac{1}{2}}$ from the orange-colored stencil and the pink-colored stencil in Fig.\eqref{image_2}. For the fifth-order AFD-WENO scheme, $[\p^2_x\hat{\bbf}_{C}]_{i+\frac{1}{2}}$ and $[\p^4_x\hat{\bbf}_{C}]_{i+\frac{1}{2}}$ is needed at zone boundary $x_{i+\frac{1}{2}}$ with the fifth-order of accuracy. To achieve this, we use non-linear hybridization to compute  $[\p^2_x\hat{\bbf}_{C}]_{i+\frac{1}{2}}$ and $[\p^4_x\hat{\bbf}_{C}]_{i+\frac{1}{2}}$ from the orange-colored, pink-colored, and dark magenta-colored stencils in Fig.\eqref{image_2}. Similarly, we can obtain the higher order derivatives of the $\con$ at the zone boundaries.

With these descriptions, the Eqn.~\eqref{eq:AFD_NC} gives the final scheme. To obtain the third-order AFD-WENO scheme, we use Eqn.\eqref{eq:AFD_NC} and keep only the red terms while removing the blue and violet terms. Removing the violet terms and keeping only the red and blue terms results in the fifth-order AFD-WENO scheme. For the seventh order of the AFD-WENO scheme, we will keep all the terms.

\subsection{HLL Riemann Solver for CGL System}
\label{subsec:HLL}
To obtain the intermediate state $\con^{\ast}_{i+\frac{1}{2}}$ that overlies the zone boundary, we can use the HLL Riemann solver at $x_{i+\frac{1}{2}}$, using $\hat{\con}^{-}_{i+\frac{1}{2}}$ and $\hat{\con}^{+}_{i+\frac{1}{2}}$ as the left and right states, respectively (see Figure \ref{image_1}). Let us denote the slowest wave speed as $S^{L}_{i+1/2}$ and the fastest wave speed as $S^{R}_{i+1/2}$. Then we consider,
\begin{align*}
S^{L}_{i+1/2} =& min\left\{\Lambda_{min}\left(\hat{\con}^{-}_{i+\frac{1}{2}}\right),\Lambda_{min}\left(\hat{\con}^{+}_{i+\frac{1}{2}}\right),\Lambda_{min}\left(\frac{\hat{\con}^{-}_{i+\frac{1}{2}}+\hat{\con}^{+}_{i+\frac{1}{2}}}{2}\right)\right\},\\
S^{R}_{i+1/2} =& max\left\{\Lambda_{max}\left(\hat{\con}^{-}_{i+\frac{1}{2}}\right),\Lambda_{max}\left(\hat{\con}^{+}_{i+\frac{1}{2}}\right),\Lambda_{max}\left(\frac{\hat{\con}^{-}_{i+\frac{1}{2}}+\hat{\con}^{+}_{i+\frac{1}{2}}}{2}\right)\right\}.
\end{align*}
where $\Lambda_{min}$ is the minimum eigenvalue and $\Lambda_{max}$ is the maximum eigenvalue given in~\eqref{subsec:eigen}.
Following~\cite{dumbser2016new,balsara2023efficient1}, the intermediate state $\con^{\ast}_{i+\frac{1}{2}}$ for HLL Riemann solver is given below,

\vspace{0.5cm}
\begin{align}\label{HLL_1}
\con^{\ast}_{i+\frac{1}{2}} =& \frac{\left(S^{R}_{i+1/2}\hat{\con}^{+}_{i+\frac{1}{2}}-S^{L}_{i+1/2}\hat{\con}^{-}_{i+\frac{1}{2}}\right)-\left(\bF_{C}\left(\hat{\con}^{+}_{i+\frac{1}{2}}\right)-\bF_{C}\left(\hat{\con}^{-}_{i+\frac{1}{2}}\right)\right)}{S^{R}_{i+1/2}-S^{L}_{i+1/2}}\nonumber\\
&-\frac{\tilde{\bC}\left(\hat{\con}^{-}_{i+\frac{1}{2}},\con^{\ast}_{i+\frac{1}{2}}\right)\left(\con^{\ast}_{i+\frac{1}{2}}-\hat{\con}^{-}_{i+\frac{1}{2}}\right)   +\tilde{\bC}\left(\con^{\ast}_{i+\frac{1}{2}},\hat{\con}^{+}_{i+\frac{1}{2}}\right)\left(\hat{\con}^{+}_{i+\frac{1}{2}}-\con^{\ast}_{i+\frac{1}{2}}\right)}{S^{R}_{i+1/2}-S^{L}_{i+1/2}}.
\end{align}
Here, we need to define $\tilde{\bC}\left(\hat{\con}^{-}_{i+\frac{1}{2}},\con^{\ast}_{i+\frac{1}{2}}\right)$ and $\tilde{\bC}\left(\con^{\ast}_{i+\frac{1}{2}},\hat{\con}^{+}_{i+\frac{1}{2}}\right)$. This depends on the choice of path connecting the states. Assuming the linear paths $\psi^{-}(\xi) = \hat{\con}^{-}_{i+\frac{1}{2}}+\xi\left(\con^{\ast}_{i+\frac{1}{2}}-\hat{\con}^{-}_{i+\frac{1}{2}}\right)$ and $\psi^{+}(\xi) = \con^{\ast}_{i+\frac{1}{2}}+\xi\left(\hat{\con}^{+}_{i+\frac{1}{2}}-\con^{\ast}_{i+\frac{1}{2}}\right)$, we define,
\begin{align*}
\tilde{\bC}\left(\hat{\con}^{-}_{i+\frac{1}{2}},\con^{\ast}_{i+\frac{1}{2}}\right)=\int_{0}^{1}\bC(\psi^{-}(\xi))d{\xi}, \qquad \text{ and }\qquad 
\tilde{\bC}\left(\con^{\ast}_{i+\frac{1}{2}},\hat{\con}^{+}_{i+\frac{1}{2}}\right)=\int_{0}^{1}\bC(\psi^{+}(\xi))d{\xi}.
\end{align*}
We use numerical integration with four-point Gauss-Lobatto quadrature to evaluate the above integrals. The resulting equation is nonlinear with respect to unknown $\con^{\ast}_{i+\frac{1}{2}}$ and requires a fixed point-based iterative process. The initial guess for the first iteration can be determined by integrating along a straight line from $\hat{\con}^{-}_{i+\frac{1}{2}}$ to $\hat{\con}^{+}_{i+\frac{1}{2}}$, i.e., the first iteration is given by,
\begin{align*}
\con^{\ast}_{i+\frac{1}{2}}=\frac{\left(S^{R}_{i+1/2}\hat{\con}^{+}_{i+\frac{1}{2}}-S^{L}_{i+1/2}\hat{\con}^{-}_{i+\frac{1}{2}}\right)-\left(\bF_{C}\left(\hat{\con}^{+}_{i+\frac{1}{2}}\right)-\bF_{C}\left(\hat{\con}^{-}_{i+\frac{1}{2}}\right)\right)}{S^{R}_{i+1/2}-S^{L}_{i+1/2}}&\\
-\frac{\tilde{\bC}\left(\hat{\con}^{-}_{i+\frac{1}{2}},\hat{\con}^{+}_{i+\frac{1}{2}}\right)\left(\hat{\con}^{+}_{i+\frac{1}{2}}-\hat{\con}^{-}_{i+\frac{1}{2}}\right)}{S^{R}_{i+1/2}-S^{L}_{i+1/2}},&
\end{align*}
where, we define the matrices $\tilde{\bC}\left(\hat{\con}^{-}_{i+\frac{1}{2}},\hat{\con}^{+}_{i+\frac{1}{2}}\right)$ using the linear path $\psi(\xi) = \hat{\con}^{-}_{i+\frac{1}{2}}+\xi\left(\hat{\con}^{+}_{i+\frac{1}{2}}-\hat{\con}^{-}_{i+\frac{1}{2}}\right)$ as given below,
\begin{align*}
\tilde{\bC}\left(\hat{\con}^{-}_{i+\frac{1}{2}},\hat{\con}^{+}_{i+\frac{1}{2}}\right)&=\int_{0}^{1}\bC(\psi(\xi))d{\xi},
\end{align*}
with a four-point Gauss-Lobatto quadrature for the integration. We perform five iterations using fixed-point iteration to find $\con^{\ast}_{i+\frac{1}{2}}$. Then the fluctuations,  $\textbf{D}_{hll}^{\ast-}\left(\hat{\con}^{-}_{i+\frac{1}{2}},\hat{\con}^{+}_{i+\frac{1}{2}}\right)$ and $\textbf{D}_{hll}^{\ast+}\left(\hat{\con}^{-}_{i+\frac{1}{2}},\hat{\con}^{+}_{i+\frac{1}{2}}\right)$, are defined as follows:

\[\textbf{D}_{hll}^{\ast-}\left(\hat{\con}^{-}_{i+\frac{1}{2}},\hat{\con}^{+}_{i+\frac{1}{2}}\right) = \begin{cases}
0, & \textrm{if } S^{L}_{i+1/2}\geq 0\\
S^{L}_{i+1/2}\left(\con^{\ast}_{i+\frac{1}{2}}-\hat{\con}^{-}_{i+\frac{1}{2}}\right) + S^{R}_{i+1/2}\left(\hat{\con}^{+}_{i+\frac{1}{2}}-\con^{\ast}_{i+\frac{1}{2}}\right), & \textrm{if } S^{R}_{i+1/2}\leq 0\\
S^{L}_{i+1/2}\left(\con^{\ast}_{i+\frac{1}{2}}-\hat{\con}^{-}_{i+\frac{1}{2}}\right), &\textrm{otherwise}
\end{cases}\]
\[\textbf{D}_{hll}^{\ast+}\left(\hat{\con}^{-}_{i+\frac{1}{2}},\hat{\con}^{+}_{i+\frac{1}{2}}\right) = \begin{cases}
S^{L}_{i+1/2}\left(\con^{\ast}_{i+\frac{1}{2}}-\hat{\con}^{-}_{i+\frac{1}{2}}\right) + S^{R}_{i+1/2}\left(\hat{\con}^{+}_{i+\frac{1}{2}}-\con^{\ast}_{i+\frac{1}{2}}\right), & \textrm{if } S^{L}_{i+1/2}\geq 0\\
0, & \textrm{if } S^{R}_{i+1/2}\leq 0\\
S^{R}_{i+1/2}\left(\hat{\con}^{+}_{i+\frac{1}{2}}-\con^{\ast}_{i+\frac{1}{2}}\right) , &\textrm{otherwise}.
\end{cases}\]

\rev{\subsection{HLLI Riemann Solver for CGL System}
\label{subsec:HLLI}
Following \cite{dumbser2016new}, we present the HLLI Riemann solver for the CGL system. Let $\con^{\ast}_{i+\frac{1}{2}}$ denote the intermediate state computed using the HLL solver.   Let the diagonal matrix of eigenvalues at the state $\con^{\ast}_{i+\frac{1}{2}}$ is denoted by $\mathbf{\Lambda}\left(\con^{\ast}_{i+\frac{1}{2}}\right)$. Let us also denote the positive and negative parts of the matrix $\mathbf{\Lambda}\left(\con^{\ast}_{i+\frac{1}{2}}\right)$ as $\mathbf{\Lambda}^{-}_\ast\left(\con^{\ast}_{i+\frac{1}{2}}\right)$ and $\mathbf{\Lambda}^{+}_\ast\left(\con^{\ast}_{i+\frac{1}{2}}\right)$, respectively. They are obtained by zeroing out all of its negative and positive eigenvalues. Also, let the matrices of right and left eigenvectors be denoted by $\mathbf{R}\left(\con^{\ast}_{i+\frac{1}{2}}\right)$ and $\mathbf{L}\left(\con^{\ast}_{i+\frac{1}{2}}\right)$, respectively, such that 
$$
\mathbf{R}\left(\con^{\ast}_{i+\frac{1}{2}}\right) \mathbf{L}\left(\con^{\ast}_{i+\frac{1}{2}}\right)=\mathbf{L}\left(\con^{\ast}_{i+\frac{1}{2}}\right) \mathbf{R}\left(\con^{\ast}_{i+\frac{1}{2}}\right)=\mathbf{I}.
$$
where, $\textbf{I}$ is the identity matrix. Let us define the matrix,
\begin{equation}\label{anti_diffusion}
	\mathbf{\delta}_\ast\left(\con^{\ast}_{i+\frac{1}{2}}\right) = \textbf{I}+\frac{\mathbf{\Lambda}^{-}_\ast\left(\con^{\ast}_{i+\frac{1}{2}}\right)}{S^{L}_{i+1/2}}-\frac{\mathbf{\Lambda}^{+}_\ast\left(\con^{\ast}_{i+\frac{1}{2}}\right)}{S^{R}_{i+1/2}}.
\end{equation}
Following \cite{dumbser2016new}, we  introduce the anti-diffusion term,
\begin{equation}\label{anti_diffusion_flux}
	\Phi^{ad}_{i+\frac{1}{2}} = -\Psi_{i+\frac{1}{2}}\frac{S^{L}_{i+1/2} S^{R}_{i+1/2}}{S^{R}_{i+1/2}-S^{L}_{i+1/2}}\mathbf{R}\left(\con^{\ast}_{i+\frac{1}{2}}\right) \delta_\ast\left(\con^{\ast}_{i+\frac{1}{2}}\right) \left(\mathbf{L}\left(\con^{\ast}_{i+\frac{1}{2}}\right)\cdot\left(\hat{\con}^{+}_{i+\frac{1}{2}}-\hat{\con}^{-}_{i+\frac{1}{2}}\right)\right)
\end{equation}
where, $\Psi_{i+\frac{1}{2}}$ is a shock detector, which is $1$ when the solution is smooth and $0$ when shocks are present (see ~\cite{balsara2012self}).  Then, the fluctuations described in section~\ref{subsec:HLL} are modified to get the fluctuations for the HLLI solver, 
\[\textbf{D}_{hlli}^{\ast-}\left(\hat{\con}^{-}_{i+\frac{1}{2}},\hat{\con}^{+}_{i+\frac{1}{2}}\right) = \begin{cases}
	0, & \textrm{if } S^{L}_{i+1/2}\geq 0\\
	S^{L}_{i+1/2}\left(\con^{\ast}_{i+\frac{1}{2}}-\hat{\con}^{-}_{i+\frac{1}{2}}\right) + S^{R}_{i+1/2}\left(\hat{\con}^{+}_{i+\frac{1}{2}}-\con^{\ast}_{i+\frac{1}{2}}\right), & \textrm{if } S^{R}_{i+1/2}\leq 0\\
	S^{L}_{i+1/2}\left(\con^{\ast}_{i+\frac{1}{2}}-\hat{\con}^{-}_{i+\frac{1}{2}}\right) + \Phi^{ad}_{i+\frac{1}{2}}, &\textrm{otherwise}
\end{cases}\]
\[\textbf{D}_{hlli}^{\ast+}\left(\hat{\con}^{-}_{i+\frac{1}{2}},\hat{\con}^{+}_{i+\frac{1}{2}}\right) = \begin{cases}
	S^{L}_{i+1/2}\left(\con^{\ast}_{i+\frac{1}{2}}-\hat{\con}^{-}_{i+\frac{1}{2}}\right) + S^{R}_{i+1/2}\left(\hat{\con}^{+}_{i+\frac{1}{2}}-\con^{\ast}_{i+\frac{1}{2}}\right), & \textrm{if } S^{L}_{i+1/2}\geq 0\\
	0, & \textrm{if } S^{R}_{i+1/2}\leq 0\\
	S^{R}_{i+1/2}\left(\hat{\con}^{+}_{i+\frac{1}{2}}-\con^{\ast}_{i+\frac{1}{2}}\right) -\Phi^{ad}_{i+\frac{1}{2}} , &\textrm{otherwise}.
\end{cases}\]}

\subsection{Implementation}
We now describe the complete algorithm, which is implemented in the numerical code and given by the schemes  \eqref{eq:AFD_NC}.\\

\begin{algorithmic}[1] 
	\State At each zone $x_i$, compute the primitive variables and use them to build the physical flux $\bF(\con_i)$ and the matrix of non-conservative product $\bC(\con_i)$.
	\vspace{0.5cm}
	\State With the help of the primitive and conservative variables, compute the right eigenvector matrices $\mathbf{R}(\mathbf{U}_i)$ at each zone $x_i$. Evaluate the corresponding left eigenvectors matrices $\mathbf{L}(\mathbf{U}_i)$ such that we have the the following property:
	$$\mathbf{L}_{i}\mathbf{R}_i=\mathbf{L}_{i}\mathbf{R}_i=\mathbf{I}, \qquad \text{ where } \mathbf{R}_i=  \mathbf{R}(\mathbf{U}_i) \text{ and } \mathbf{L}_i =  \mathbf{L}(\mathbf{U}_i)$$
	\State For each zone $x_i$, project the neighbouring stencil variables into the characteristic space of zone $x_i$, i.e.,
	\[{\bf{\mathcal{{U}}}}_i^j =\mathbf{L}_{i}\con_j\] 
	for all neighbouring cells given by $j\in \mathcal{S}_i; \ \mathcal{S}_i:=\{i-s,\cdots,i-1,i,i+1,\cdots,i+s\}$ associated with $x_{i}$. For the third-order scheme, $s=1$, and for the fifth-order scheme, we choose $s=2$, and for the seventh-order scheme, we take $s=3$.
	\vspace{0.5cm}
	\State For each zone $x_i$, employ the WENO-AO algorithm on ${\bf{\mathcal{{U}}}}_i^j;$ $j\in \mathcal{S}_i$ from Section III of~\cite{balsara2024efficient} which provides the closed form formulae required for WENO interpolation in one dimension. This involves creating a non-linear hybridization of large high-order and small lower-order accurate stencils. Thus, after this step, we obtain an interpolated polynomial $\mathcal{P}_i(x)$. \vspace{0.5cm}
	\State Observe that the WENO-AO interpolation is conducted in the characteristic space. Using $\mathcal{P}_i(x)$ and the set of right eigenvectors from STEP 2, we obtain the boundary values $\hat{\con}^{+}_{i-\frac{1}{2}}$ and $\hat{\con}^{-}_{i+\frac{1}{2}}$ in the physical space as
	\[
	\hat{\con}^{+}_{i-\frac{1}{2}} = \mathbf{R}_i\mathcal{P}_i(x_{i-\frac{1}{2}})
	\quad
	\text{and}
	\quad
	\hat{\con}^{-}_{i+\frac{1}{2}} = \mathbf{R}_i\mathcal{P}_i(x_{i+\frac{1}{2}})
	\]
	\State Above three steps are costly steps of the algorithm because we are projecting all the stencils of interest into the characteristic space of each zone $i$ using eigenvectors. At the end of these steps, we get WENO interpolation-based $\hat{\con}^{-}_{i+\frac{1}{2}}$ and $\hat{\con}^{+}_{i-\frac{1}{2}}$ at each zone boundary.\vspace{0.5cm}
	\State With boundary values in hand at each interface $x_{i+\frac{1}{2}}$, use two states: $ \hat{\con}^{-}_{i+\frac{1}{2}},\ \hat{\con}^{+}_{i+\frac{1}{2}}$ and follow section~\ref{subsec:HLL},~\ref{subsec:HLLI} and~\cite{balsara2023efficient1} to obtain the slowest $S^{L}_{i+1/2}$ and fastest speeds $S^{R}_{i+1/2}$ of the Riemann fan. \vspace{0.5cm}
	\State Using these boundary states and speeds, use the Riemann solver to obtain the intermediate state $\con^{\ast}_{i+\frac{1}{2}}$ and the left- and right-going fluctuations $\textbf{D}^{\ast-}\left(\hat{\con}^{-}_{i+\frac{1}{2}},\hat{\con}^{+}_{i+\frac{1}{2}}\right)$ and $\textbf{D}^{\ast+}\left(\hat{\con}^{-}_{i+\frac{1}{2}},\hat{\con}^{+}_{i+\frac{1}{2}}\right)$ at each zone boundary $x_{i+\frac{1}{2}}$ using either HLL or HLLI solver. \vspace{0.5cm}
	\State Now, if we want to make a characteristic projection of the higher derivatives of the flux variable, we can do so with $\con^{\ast}_{i+\frac{1}{2}}$. As a result, we determine the right and left eigenvector matrices for the intermediate state $\con^{\ast}_{i+\frac{1}{2}}$ at each zone boundary $x_{i+\frac{1}{2}}$. Following ~\cite{balsara2024efficient,balsara2024efficientnc}, this characteristic projection is not needed, allowing us to apply WENO interpolation directly to flux components.\vspace{0.5cm}
	\State Use the boundary-centered WENO-AO interpolation scheme from Section IV of~\cite{balsara2024efficient} and~\cite{balsara2024efficientnc} to obtain suitably high order derivatives $\left([\p^2_x\hat{\bF}_{C}]_{i+\frac{1}{2}},~[\p^4_x\hat{\bF}_{C}]_{i+\frac{1}{2}},~[\p^6_x\hat{\bF}_{C}]_{i+\frac{1}{2}}\right)$ of the flux variables at each zone boundary $x_{i+\frac{1}{2}}$.\vspace{0.5cm}
	\State Similarly, we can get high order derivatives of $\left([\p^2_x\hat{\con}]_{i+\frac{1}{2}},~[\p^4_x\hat{\con}]_{i+\frac{1}{2}},~[\p^6_x\hat{\con}]_{i+\frac{1}{2}}\right)$ of conservative variable $\con$ at each zone boundary $x_{i+\frac{1}{2}}$.\vspace{0.5cm}
\end{algorithmic}

\section{Numerical results}
\label{sec:num}
For the time discretization, we use a third-order SSP-Runge-Kutta (RK) scheme~\cite{shu1988efficient,spiteri2002new,spiteri2003non}. To compute the {\em isotropic} limit, we consider the stiff source term \eqref{eq:source} with $\tau = 10^{-8}$. Hence, explicit methods will be highly inefficient. So, when stiff source terms are present, we will use the third-order Runge-Kutta IMEX methods presented in~\cite{pareschi2005implicit,kupka2012total}. To reduce the oscillations, we use a flattener, which is described in~\cite{balsara1998total,balsara2024efficient}.

For the accuracy test cases of the fifth and seventh-order schemes, we reduced the timestep size as the mesh was refined so that the temporal error remained dominated by the spatial error. Every mesh doubling necessitates a timestep reduction that follows the formula $\Delta t \rightarrow\Delta t (1/2)^{5/3}$ when a temporally third-order accurate time-stepping method is combined with a spatially fifth-order scheme. In a similar way, if a temporally third-order time-stepping technique is combined with a spatially seventh-order system, then each mesh doubling necessitates a $\Delta t \rightarrow\Delta t (1/2)^{7/3}$ timestep reduction. We use the CFL number $0.3$ for the accuracy test cases. For all other one-dimensional test cases, we use the CFL number to be $0.8$.

We present two test cases to demonstrate the accuracy of the proposed schemes. Then, we present several test cases in one dimension to demonstrate the stability and performance of the proposed schemes. Several of these test cases are generalized from the MHD test cases. To compare the CGL results with MHD results, we compute the MHD reference solutions using a second-order scheme with Rusanov solver and \textit{MinMod} limiter for spatial discretization and second-order SSP-RK method for the time update. We use $10000$ cells to obtain the reference MHD solutions.

\subsection{Accuracy test}
\label{test:Accuracytest}
In this test case, we ignore the source terms. We consider a test case with smooth initial conditions from~\cite{singh2024entropy}. We consider the computational domain of $[0, 1]$ with periodic boundary conditions. The initial density profile is taken to be $\rho(x, 0) = 2 + \sin{(2\pi x)}$. Density is then advected with velocity $\bu = (1, 0, 0)$. The pressure components are taken to be constant with $\per = \pll = 1.0$. The magnetic field is taken to be $\B = (1, 1, 0)$. The exact solution is then given by $\rho(x, t) = 2 + \sin{(2\pi(x - t))}$. The simulations are performed till the final time $t = 2.0$.

\begin{table}[ht]
\centering
\begin{tabular}{ |c|c|c|c|c| } 
	\hline
	\multicolumn{5}{|c|}{\textbf{$L_1$ and $L_\infty$ Errors and Order of Accuracy for $\rho$}}\\
	\hline
	\textbf{Order 3} & $L_1$ Error & $L_1$ Accuracy &  $L_\infty$ Error & $L_\infty$ Accuracy  \\
	\hline
	10 & 3.79E-01 & $--$ & 6.52E-02 & $--$ \\
	\hline
	20	 & 9.76E-02 & 1.955522317 & 1.11E-02 & 2.549737343 \\
	\hline
	40 & 3.03E-02 & 1.685234163	& 2.12E-03 & 2.395857442 \\
	\hline
	80	 & 6.55E-03 & 2.21258628 & 3.73E-04 & 2.502853459 \\
	\hline
	160	& 1.31E-03 & 2.316442269 & 6.37E-05 & 2.551358214 \\
	\hline
	320 & 2.52E-04 & 2.382112478 & 1.07E-05 & 2.579989011\\
	\hline
	\textbf{Order 5} &  &  &  &   \\ 
	\hline
	10 & 9.49E-03 & $--$ & 1.55E-03 & $--$ \\
	\hline
	20	 & 2.94E-04 & 5.013838473 & 2.88E-05 & 5.751974248 \\
	\hline
	40 & 9.04E-06 & 5.021341591	& 4.73E-07 & 5.925146837 \\
	\hline
	80	& 2.81E-07 & 5.007533709 & 7.48E-09 & 5.983669743 \\
	\hline
	160	& 8.79E-09 & 4.999485867 & 1.17E-10 & 5.997987111 \\
	\hline
	\textbf{Order 7} &  &  & & \\ 
	\hline
	10 & 9.15E-04 & $--$ & 1.52E-04 & $--$ \\
	\hline
	20 & 6.11E-06 &	7.227753471 & 6.16E-07 & 7.946465664 \\
	\hline
	40 & 4.66E-08 & 7.035223584	& 2.45E-09 & 7.976163789 \\
	\hline
	80	 & 3.64E-10 & 6.998951661 & 9.54E-12 & 8.002131627 \\
	\hline
\end{tabular}
\caption{\textbf{\nameref{test:Accuracytest}:} $L_1$ and $L_\infty$ errors and order of accuracy of $\rho$ for the 3rd, 5th and 7th order AFD-WENO schemes using HLL Riemann solver.}
\label{tab:1}
\end{table}
\begin{table}[ht]
\centering
\begin{tabular}{ |c|c|c|c|c| } 
	\hline
	\multicolumn{5}{|c|}{\textbf{$L_1$ and $L_\infty$ Errors and Order of Accuracy for $\rho$}}\\
	\hline
	\textbf{Order 3} & $L_1$ Error & $L_1$ Accuracy &  $L_\infty$ Error & $L_\infty$ Accuracy  \\
	\hline
	10 & 3.66E-01 & $--$ & 6.33E-02 & $--$ \\
	\hline
	20	 & 9.52E-02 & 1.943517796 & 1.02E-02 & 2.636645366 \\
	\hline
	40 & 2.77E-02 & 1.781955009 & 1.89E-03 & 2.430993592 \\
	\hline
	80	 & 6.06E-03	& 2.190547617 &	3.31E-04 & 2.512635623 \\
	\hline
	160	& 1.20E-03 & 2.333828379 & 5.65E-05 & 2.550057852 \\
	\hline
	320 & 2.32E-04 & 2.372723485 & 9.48E-06 & 2.5753723 \\
	\hline
	\textbf{Order 5} &  &  &  &   \\ 
	\hline
	10 & 9.04E-03 & $--$ & 1.40E-03 & $--$ \\
	\hline
	20	 &  2.85E-04 & 4.986276706 & 2.24E-05 &	5.959755971 \\
	\hline
	40 &  8.95E-06 & 4.994007241 & 3.52E-07 & 5.995243159 \\
	\hline
	80	&  2.80E-07 & 4.997824175 &	5.50E-09 & 5.998486558 \\
	\hline
	160	&  8.76E-09 & 4.999123028 &	8.60E-11 & 5.999344056 \\
	\hline
	\textbf{Order 7} &  &  & & \\ 
	\hline
	10 & 8.69E-04 & $--$ & 1.32E-04 & $--$ \\
	\hline
	20 & 5.98E-06 &	7.183576421 & 4.70E-07 & 8.13878327 \\
	\hline
	40 & 4.61E-08 & 7.018615212	& 1.81E-09 & 8.020269672 \\
	\hline
	80	 & 3.61E-10 & 6.994991439 & 7.09E-12 & 7.996296319 \\
	\hline
\end{tabular}
\caption{\textbf{\nameref{test:Accuracytest}:} $L_1$ and $L_\infty$ errors and order of accuracy of $\rho$ for the 3rd, 5th and 7th order AFD-WENO schemes using HLLI Riemann solver.}
\label{tab:1.1}
\end{table}
\rev{In Tables \eqref{tab:1} and \eqref{tab:1.1}, we present the  $L_1$-errors and $L_\infty$-errors of the density profile for the 3rd, 5th, and 7th order AFD-WENO schemes using HLL and HLLI Riemann solvers, respectively.}  We observe that in both norms, the 5th and 7th-order schemes achieve the desired order of accuracy. The 3rd order scheme, however, has a slightly lower order of accuracy. This is consistent with the results in \cite{balsara2024efficientnc,balsara2024efficient}. Furthermore, we also note that the HLLI solver has lower errors than the HLL solver.

\subsection{Circularly polarized Alfv\'en waves}
\label{test:Circularlypolarized}
In this test case, following \cite{Hirabayashi2016new}, we extend the circular polarized Alfv\'en propagation along the magnetic field lines for the isotropic MHD equation to the CGL equations. For the MHD test case, the exact solution is given by,
\begin{align} 
	u_y &= \delta u \sin{(kx - \omega t)},~~~~~~~~~~~~~u_z = \delta u \cos{(kx - \omega t)}&\\
	B_y &= \sqrt{4\pi}\times\delta B \sin{(kx - \omega t)},~~B_z = \sqrt{4\pi}\times\delta B \cos{(kx - \omega t)}
\end{align} 
for constant density, pressure, normal velocity, and a normal magnetic field. The Walen relation connects the amplitudes of velocity and magnetic field as follows:
\begin{equation}\label{31}
	\frac{\delta u}{V_A} = \frac{\delta B}{B_x}
\end{equation}
In the case of the CGL model, the pressure anisotropy $\pll-\per$ affects the magnetic tension force by a factor of $\epsilon = 1 - \frac{\pll-\per}{|\B|^2}$ compared to the MHD case. Again, we do not consider the source terms. Significantly, this modification also modifies the standard Alfv\'en velocity $V_A$ in the Walden relation with a factor of $\sqrt{\epsilon}$, i.e., $V^\ast_A = \sqrt{\epsilon} V_A$, which need to take into consideration for the CGL system. 
The pressures and density were held constant with $\rho = \per = \pll = 1.0$. Hence $\epsilon=1.0$. Other parameter are taken to be as follows: $u_x = 0$, $B_x = \sqrt{4\pi}$, $V_A =u_x + \sqrt{{\frac{B_x^2}{4 \pi\rho}-\frac{(P_\parallel-P_\perp)b_x^2}{\rho}}}$ and $\delta B=0.1$. We also take  $L_x=1$ and $k =\frac{1}{2\pi}$. The simulations are carried out till final time of $t = \frac {5 L_x}{V^\ast_A}$, on the computational domain of $[0, 1]$ with periodic boundary conditions.

\begin{table}
\centering
\begin{tabular}{ |c|c|c|c|c| }
	\hline
	\multicolumn{5}{|c|}{\textbf{$L_1$ and $L_\infty$ Errors and Order of Accuracy for $B_y$}}\\
	\hline
	\textbf{Order 3} & $L_1$ Error & $L_1$ Accuracy &  $L_\infty$ Error & $L_\infty$ Accuracy  \\
	\hline
	10 & 1.60E-01 & $--$ & 2.47E-02 & $--$ \\
	\hline
	20	 & 4.24E-02 & 1.91337666 & 3.31E-03 & 2.896561328 \\
	\hline
	40 & 9.38E-03 & 2.175908167	& 3.67E-04 & 3.174751321 \\
	\hline
	80	 & 2.22E-03 & 2.076966688 & 4.36E-05 & 3.07239535 \\
	\hline
	160	& 5.47E-04 & 2.022350079 & 5.37E-06 & 3.021399766 \\
	\hline
	320 & 1.36E-04 & 2.004415387 & 6.69E-07 & 3.004181062\\
	\hline
	\textbf{Order 5} &  &  &   &  \\
	\hline
	10 & 1.36E-02 & $--$ & 2.10E-03 & $--$ \\
	\hline
	20	 & 4.32E-04 & 4.974156817 & 3.39E-05 & 5.949052993 \\
	\hline
	40 & 1.36E-05 & 4.983639858	& 5.36E-07 & 5.984577939 \\
	\hline
	80	 & 4.28E-07 & 4.995082389 & 8.40E-09 & 5.995348941 \\
	\hline
	160	& 1.34E-08 & 4.998296981 & 1.31E-10 & 5.998362307 \\
	\hline
	\textbf{Order 7} &  &  & & \\
	\hline
	10 & 1.47E-03 & $--$ & 2.27E-04 & $--$ \\
	\hline
	20	 & 9.14E-06 & 7.326078469 & 7.19E-07 & 8.300300644 \\
	\hline
	40 & 7.03E-08 & 7.022848348	& 2.76E-09 & 8.024072751 \\
	\hline
	80	 & 5.52E-10 & 6.992343352 & 1.08E-11 & 7.993190315 \\
	\hline
\end{tabular}
\caption{\textbf{\nameref{test:Circularlypolarized}:} $L_1$ and $L_\infty$ errors and order of accuracy of $B_y$ for the 3rd, 5th and 7th order AFD-WENO schemes using HLL Riemann solver..}
\label{tab:2}
\end{table}

\begin{table}
\centering
\begin{tabular}{ |c|c|c|c|c| }
	\hline
	\multicolumn{5}{|c|}{\textbf{$L_1$ and $L_\infty$ Errors and Order of Accuracy for $B_y$}}\\
	\hline
	\textbf{Order 3} & $L_1$ Error & $L_1$ Accuracy &  $L_\infty$ Error & $L_\infty$ Accuracy  \\
	\hline
	10 & 1.18E-01 & $--$ & 1.82E-02 & $--$ \\
	\hline
	20	 &  3.39E-02 & 1.794015198 & 2.63E-03 & 2.79078267 \\
	\hline
	40 &  8.67E-03 & 1.968794666 & 3.39E-04 & 2.952393485 \\
	\hline
	80	 &  2.17E-03 & 1.995325632 & 4.27E-05 & 2.991864559 \\
	\hline
	160	&  5.44E-04 & 1.999136678 & 5.34E-06 & 2.998561641 \\
	\hline
	320 &  1.36E-04 & 1.998375754 & 6.68E-07 & 2.998265677 \\
	\hline
	\textbf{Order 5} &  &  &   &  \\
	\hline
	10 & 9.57E-03 & $--$ & 1.48E-03 & $--$ \\
	\hline
	20	 &  2.66E-04 & 5.166707771 & 2.07E-05 & 6.160904491 \\
	\hline
	40 &  8.04E-06 & 5.050353157 & 3.14E-07 & 6.037565543 \\
	\hline
	80	 &  2.49E-07 & 5.012555797 & 4.88E-09 & 6.008684932 \\
	\hline
	160	&  7.77E-09	& 5.002372481 & 7.62E-11 & 6.001378352 \\
	\hline
	\textbf{Order 7} &  &  & & \\
	\hline
	10 & 1.07E-03 & $--$ & 1.65E-04 & $--$ \\
	\hline
	20	 & 5.78E-06 & 7.52708685 & 4.49E-07 & 8.519879058 \\
	\hline
	40 & 4.16E-08 & 7.117318317 & 1.63E-09 & 8.106495695 \\
	\hline
	80	& 3.22E-10 & 7.014552697 & 6.32E-12 & 8.008004036 \\
	\hline
\end{tabular}
\caption{\textbf{\nameref{test:Circularlypolarized}:} $L_1$ and $L_\infty$ errors and order of accuracy of $B_y$ for the 3rd, 5th and 7th order AFD-WENO schemes using HLLI Riemann solver.}
\label{tab:2.1}
\end{table}
\rev{In Tables \eqref{tab:2} and  \eqref{tab:2.1}, we have presented the $L_1$-errors and $L_\infty$-errors of $B_y$ for the 3rd, 5th, and 7th order AFD-WENO schemes using HLL and HLLI Riemann solvers, respectively.} We observe that both 5th, and 7th order schemes have the desired accuracy in $L_1$ norm and one higher order in $L_\infty$ norm. For the 3rd order scheme, we note that $L_1$-accuracy is one order less than the expected third order. Similar to the earlier case, we also observe that the HLLI solver is more accurate than the HLL solver.

\subsection{One-Dimensional reconnection layer}
\label{test:reconnectionlayer}
For this test case, we consider the one-dimensional Riemann problem described in \cite{Hirabayashi2016new,Hirabayashi2013Magnetic}, which is used to examine the various waves in a self-similarly evolving reconnection layer. We do not consider the source terms. It involves the behavior of slow-mode and Alfven waves in the CGL approximation. Unlike the previous test case, this test is a non-coplanar problem. The computational domain is  $[-200 L, 200 L]$ with free boundary conditions. The initial setup involves an isotropic and isothermal Harris-type current sheet featuring a uniform guide field,
\begin{align}
	B_{y}(x) &= B_{0}\sqrt{4\pi}\cos{(\phi)}\tanh{\left(\frac{x}{L}\right)}\\
	B_{z}(x) &= B_{0}\sqrt{4\pi}\sin{(\phi)}
\end{align}
Here, the magnetic field strength in the lobe region, denoted as $B_0$, is a critical parameter. Additionally, $\phi$ represents the angle between the lobe magnetic field and the $x$-axis, while $L$ corresponds to the half-width of the current sheet. To achieve an initial isotropic pressure balance, adjustments are made to set the plasma beta $(\beta = \frac{2p}{|\B|^2})$, as measured in the lobe region, to a value of $0.25$. The normal magnetic field component $B_x$ is assumed to be  $5 \%$ of $B_0$. This leads to the evolution of several waves moving away from the current sheets and towards both lobes.

Upon introducing the normal magnetic field component, $B_x$, various wave phenomena are observed. Specifically, fast rarefaction waves, rotational discontinuities, and slow shocks propagate away from the current sheet toward both lobes. The magnitude of $B_x$ is set at $\sqrt{4\pi}\times(5\%~of~B_0)$. To normalize, we take $L$, $B_0$, and the lobe density $\rho_0$ to be unity. This leads to the velocity and pressure being scaled by $V_A = \frac{B_0}{\sqrt{\rho_0}}$ and $B_0^2$, respectively.

We take velocity $u_x = 0$, $u_y = 0$ and $u_z = 0$ and pressure $\pll = \per = 0.125$. We set angle $\phi=30^{\circ}$. We use $2000$ cells and the final time of $t=3500$.
\begin{figure}[!htbp]
\begin{center}
	\subfigure[$\rho$]{\includegraphics[width=2.0in, height=1.5in]{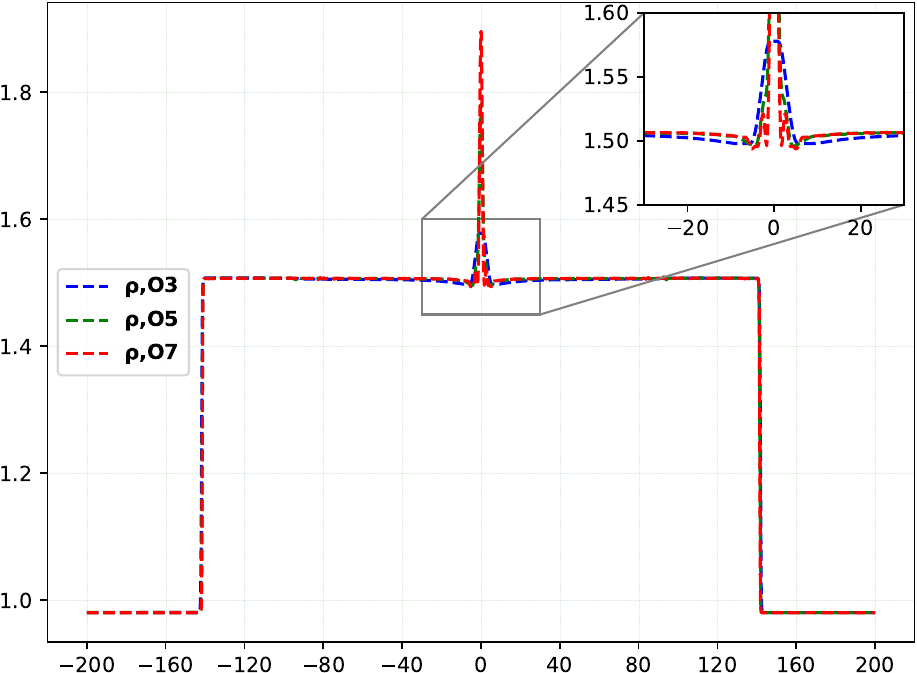}}
	\subfigure[$u_x~\&~u_y~\&~u_z$]{\includegraphics[width=2.0in, height=1.5in]{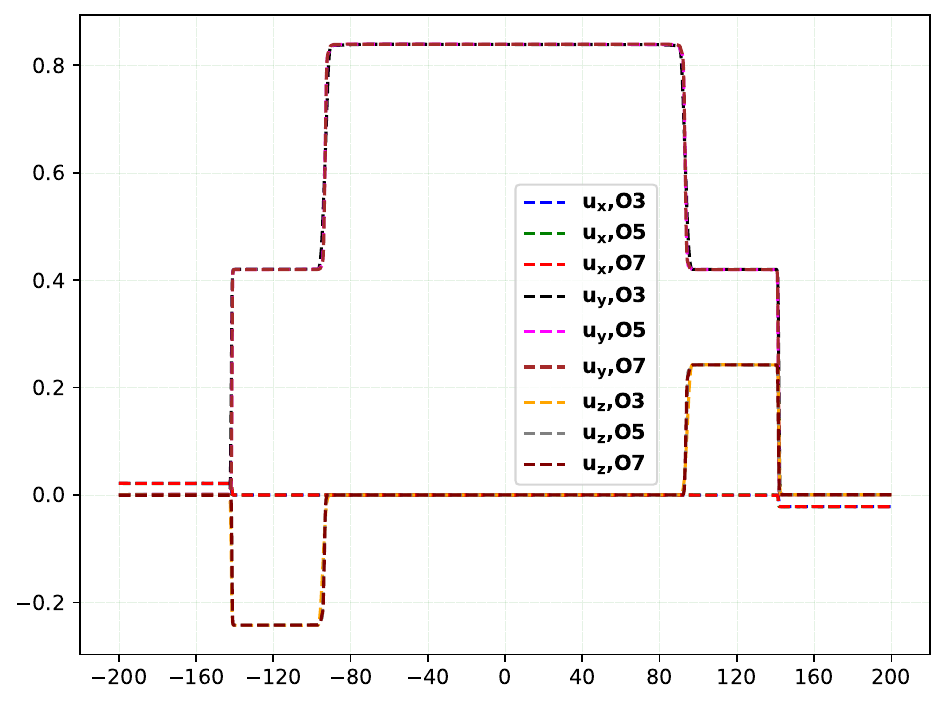}}
	\subfigure[$\pll$]{\includegraphics[width=2.0in, height=1.5in]{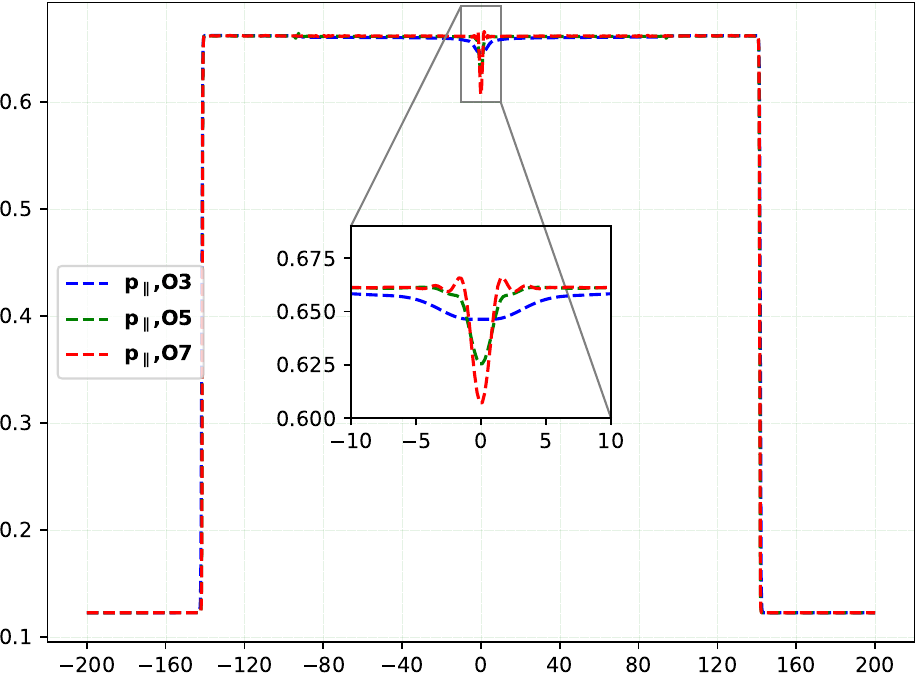}}
	\subfigure[$\per$ ]{\includegraphics[width=2.0in, height=1.5in]{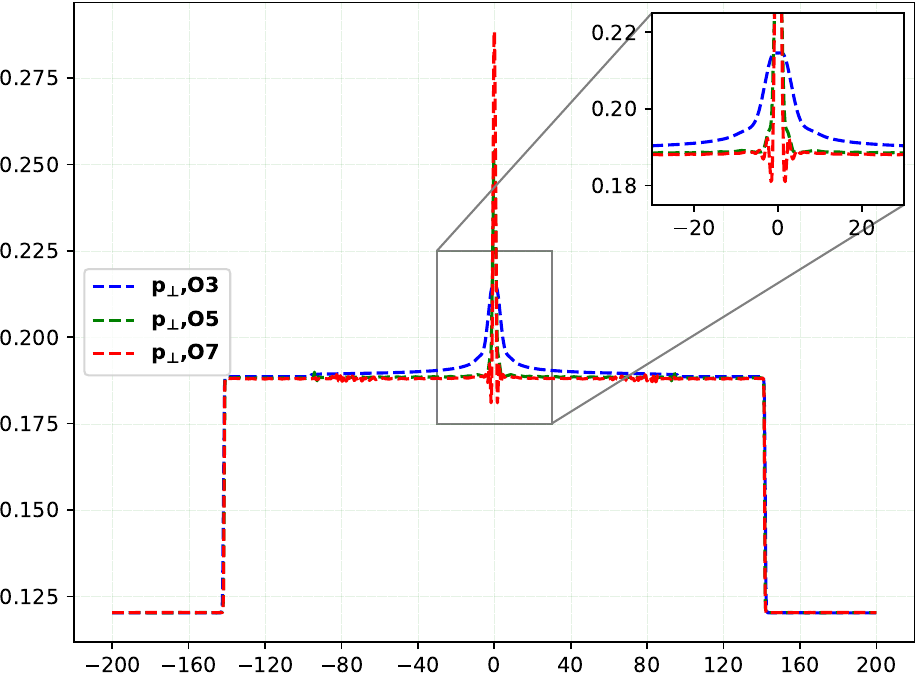}}
	\subfigure[$B_y~\&~B_z$ ]{\includegraphics[width=2.0in, height=1.5in]{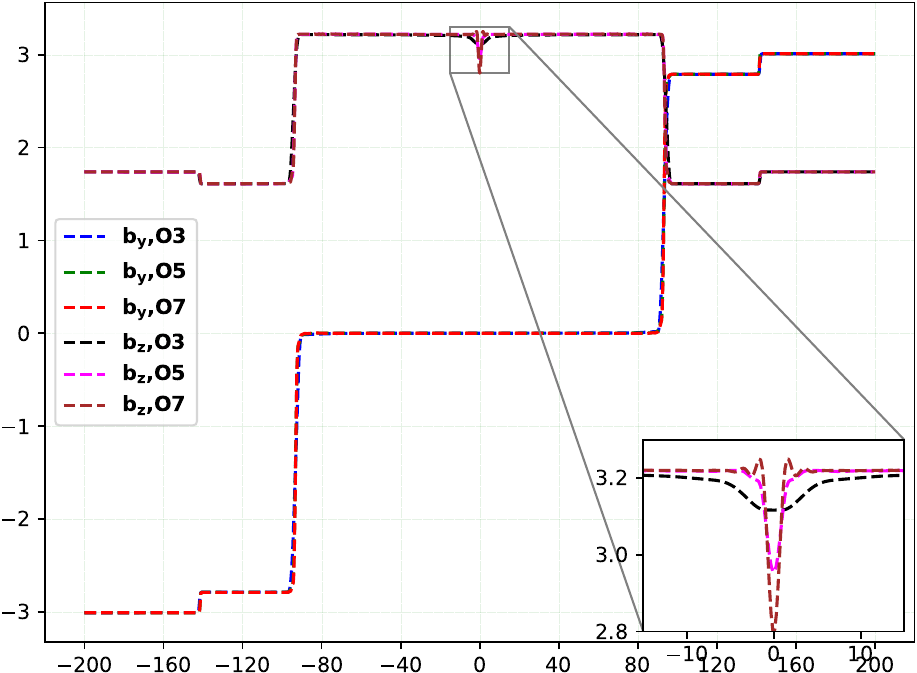}}
	\caption{\textbf{\nameref{test:reconnectionlayer}:} Plots of density, velocity in $x,~y$ and $z$ direction, parallel and perpendicular pressure components and magnetic field in $y$ and $z$ direction for 3rd, 5th and 7th order numerical schemes using HLL Riemann solver and 2000 cells at final time t = 3500.}
	\label{fig:reconnection_layer}
\end{center}
\end{figure}
\begin{figure}[!htbp]
\begin{center}
	\subfigure[$\rho$]{\includegraphics[width=2.0in, height=1.5in]{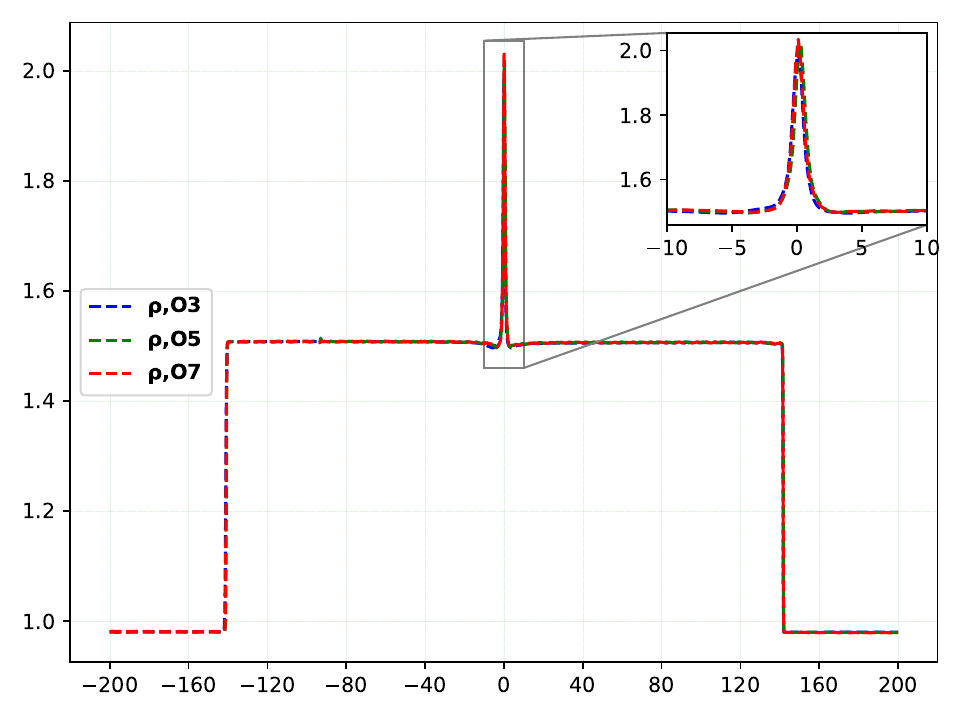}}
	\subfigure[$u_x~\&~u_y~\&~u_z$]{\includegraphics[width=2.0in, height=1.5in]{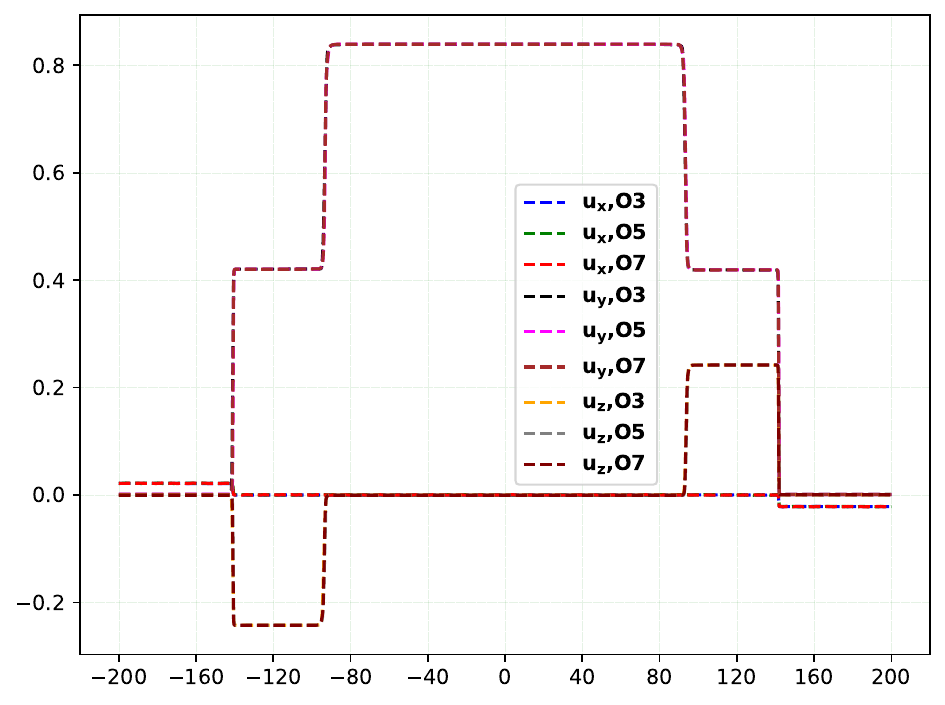}}
	\subfigure[$\pll$]{\includegraphics[width=2.0in, height=1.5in]{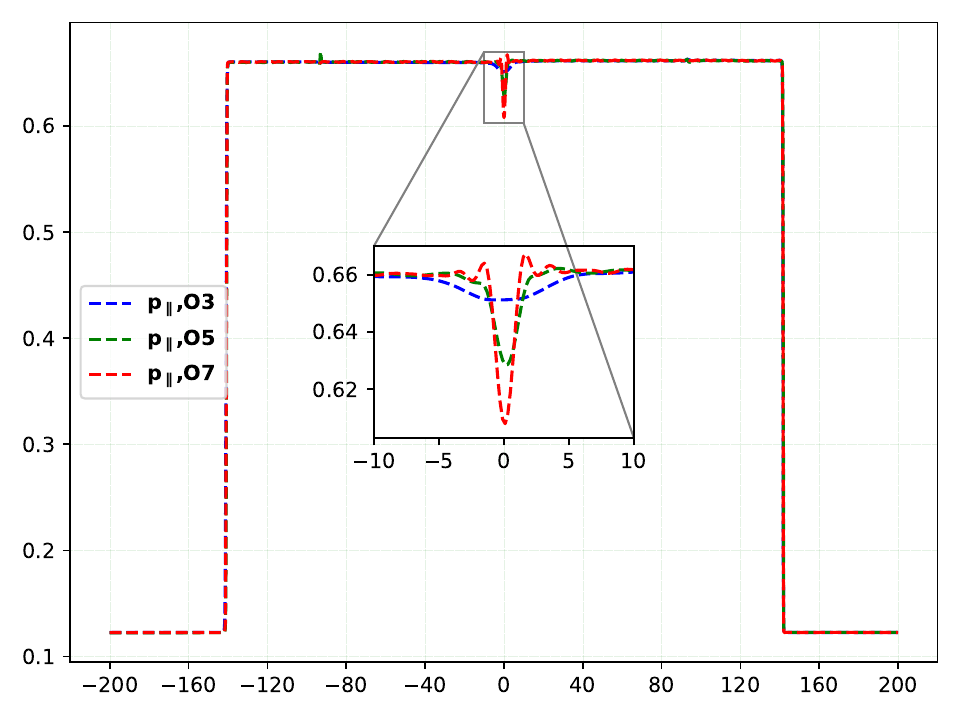}}
	\subfigure[$\per$ ]{\includegraphics[width=2.0in, height=1.5in]{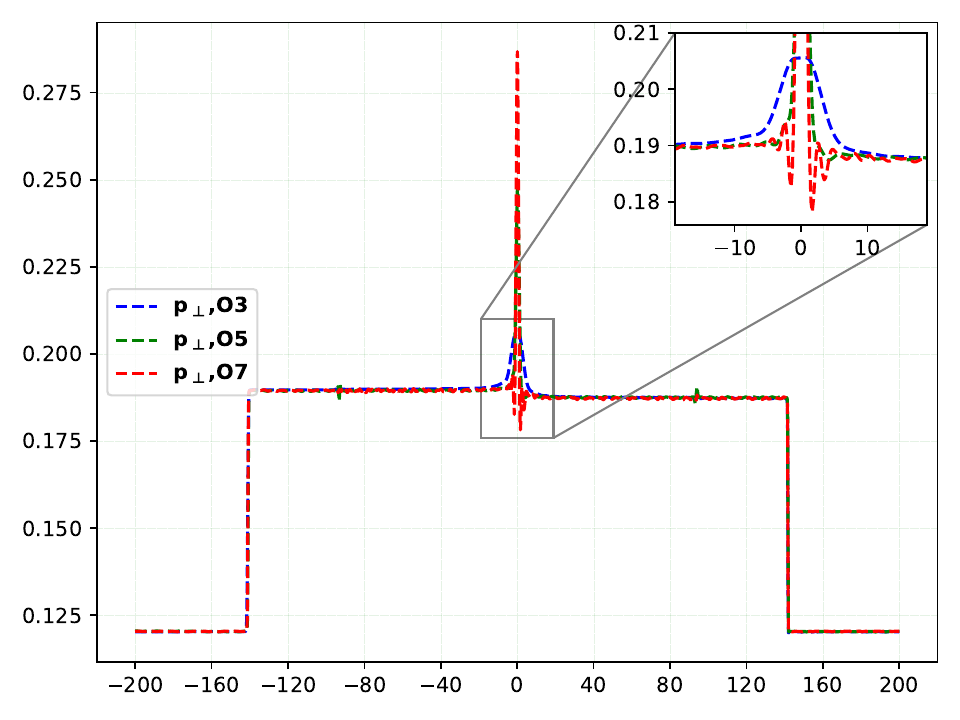}}
	\subfigure[$B_y~\&~B_z$ ]{\includegraphics[width=2.0in, height=1.5in]{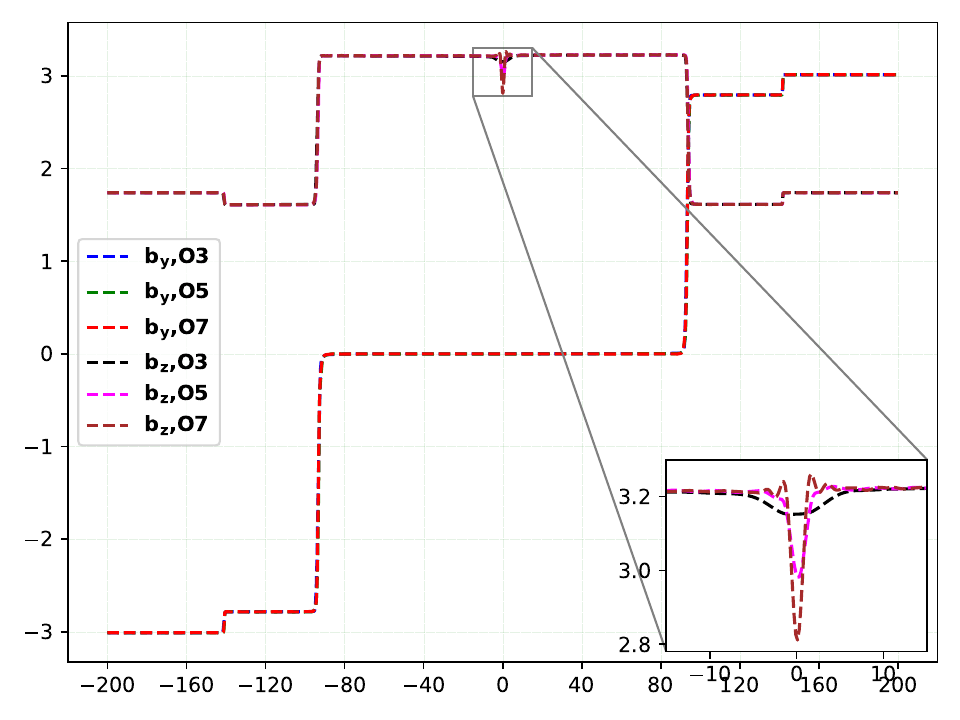}}
	\caption{\textbf{\nameref{test:reconnectionlayer}:} Plots of density, velocity in $x,~y$ and $z$ direction, parallel and perpendicular pressure components and magnetic field in $y$ and $z$ direction for 3rd, 5th and 7th order numerical schemes using HLLI Riemann solver and 2000 cells at final time t = 3500.}
	\label{fig:reconnection_layer_1}
\end{center}
\end{figure}
\rev{The numerical results using the HLL Riemann solver are presented in Fig.$\eqref{fig:reconnection_layer}$. The results for the HLLI Riemann solver are presented in Fig.$\eqref{fig:reconnection_layer_1}$}. We note that all the waves are resolved accurately for all the schemes using both solvers. We do observe some small-scale oscillations in the 7th-order scheme, but the 3rd and 5th-order schemes do not have any oscillations. We also note that for density, the HLLI solver is more accurate than the HLL solver. Furthermore, the results of the proposed schemes compare favorably when compared with those presented in ~\cite{Hirabayashi2016new}. 

\subsection{Riemann problem 1: Brio and Wu Shock tube problem}
\label{test:rp1:brio}
In this test case, we consider a CGL generalization of the Brio-Wu shock tube problem (\cite{Brio1988upwind}) for MHD equations. This test case is also considered in \cite{Hirabayashi2016new,singh2024entropy}. The computational domain is taken to be $[-1,1]$ with outflow boundary conditions. The initial discontinuity is at $x=0$, separating the left and right states, which are given by,
\[
(\rho, u_{x}, u_{y}, u_{z}, \pll, \per, B_{y}, B_{z}) 
= \begin{cases}
(1, 0, 0, 0, 1, 1, \sqrt{4\pi}, 0), & \textrm{if } x\leq 0\\
(0.125, 0, 0, 0, 0.1, 0.1, -\sqrt{4\pi}, 0), & \textrm{otherwise}
\end{cases}
\]

We also set $B_{x}=0.75\sqrt{4\pi}$. In addition to the CGL solution, to capture the isotropic limit, we also use the source terms and compare the results with the MHD reference solution. The simulations are performed using $800$ cells, and the final simulation time is $t=0.2$.

\begin{figure}[!htbp]
\begin{center}
	\subfigure[$\rho$ (without source term)]{\includegraphics[width=2.0in, height=1.5in]{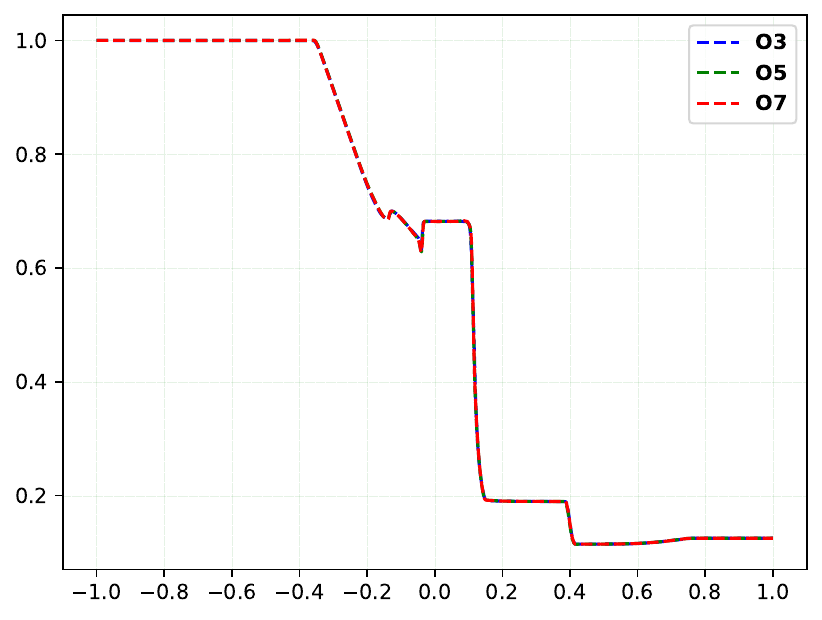}}
    \subfigure[$\rho$ (with source term)]{\includegraphics[width=2.0in, height=1.5in]{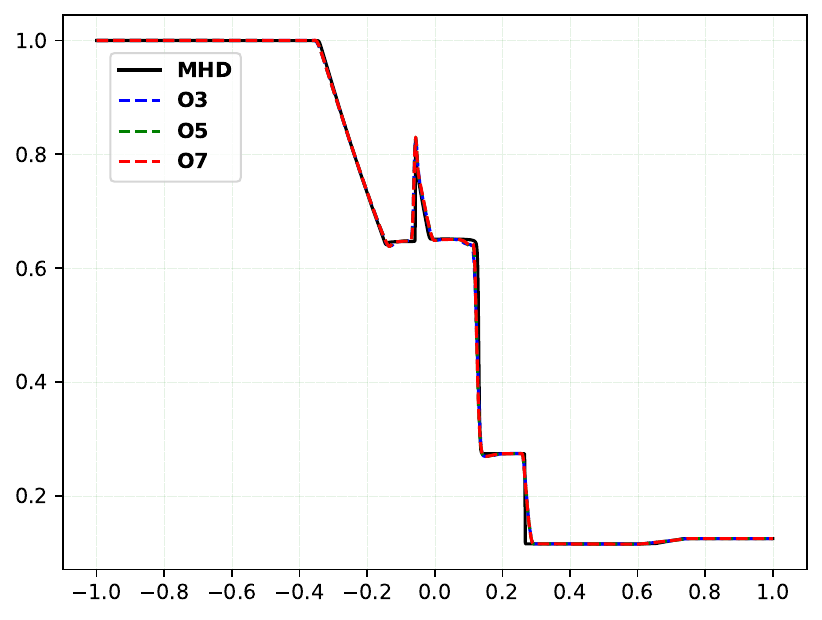}}\\
	\subfigure[$u_x$ (without source term)]{\includegraphics[width=2.0in, height=1.5in]{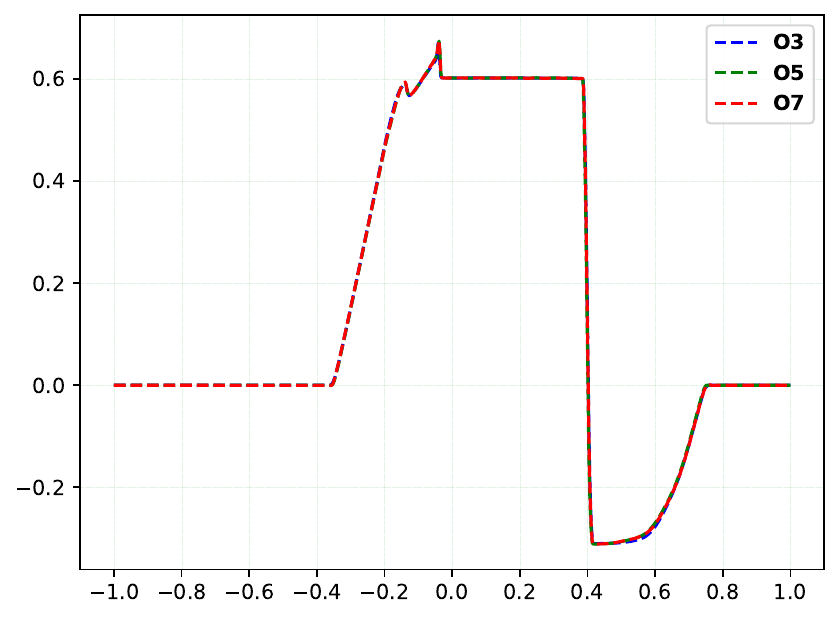}}
    \subfigure[$u_x$ (with source term)]{\includegraphics[width=2.0in, height=1.5in]{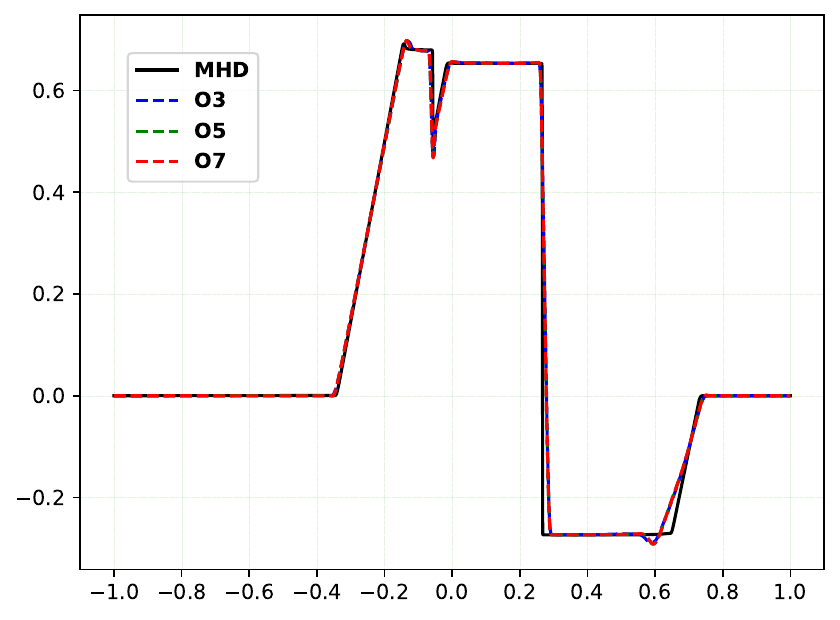}}\\
	\subfigure[$\pll$ (without source term)]{\includegraphics[width=2.0in, height=1.5in]{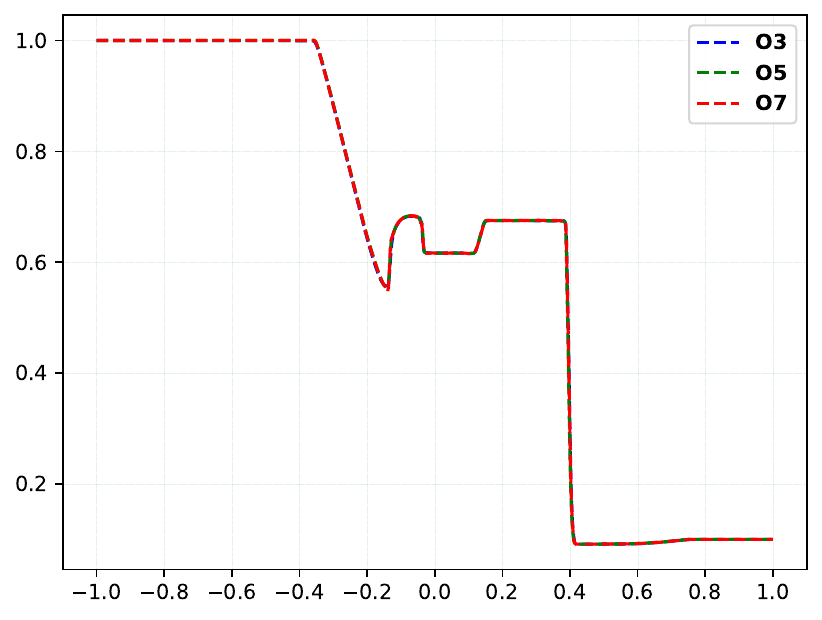}}
    \subfigure[$\pll$ (with source term)]{\includegraphics[width=2.0in, height=1.5in]{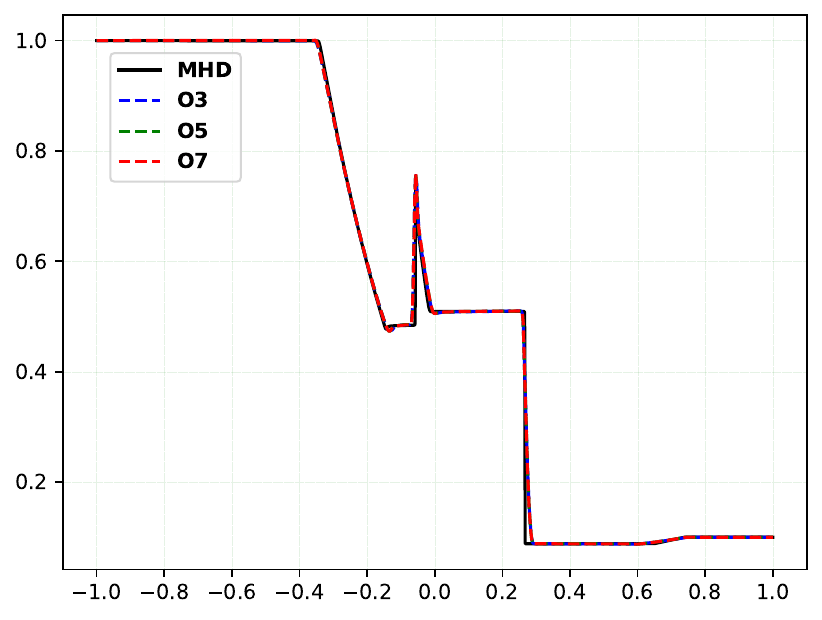}}\\
	\subfigure[$\per$ (without source term)]{\includegraphics[width=2.0in, height=1.5in]{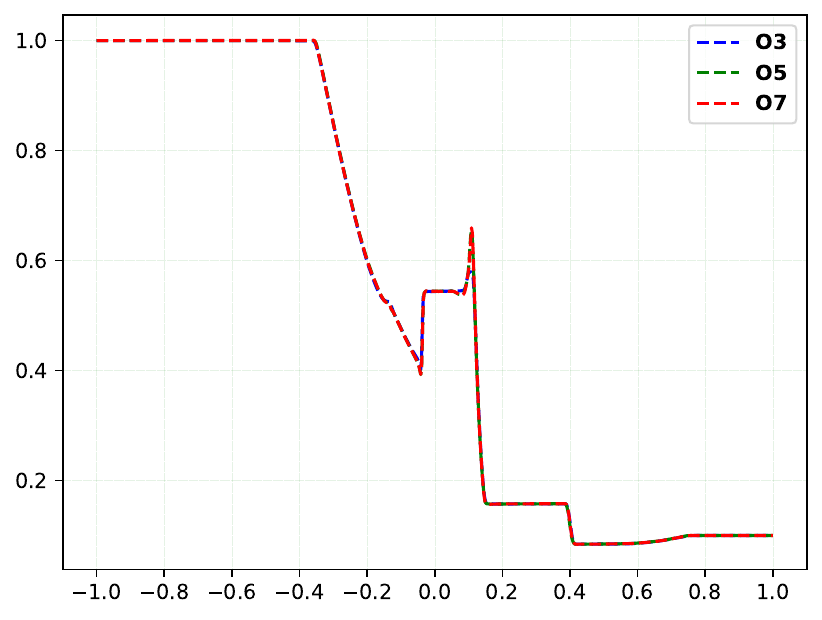}}
    \subfigure[$\per$ (with source term)]{\includegraphics[width=2.0in, height=1.5in]{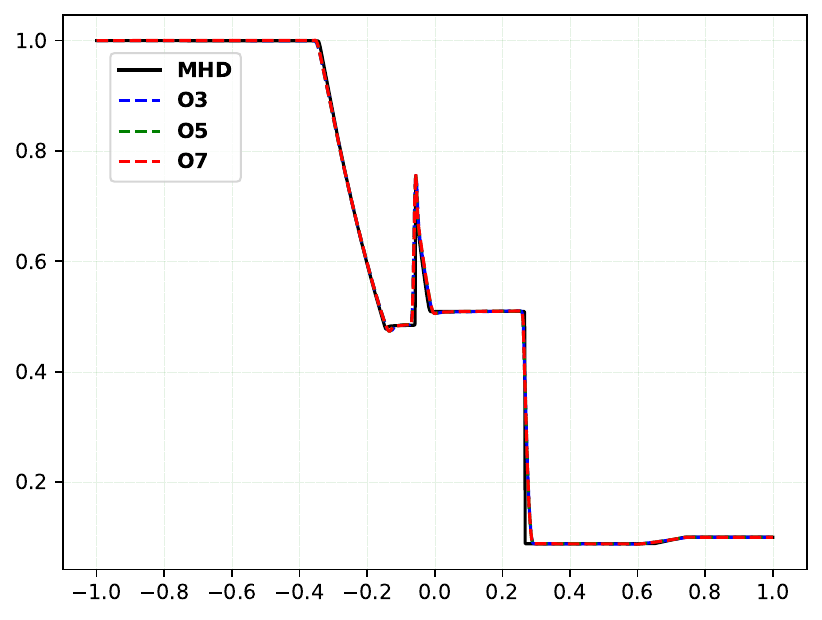}}
	\caption{\textbf{\nameref{test:rp1:brio}:} Plots of density, velocity in $x$-direction and parallel and perpendicular pressure components for 3rd, 5th, and 7th order numerical schemes without and with source term using the HLL Riemann solver and 800 cells at final time t = 0.2.}
	\label{fig:3}
\end{center}
\end{figure}
\begin{figure}[!htbp]
\begin{center}
	\subfigure[$\rho$ (without source term)]{\includegraphics[width=2.0in, height=1.5in]{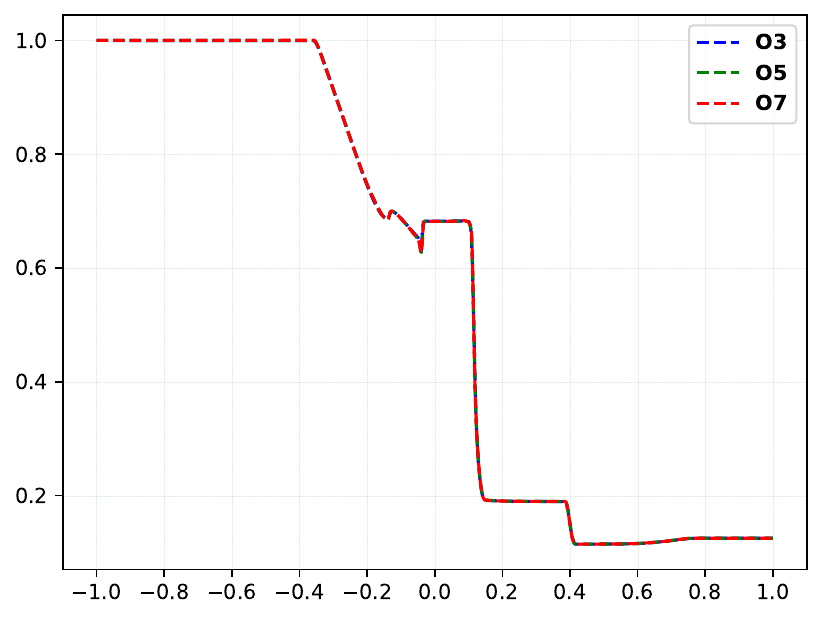}}
    \subfigure[$\rho$ (with source term)]{\includegraphics[width=2.0in, height=1.5in]{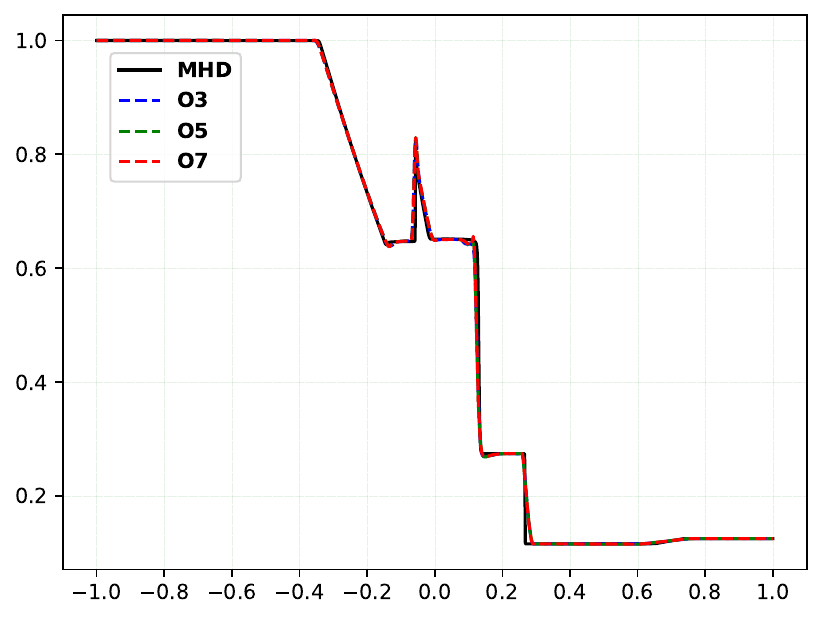}}\\
	\subfigure[$u_x$ (without source term)]{\includegraphics[width=2.0in, height=1.5in]{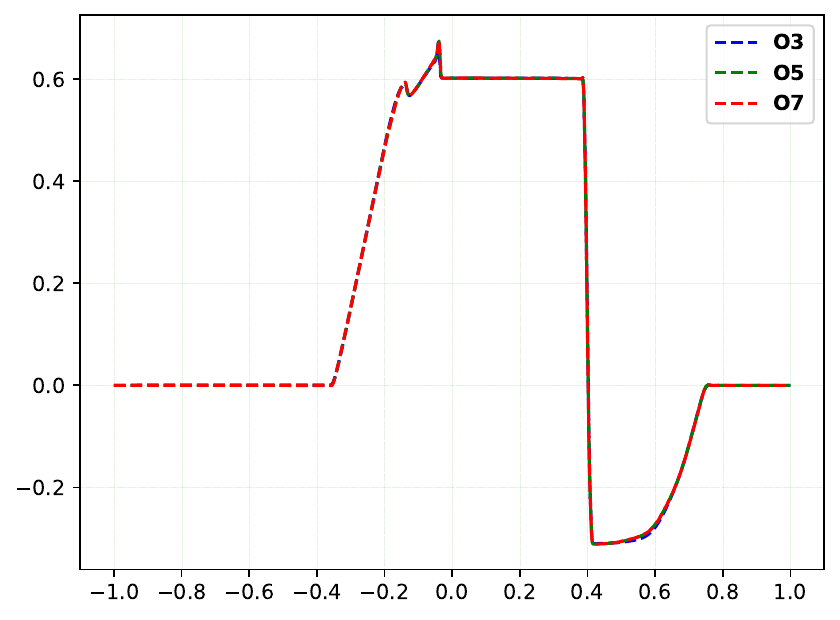}}
    \subfigure[$u_x$ (with source term)]{\includegraphics[width=2.0in, height=1.5in]{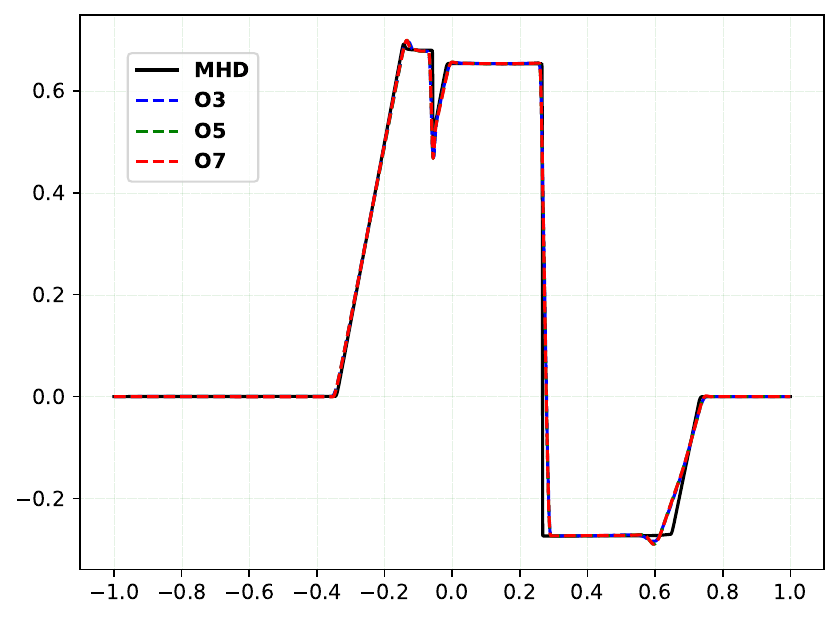}}\\
	\subfigure[$\pll$ (without source term)]{\includegraphics[width=2.0in, height=1.5in]{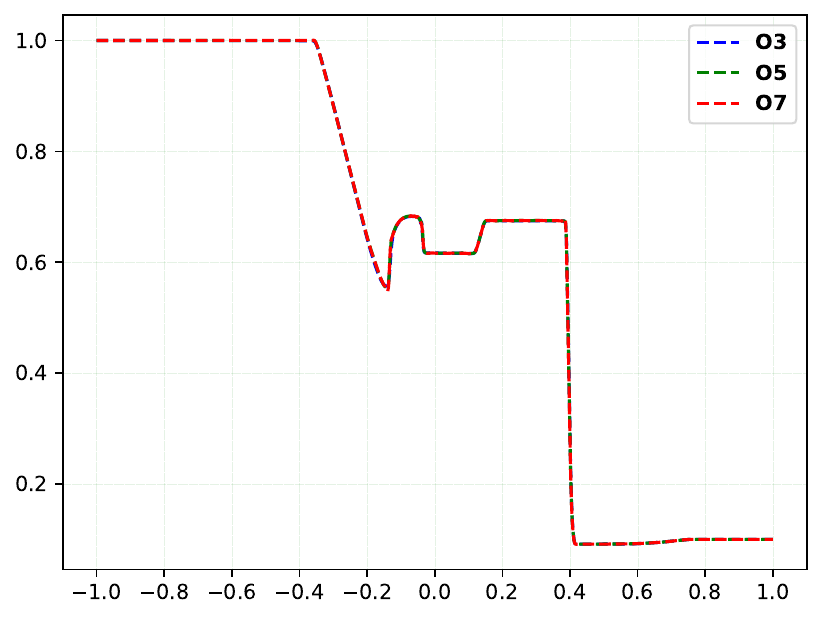}}
    \subfigure[$\pll$ (with source term)]{\includegraphics[width=2.0in, height=1.5in]{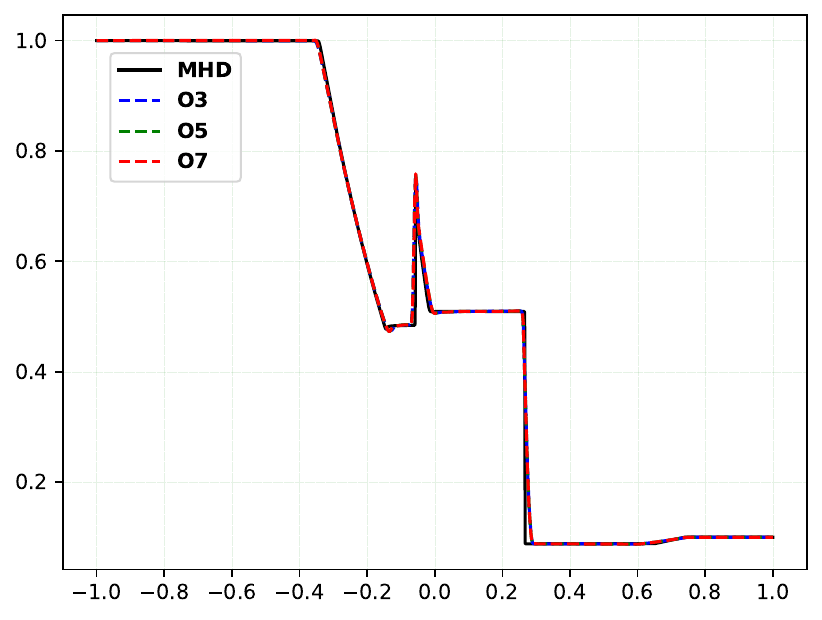}}\\
	\subfigure[$\per$ (without source term)]{\includegraphics[width=2.0in, height=1.5in]{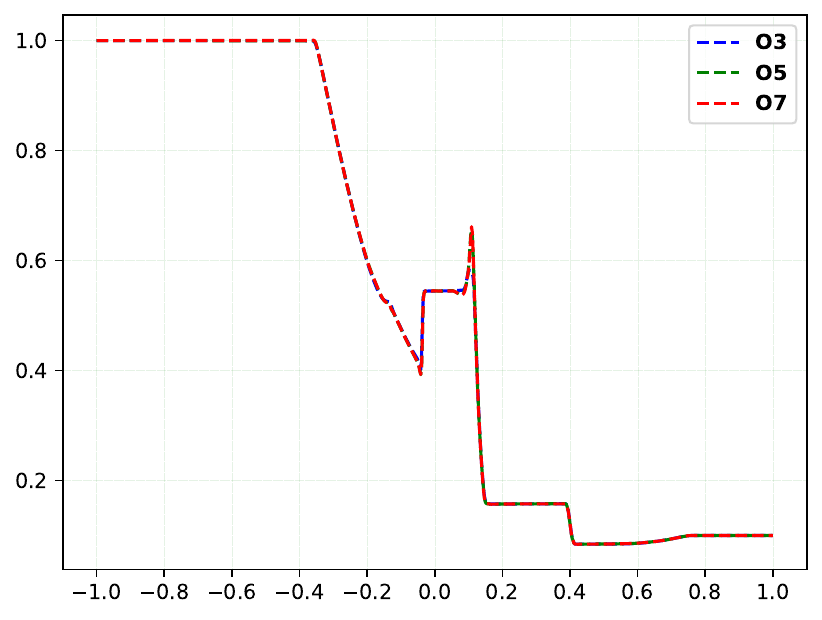}}
    \subfigure[$\per$ (with source term)]{\includegraphics[width=2.0in, height=1.5in]{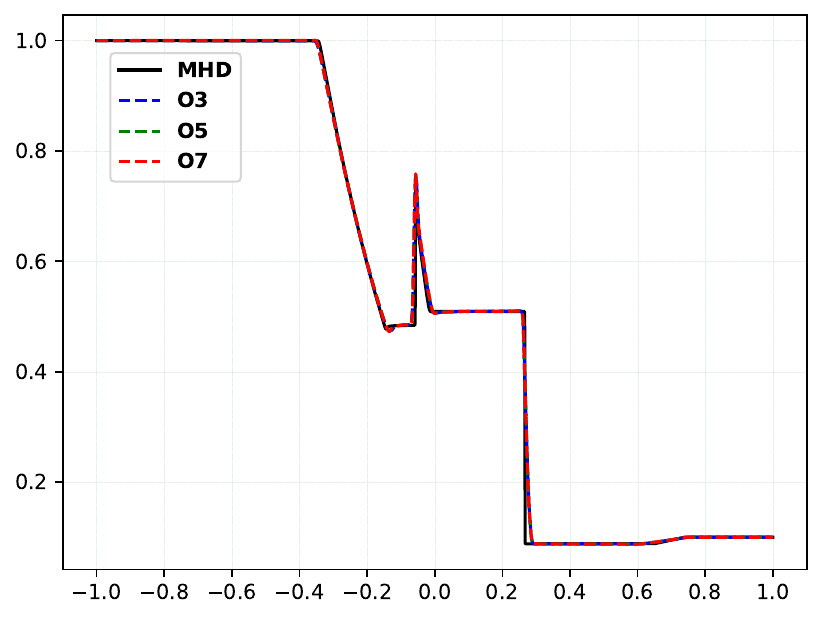}}
	\caption{\textbf{\nameref{test:rp1:brio}:} Plots of density, velocity in $x$-direction and parallel and perpendicular pressure components for 3rd, 5th, and 7th order numerical schemes without and with source term using the HLLI Riemann solver and 800 cells at final time t = 0.2.}
	\label{fig:3.1}
\end{center}
\end{figure}
\rev{The numerical results using the HLL and HLLI Riemann solvers are presented in Fig.$\eqref{fig:3}$ and Fig.$\eqref{fig:3.1}$}, respectively. We have plotted density, velocity in $x$-directions, and pressure components. We observe that without the source terms, we have the CGL solution (anisotropic case) with additional waves, and the solution structure is entirely different from that of the isotropic MHD solution. We also note that the proposed schemes are able to resolve all the waves without any oscillations for both solvers.

For the isotropic case (with source terms), we have noted that the solutions match the MHD solution, and both the pressure components have the same profile. Furthermore, the proposed schemes are able to resolve all the isotropic waves. We also do not see any significant difference in the results for both solvers. The results are similar to those presented in \cite{singh2024entropy} for the CGL equations and in \cite{Hirabayashi2016new} gyrotropic anisotropic and isotropic cases. 

\subsection{Riemann problem 2: Ryu–Jones test problem}
\label{test:rp2:ryu}
In this test case, we consider the CGL generalization of the Ruy-Jones test problem for the MHD equation proposed in \cite{Ryu1995Numerical}. The CGL version is presented in \cite{singh2024entropy}, which is based on the MHD data given in \cite{Chandrashekar2016Entropy}. The computational domain is $[-0.5,0.5]$ with an outflow boundary condition. The initial conditions are given by,
\[(\rho, u_{x}, u_{y}, u_{z}, \pll, \per, B_{y}, B_{z}) = \begin{cases}
(1.08, 1.2, 0.0, 0.0, 0.95, 0.95, 3.6, 2.0), & \textrm{if } x\leq 0\\
(1, 0, 0, 0, 1, 1, 4, 2), & \textrm{otherwise}
\end{cases}\]
with $B_{x}=2$.
We again consider the equations without (anisotropic case) and with (isotropic cases) source terms. The simulations are performed using $800$ cells, and we compute till the final time of $t=0.2$ with CFL number 0.8. This problem was run on an 800-zone mesh spanning the domain $[-0.5,0.5]$.  

\begin{figure}[!htbp]
\begin{center}
	\subfigure[$\rho$ (without source term)]{\includegraphics[width=2.0in, height=1.5in]{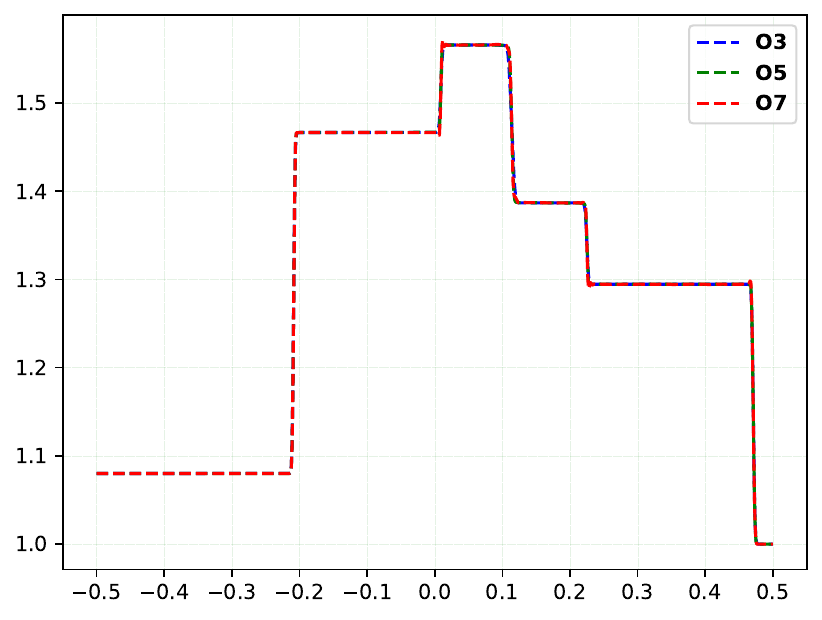}}
    \subfigure[$\rho$ (with source term)]{\includegraphics[width=2.0in, height=1.5in]{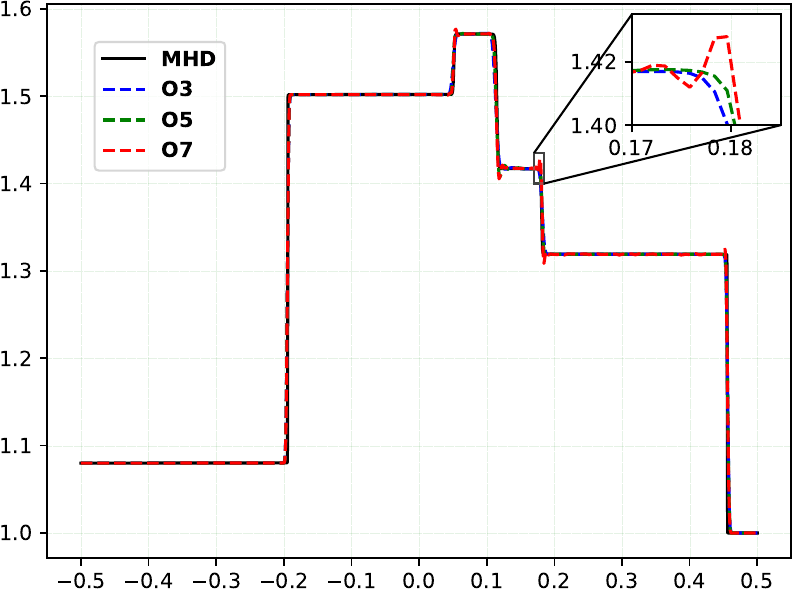}}\\
	\subfigure[$\pll$ (without source term)]{\includegraphics[width=2.0in, height=1.5in]{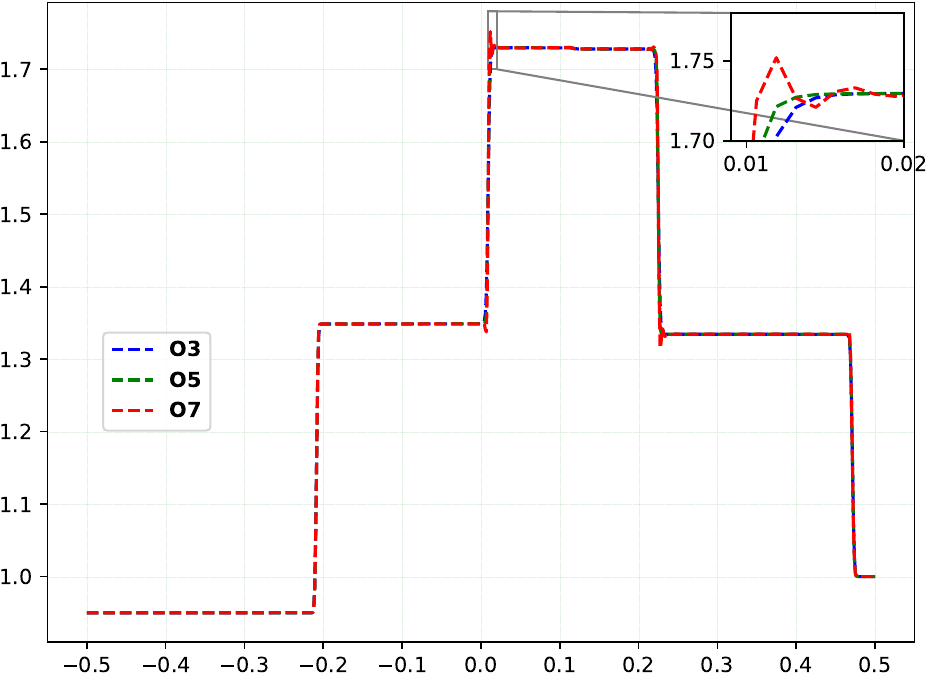}}
    \subfigure[$\pll$ (with source term)]{\includegraphics[width=2.0in, height=1.5in]{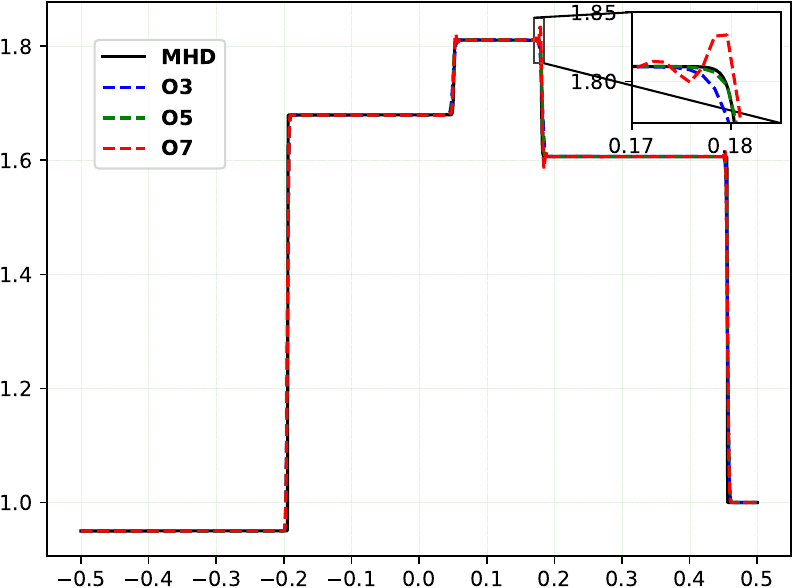}}\\
	\subfigure[$\per$ (without source term)]{\includegraphics[width=2.0in, height=1.5in]{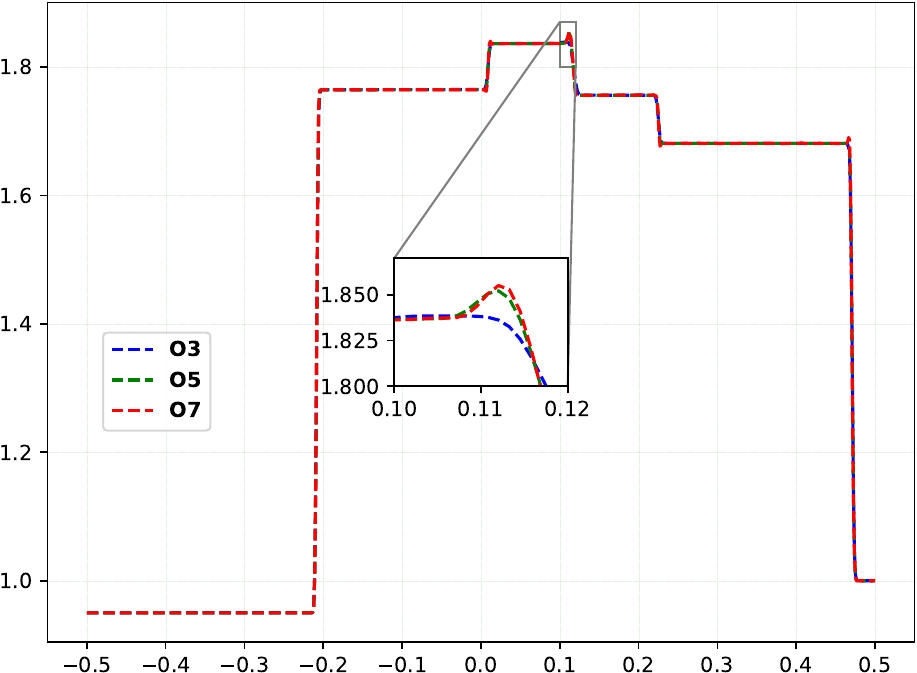}}
    \subfigure[$\per$ (with source term)]{\includegraphics[width=2.0in, height=1.5in]{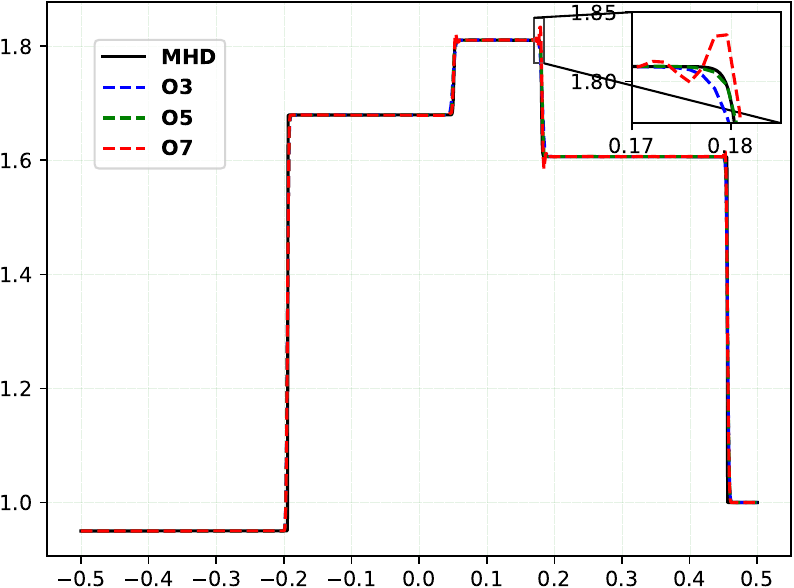}}\\
	\subfigure[$B_y$ (without source term)]{\includegraphics[width=2.0in, height=1.5in]{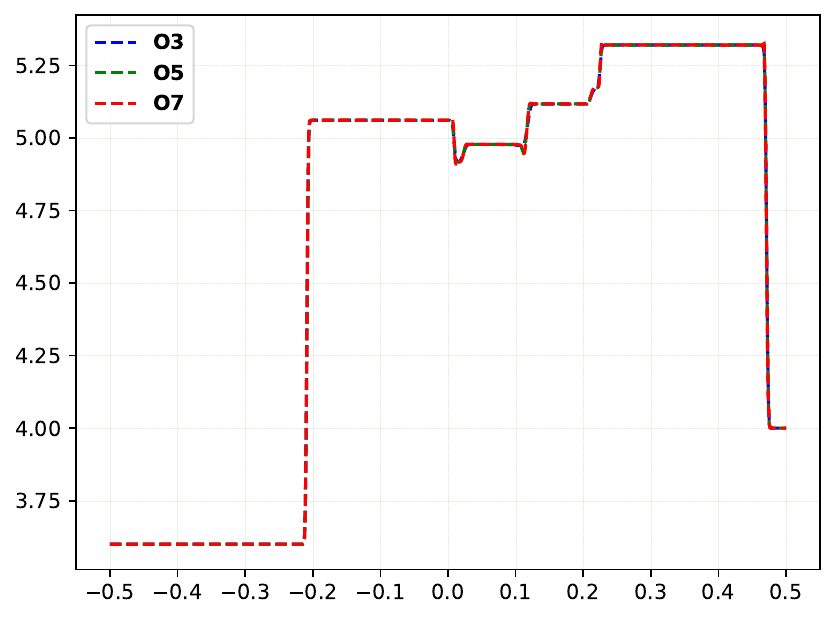}}
    \subfigure[$B_y$ (with source term)]{\includegraphics[width=2.0in, height=1.5in]{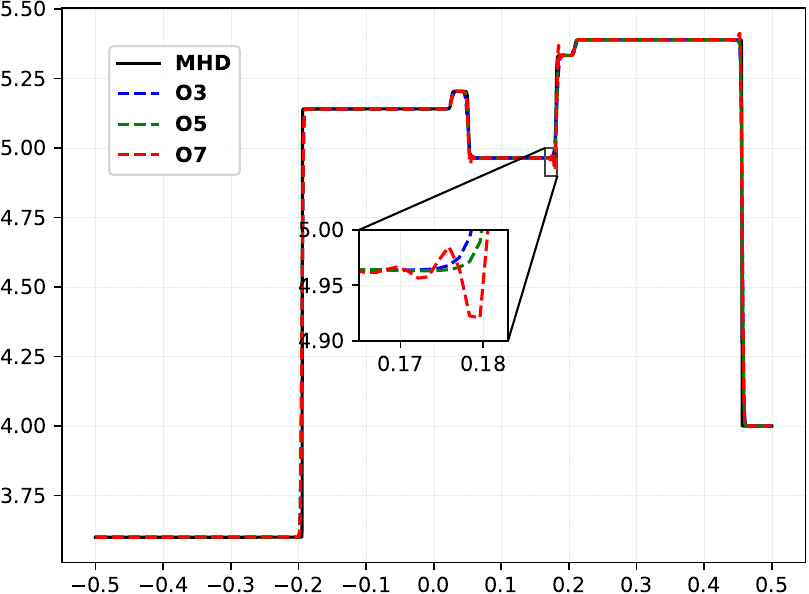}}
	\caption{\textbf{\nameref{test:rp2:ryu}:} Plots of density, parallel and perpendicular pressure components, and magnetic field in $y$ direction for 3rd, 5th, and 7th order numerical schemes without and with source term using the HLL Riemann solver and 800 cells at final time t = 0.2.}
	\label{fig:4}
\end{center}
\end{figure}
\begin{figure}[!htbp]
\begin{center}
	\subfigure[$\rho$ (without source term)]{\includegraphics[width=2.0in, height=1.5in]{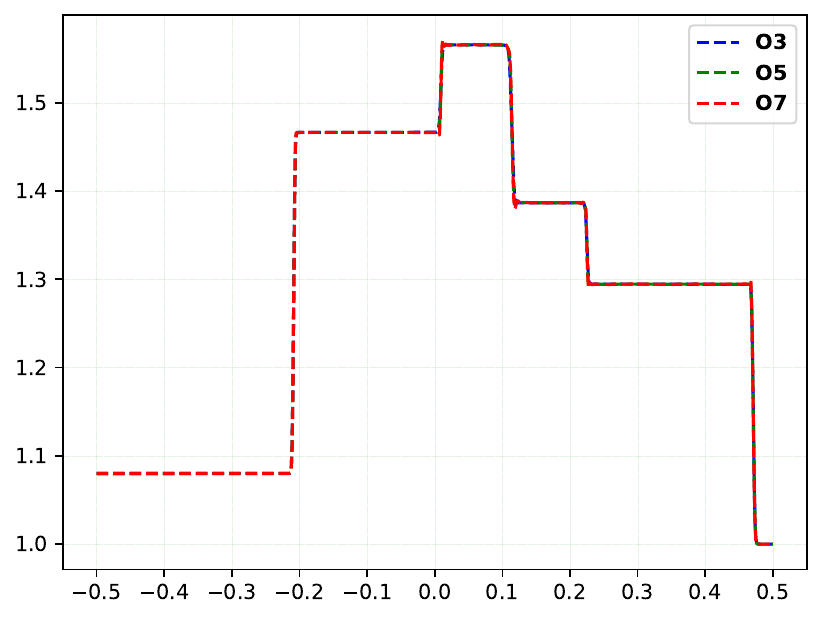}}
    \subfigure[$\rho$ (with source term)]{\includegraphics[width=2.3in, height=1.7in]{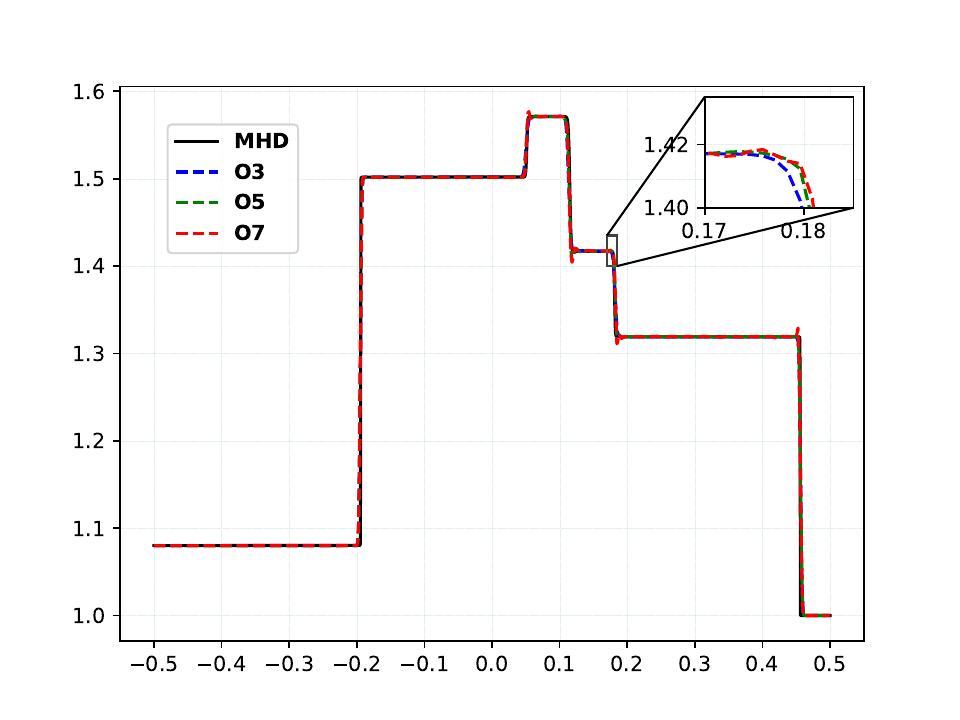}}\\
	\subfigure[$\pll$ (without source term)]{\includegraphics[width=2.0in, height=1.5in]{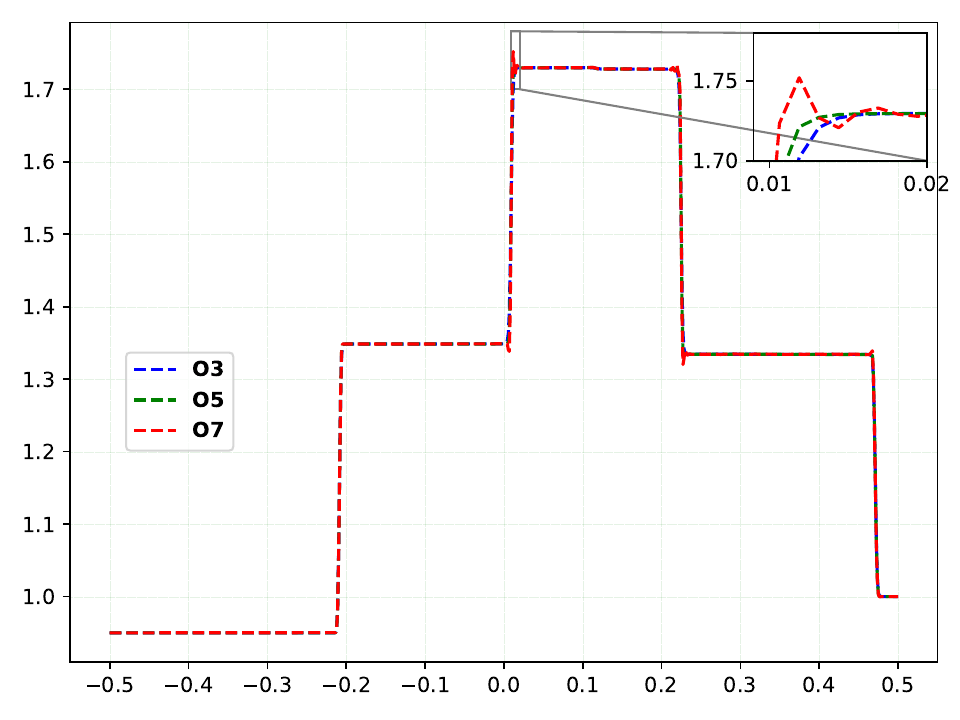}}
    \subfigure[$\pll$ (with source term)]{\includegraphics[width=2.3in, height=1.7in]{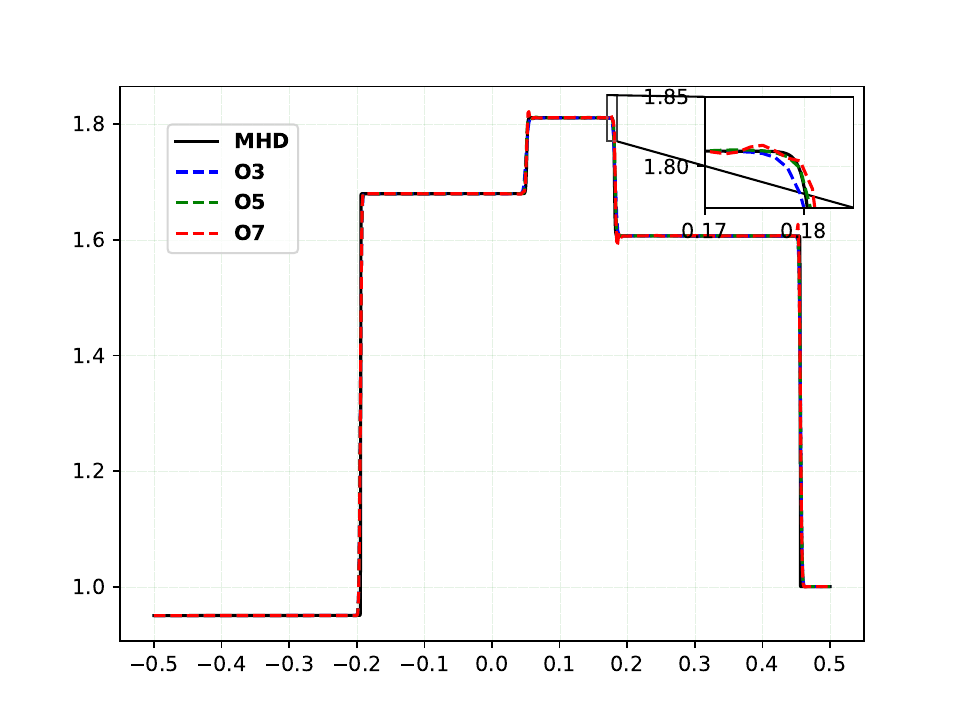}}\\
	\subfigure[$\per$ (without source term)]{\includegraphics[width=2.0in, height=1.5in]{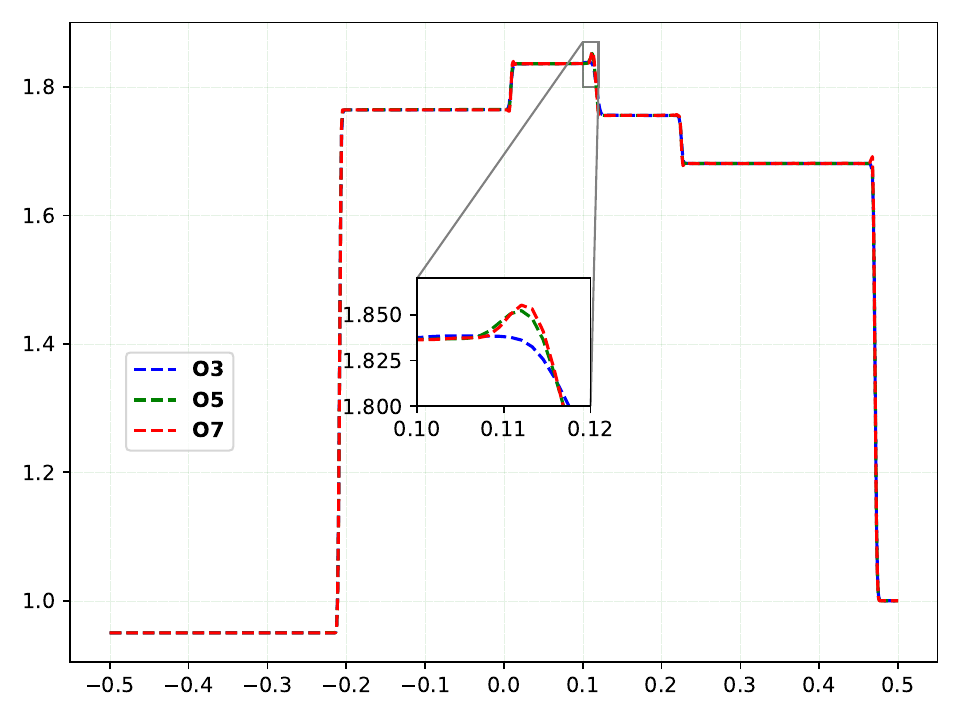}}
    \subfigure[$\per$ (with source term)]{\includegraphics[width=2.3in, height=1.7in]{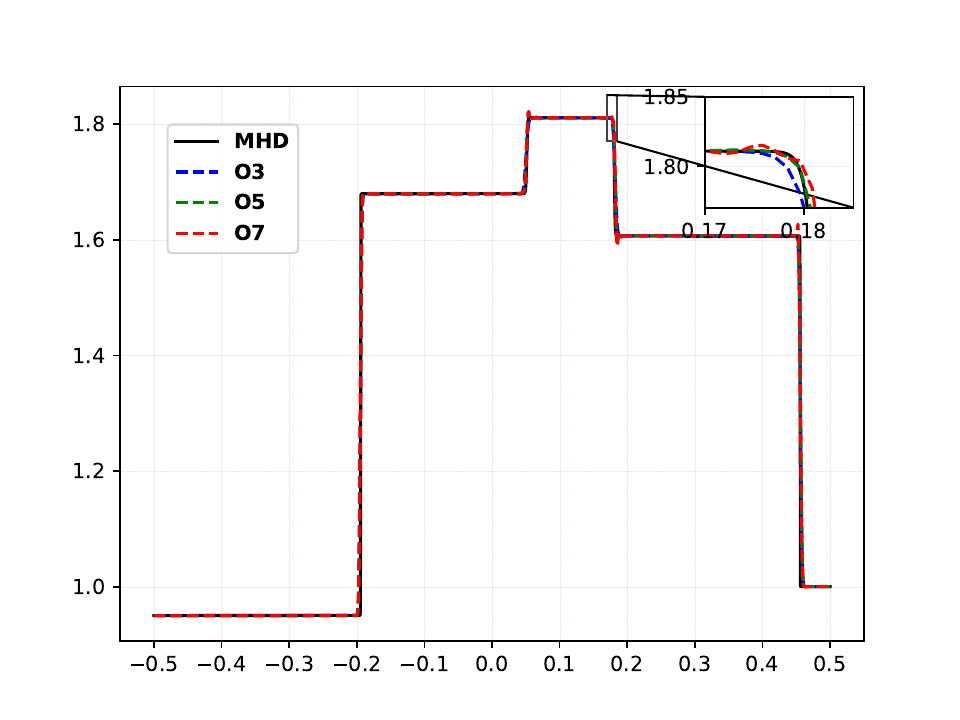}}\\
	\subfigure[$B_y$ (without source term)]{\includegraphics[width=2.0in, height=1.5in]{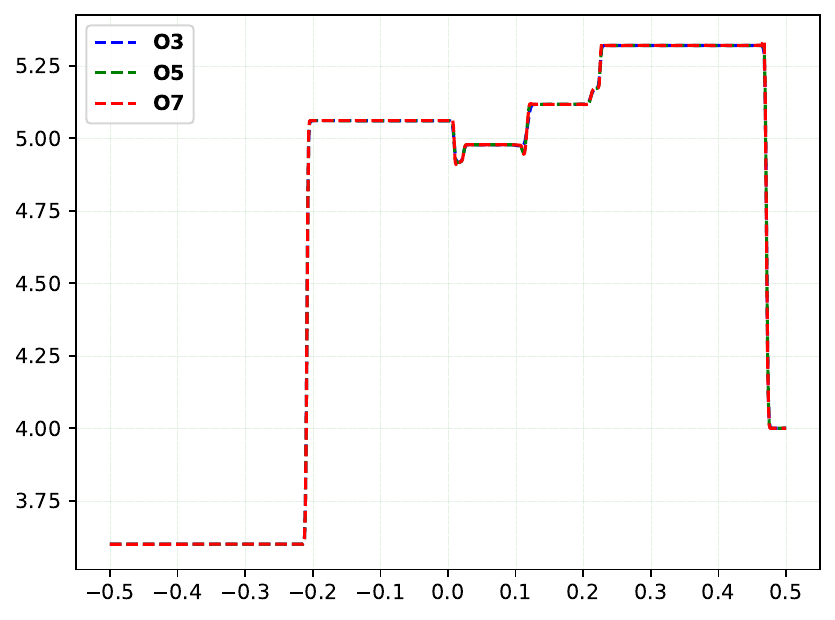}}
    \subfigure[$B_y$ (with source term)]{\includegraphics[width=2.3in, height=1.7in]{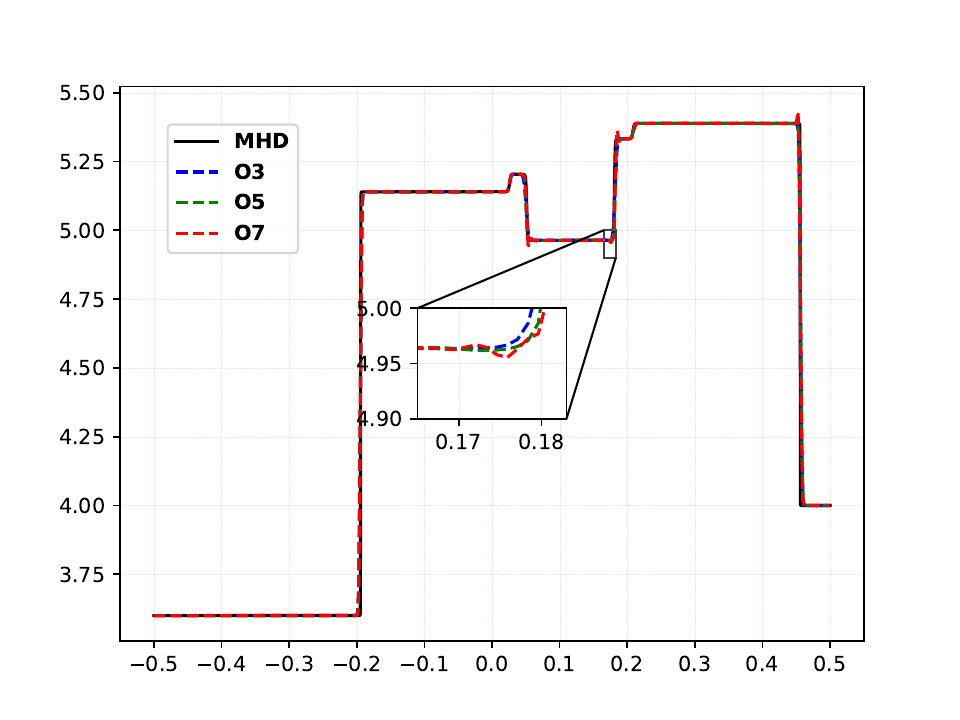}}
	\caption{\textbf{\nameref{test:rp2:ryu}:} Plots of density, parallel and perpendicular pressure components, and magnetic field in $y$ direction for 3rd, 5th, and 7th order numerical schemes without and with source term using the HLLI Riemann solver and 800 cells at final time t = 0.2.}
	\label{fig:4.1}
\end{center}
\end{figure}
\rev{The numerical results using the HLL Riemann solver are presented in Fig.$\eqref{fig:4}$ and using the HLLI Riemann solver are presented in Fig.$\eqref{fig:4.1}$}. Again, we observe that for anisotropic cases (without source terms), the proposed schemes have resolved all the waves. We do observe some small-scale oscillations for the 7th-order scheme near oscillations, but they are stable with respect to further refinement. We do observe that the HLLI solver does decrease the oscillations for the 7th order scheme in the isotropic case (see Fig.$\eqref{fig:4.1}$). Furthermore, the wave structure of the solution matches with the results presented in \cite{singh2024entropy}. 

For the isotropic case (with source terms), we again see that both the pressure components have the same profile, and all the variables match with the MHD reference solution. Again, the 7th-order scheme has some small-scale oscillations, but the 3rd and 5th-order schemes are oscillation-free. Again, the results are similar to those presented in \cite{singh2024entropy}.

\subsection{Riemann Problem 3}
\label{test:rp3}
In this test case, we again generalize an MHD test case from \cite{dumbser2016new}, which is generalized to the CGL test case in \cite{singh2024entropy}. The computational domain is $[-0.5,0.5]$, and we consider the outflow boundary conditions. The initial conditions are given as follows:
\[(\rho, u_{x}, u_{y}, u_{z}, \pll, \per, B_{y}, B_{z}) = \begin{cases}
	(1.7, 0, 0, 0, 1.7, 1.7, 3.544908, 0), & \textrm{if } x\leq 0\\
	(0.2, 0, 0, -1.496891, 0.2, 0.2, 2.785898, 2.192064), & \textrm{otherwise}
\end{cases}\]
with $B_{x}=3.899398$. Again, we consider the CGL model without source terms (anisotropic case) and with source terms (isotropic case). The simulations are performed on $800$ cells till the final time of $t=0.15$. 
\begin{figure}[!htbp]
\begin{center}
	\subfigure[$\rho$ (without source term)]{\includegraphics[width=2.0in, height=1.5in]{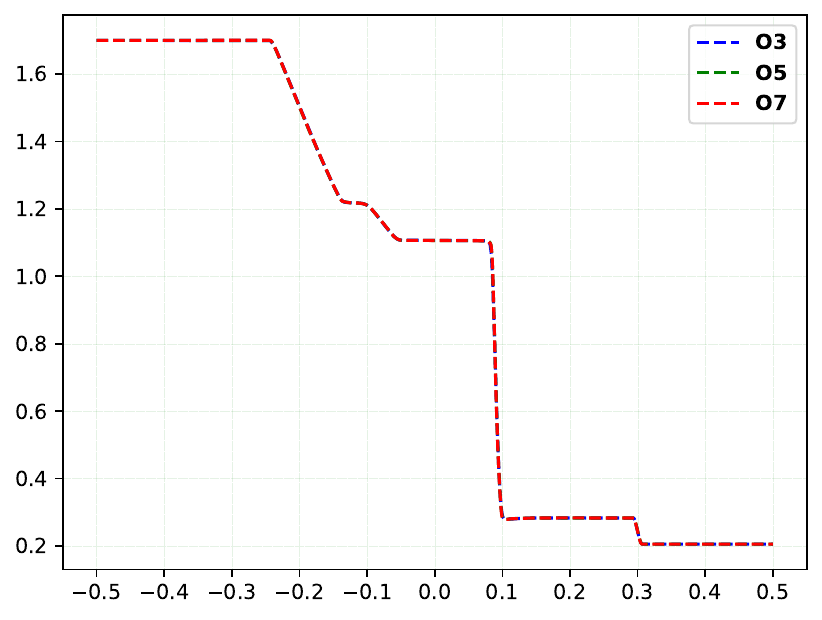}}
    \subfigure[$\rho$ (with source term)]{\includegraphics[width=2.0in, height=1.5in]{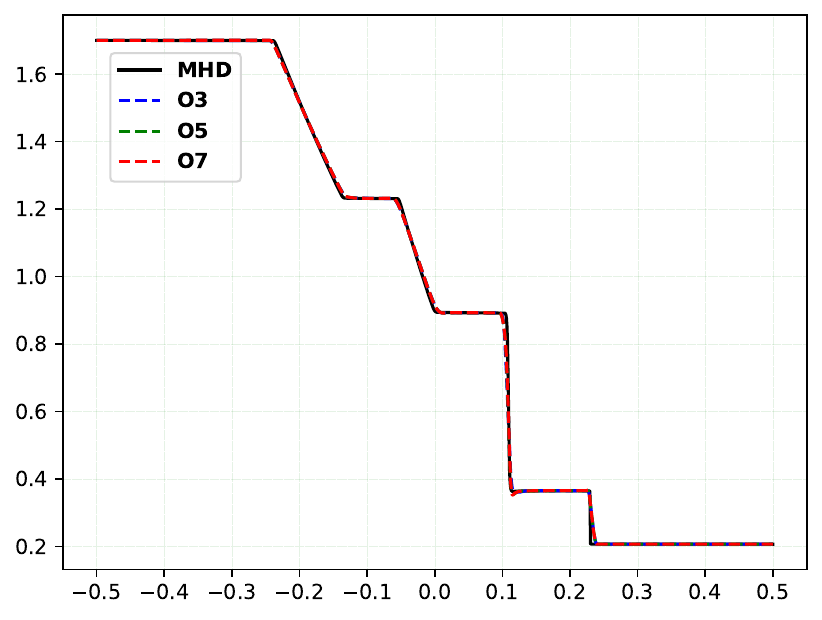}}\\
	\subfigure[$\pll$ (without source term)]{\includegraphics[width=2.0in, height=1.5in]{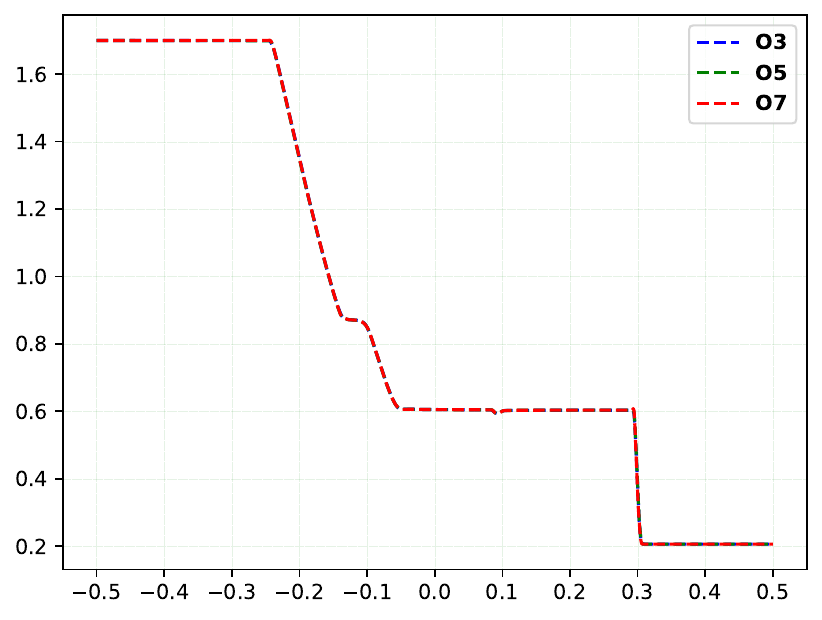}}
    \subfigure[$\pll$ (with source term)]{\includegraphics[width=2.0in, height=1.5in]{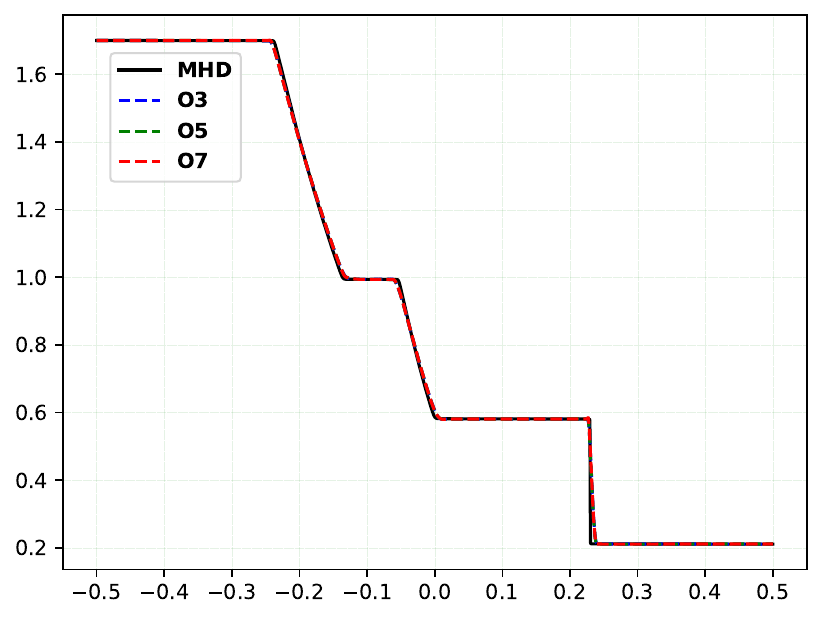}}\\
	\subfigure[$\per$ (without source term)]{\includegraphics[width=2.0in, height=1.5in]{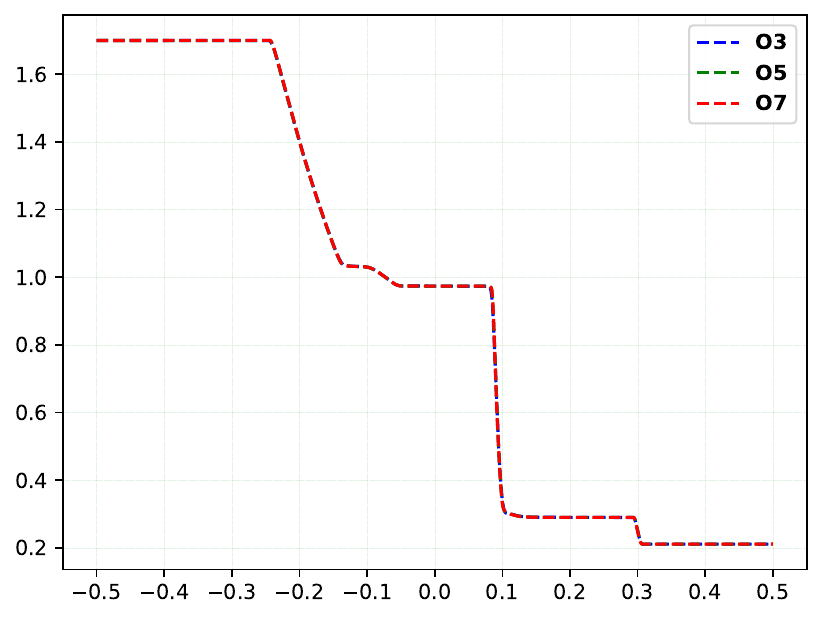}}
    \subfigure[$\per$ (with source term)]{\includegraphics[width=2.0in, height=1.5in]{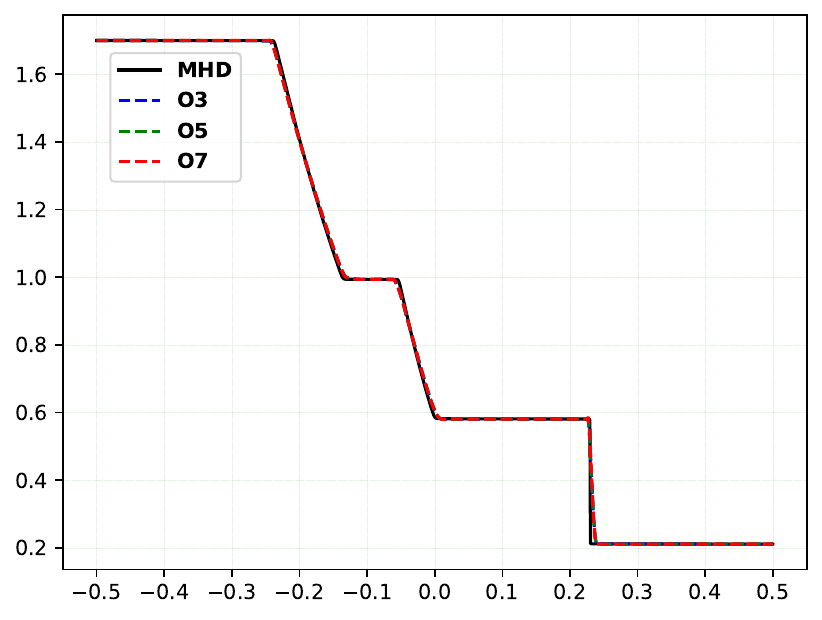}}\\
	\subfigure[$B_y$ (without source term)]{\includegraphics[width=2.0in, height=1.5in]{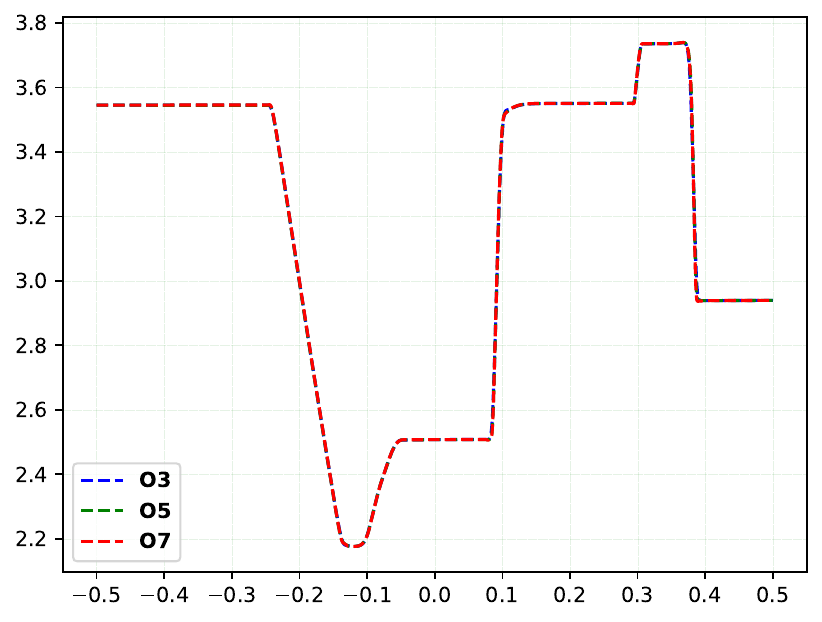}}
        \subfigure[$B_y$ (with source term)]{\includegraphics[width=2.0in, height=1.5in]{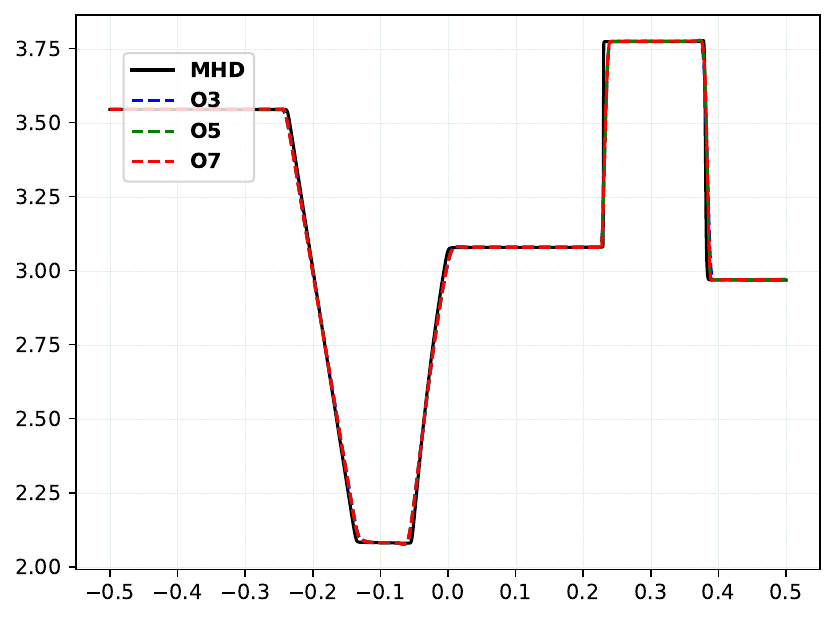}}
	\caption{\textbf{\nameref{test:rp3}:} Plots of density, parallel and perpendicular pressure components and magnetic field in $y$ direction for 3rd, 5th and 7th order numerical schemes without and with source term using the HLL Riemann solver and 800 cells at final time t = 0.15.}
	\label{fig:5}
\end{center}
\end{figure}
\begin{figure}[!htbp]
\begin{center}
	\subfigure[$\rho$ (without source term)]{\includegraphics[width=2.0in, height=1.5in]{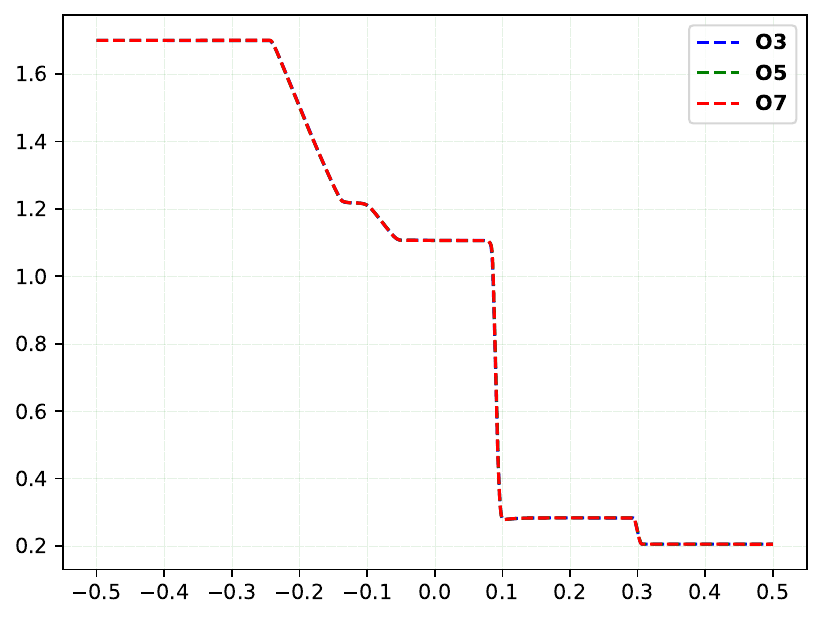}}
    \subfigure[$\rho$ (with source term)]{\includegraphics[width=2.0in, height=1.5in]{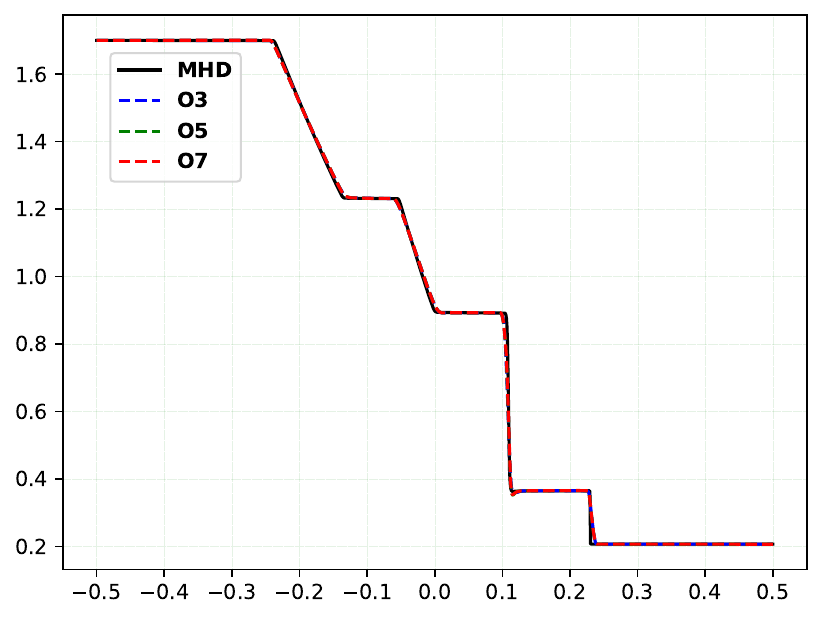}}\\
	\subfigure[$\pll$ (without source term)]{\includegraphics[width=2.0in, height=1.5in]{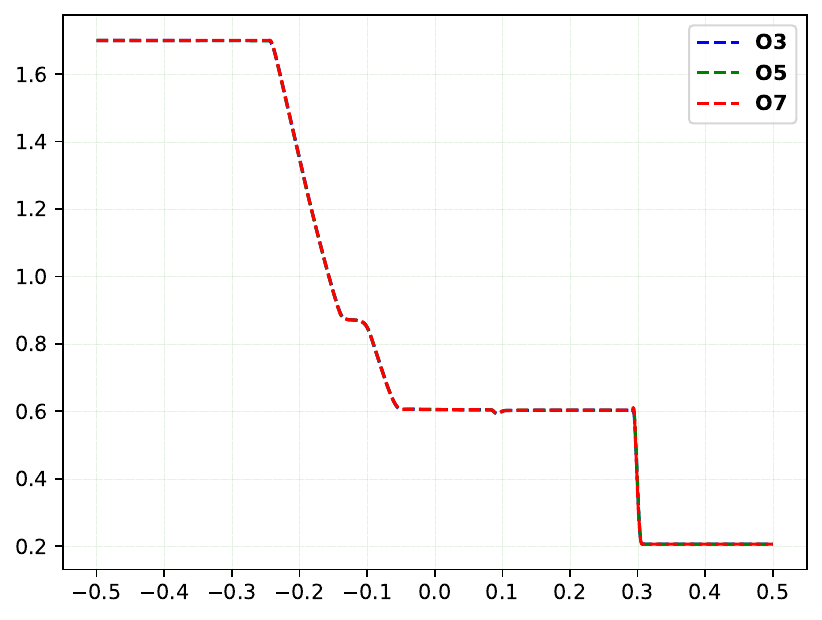}}
    \subfigure[$\pll$ (with source term)]{\includegraphics[width=2.0in, height=1.5in]{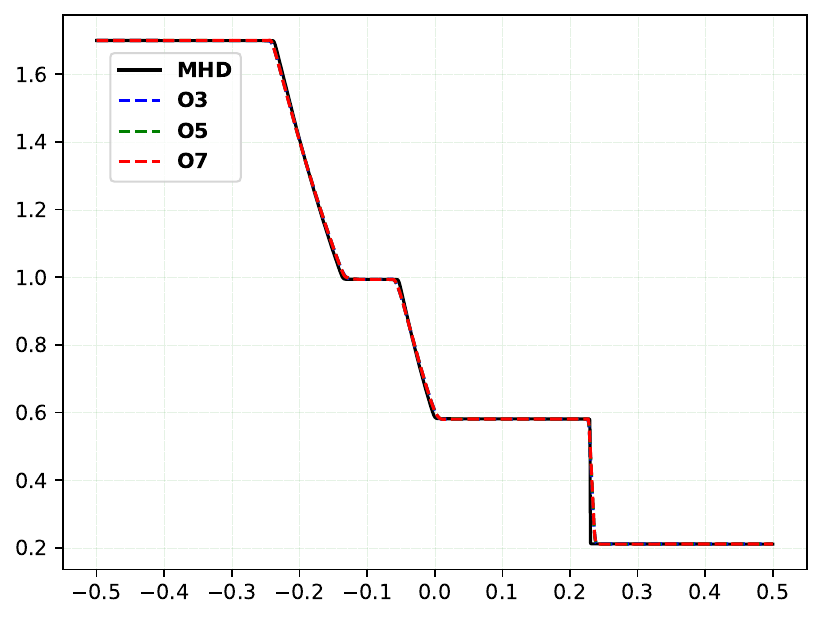}}\\
	\subfigure[$\per$ (without source term)]{\includegraphics[width=2.0in, height=1.5in]{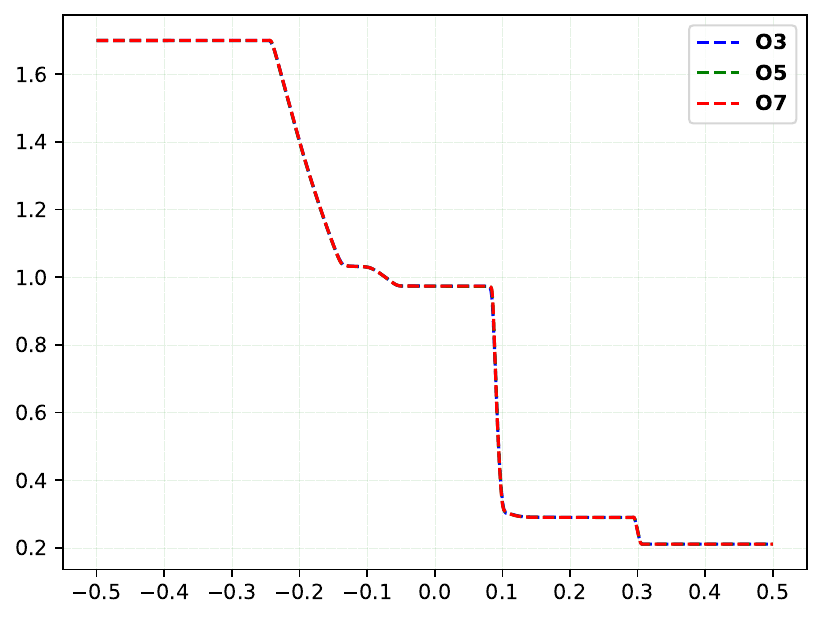}}
    \subfigure[$\per$ (with source term)]{\includegraphics[width=2.0in, height=1.5in]{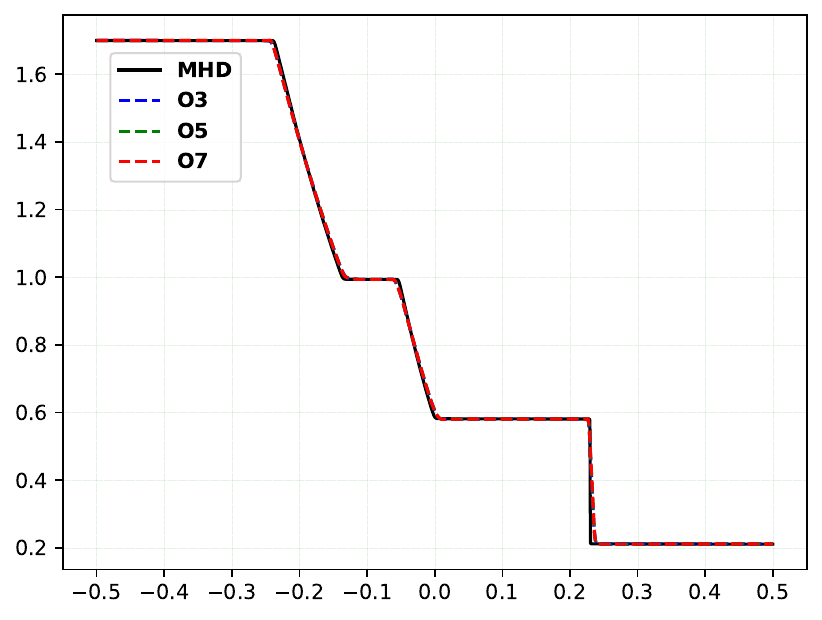}}\\
	\subfigure[$B_y$ (without source term)]{\includegraphics[width=2.0in, height=1.5in]{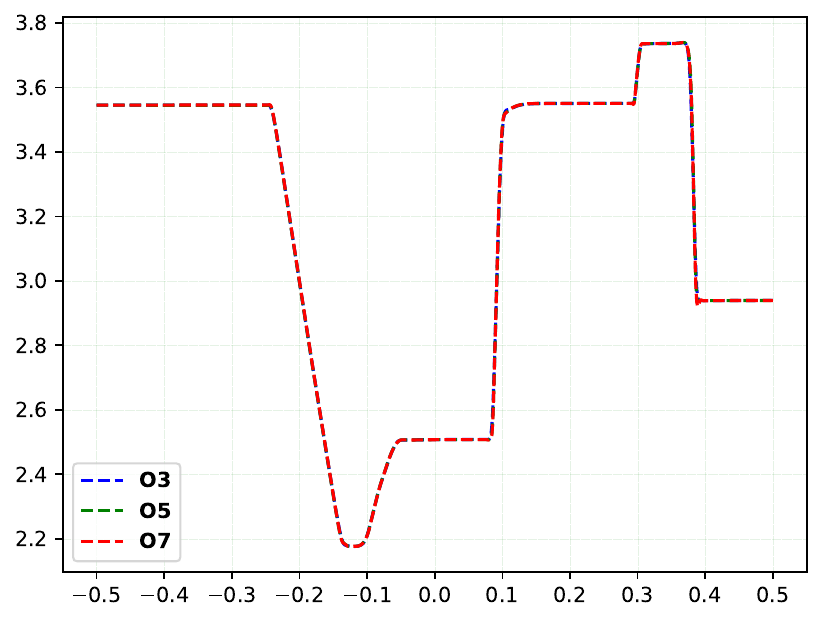}}
        \subfigure[$B_y$ (with source term)]{\includegraphics[width=2.0in, height=1.5in]{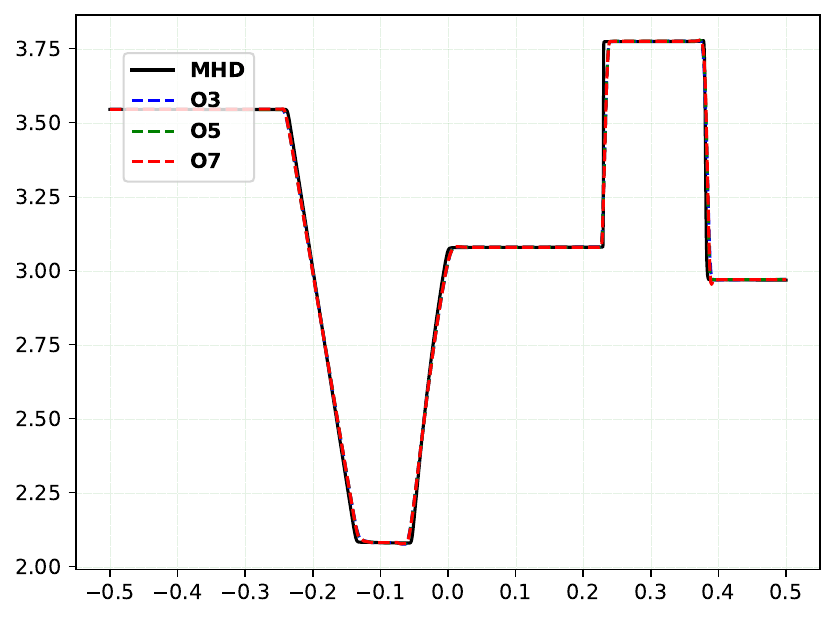}}
	\caption{\textbf{\nameref{test:rp3}:} Plots of density, parallel and perpendicular pressure components and magnetic field in $y$ direction for 3rd, 5th and 7th order numerical schemes without and with source term using the HLLI Riemann solver and 800 cells at final time t = 0.15.}
	\label{fig:5.1}
\end{center}
\end{figure}
\rev{The numerical results using the HLL and HLLI Riemann solvers are presented in Fig.$\eqref{fig:5}$ and Fig.$\eqref{fig:5.1}$}, respectively. We again observe that the results in both isotropic (with source terms) and anisotropic (without source terms) cases are similar to those presented in \cite{singh2024entropy}. In the isotropic case, both pressure profiles match the MHD pressure. We note that the proposed schemes are able to capture all the waves in both cases and do not produce any unphysical oscillations. Furthermore, both solvers have similar results.

\subsection{Riemann Problem 4}
\label{test:rp4}
We again consider a CGL test case from \cite{singh2024entropy}, which is modified from the MHD test case in \cite{dumbser2016new}. The computational domain is $[-0.5,0.5]$ with outflow boundary conditions. The initial conditions are given by,
\[(\rho, u_{x}, u_{y}, u_{z}, \pll, \per, B_{y}, B_{z}) = \begin{cases}
	(1, 0, 0, 0, 1, 1, \sqrt{4\pi}, 0), & \textrm{if } x\leq 0\\
	(0.4, 0, 0, 0, 0.4, 0.4, -\sqrt{4\pi}, 0), & \textrm{otherwise}
\end{cases}\]
Also, we take $B_{x}=1.3 \sqrt{4\pi}$.  The simulation is performed using $800$ computational zones till the final time of $t=0.15$. Numerical results are presented in Figures $\eqref{fig:6}$ and $\eqref{fig:6.1}$ using the HLL and HLLI solvers, respectively. 

We see that all the schemes with both solvers are able to resolve all the waves in both isotropic and anisotropic cases. Furthermore, in the isotropic case (with source terms), the computed solution matches the MHD reference solution. The results are similar to those presented in \cite{singh2024entropy}.
\begin{figure}[!htbp]
\begin{center}
	\subfigure[$\rho$ (without source term)]{\includegraphics[width=2.0in, height=1.5in]{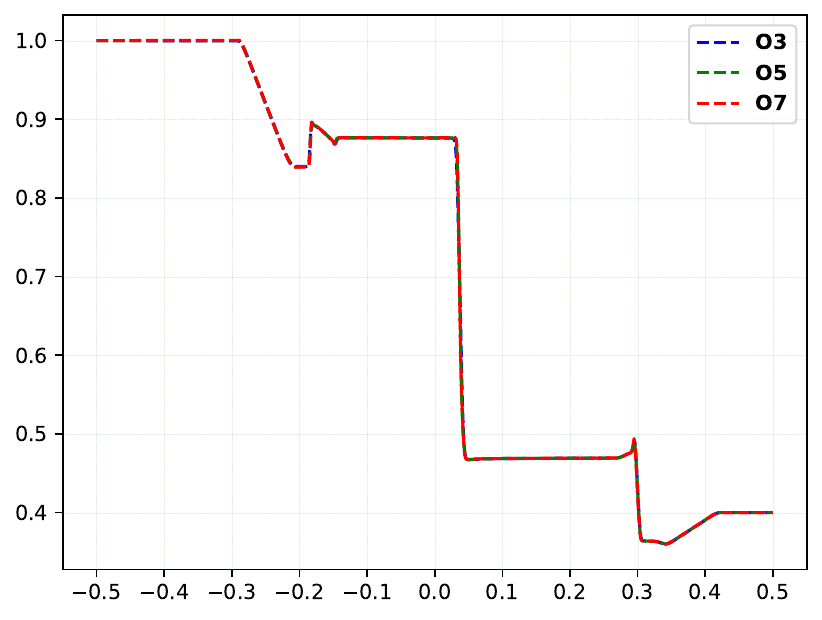}}
    \subfigure[$\rho$ (with source term)]{\includegraphics[width=2.0in, height=1.5in]{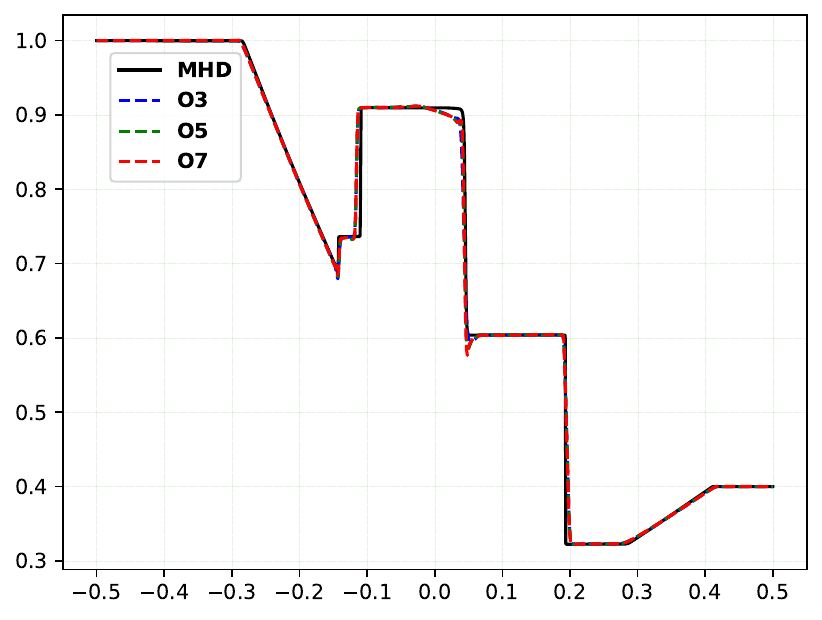}}\\
	\subfigure[$\pll$ (without source term)]{\includegraphics[width=2.0in, height=1.5in]{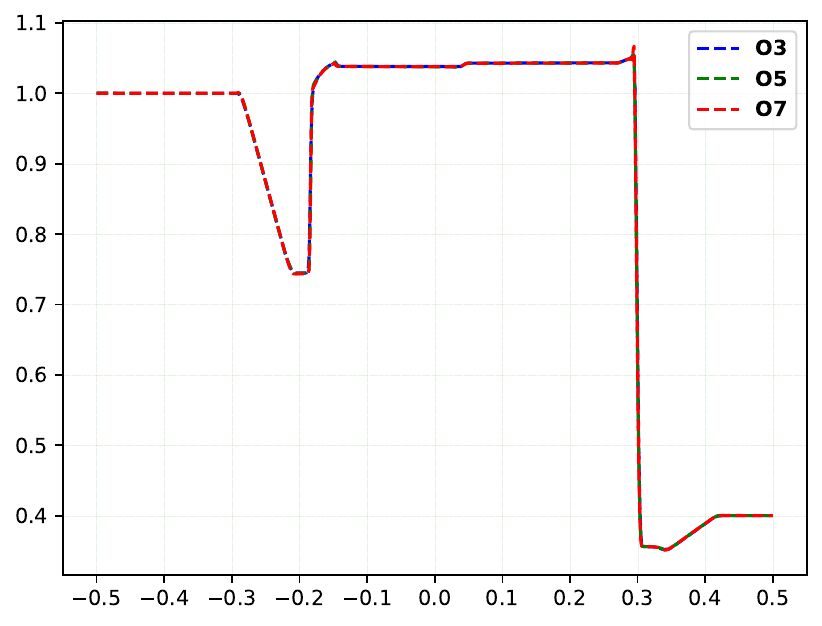}}
    \subfigure[$\pll$ (with source term)]{\includegraphics[width=2.0in, height=1.5in]{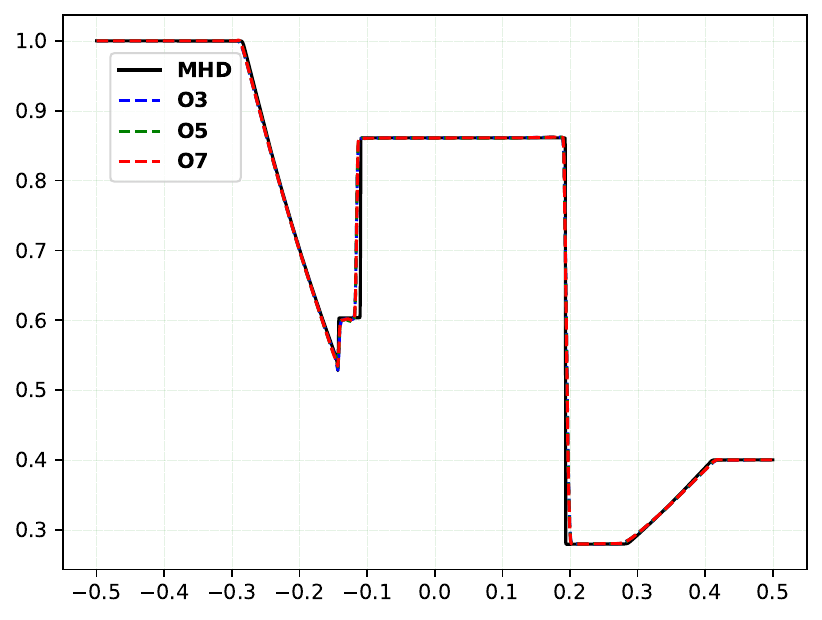}}\\
	\subfigure[$\per$ (without source term)]{\includegraphics[width=2.0in, height=1.5in]{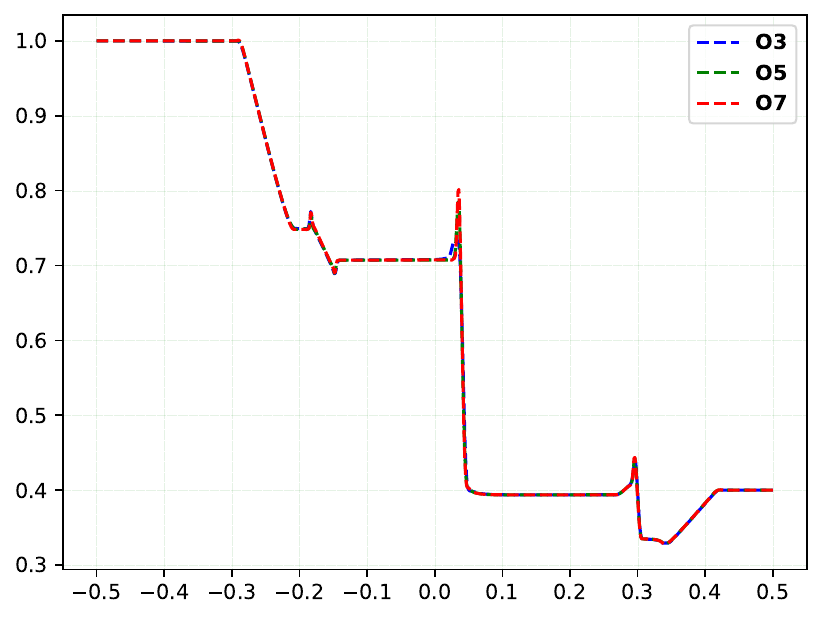}}
    \subfigure[$\per$ (with source term)]{\includegraphics[width=2.0in, height=1.5in]{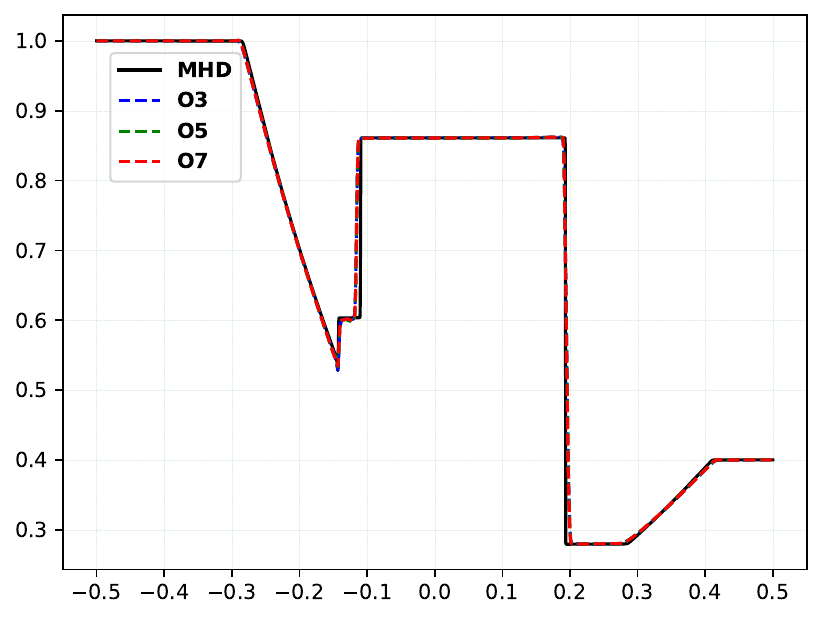}}\\
	\subfigure[$B_y$ (without source term)]{\includegraphics[width=2.0in, height=1.5in]{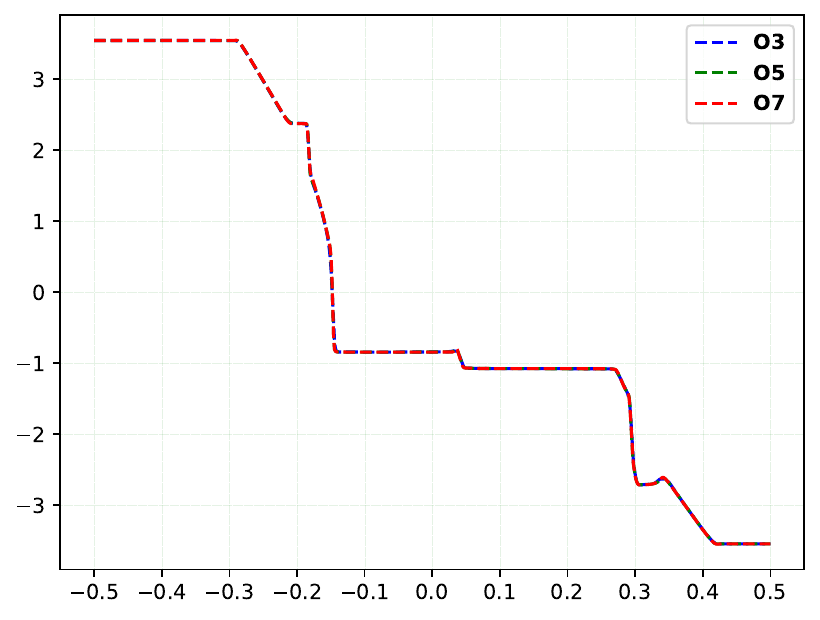}}
	\subfigure[$B_y$ (with source term)]{\includegraphics[width=2.0in, height=1.5in]{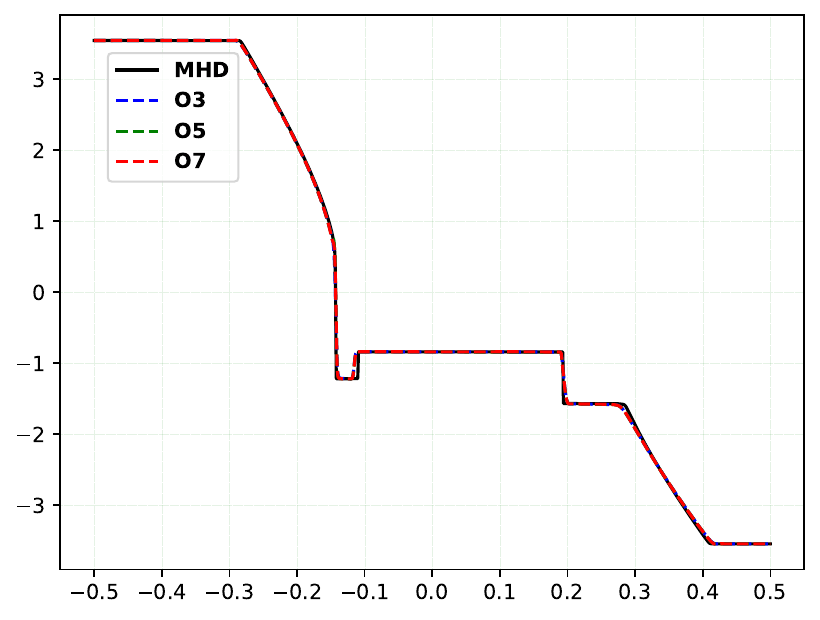}}
	\caption{\textbf{\nameref{test:rp4}:} Plots of density, parallel and perpendicular pressure components and magnetic field in $y$ direction for 3rd, 5th and 7th order numerical schemes without and with source term using the HLL Riemann solver and 800 cells at final time t = 0.15.}
	\label{fig:6}
\end{center}
\end{figure}
\begin{figure}[!htbp]
\begin{center}
	\subfigure[$\rho$ (without source term)]{\includegraphics[width=2.0in, height=1.5in]{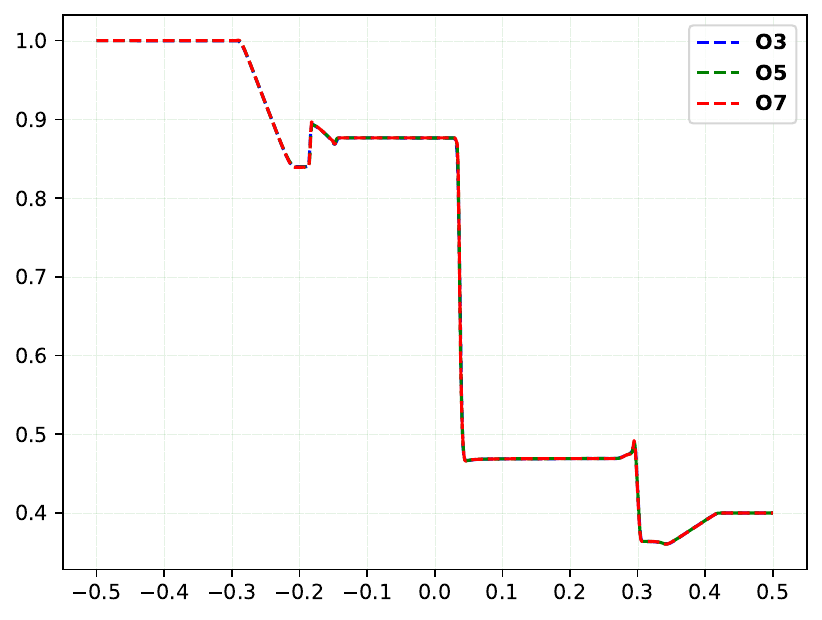}}
    \subfigure[$\rho$ (with source term)]{\includegraphics[width=2.0in, height=1.5in]{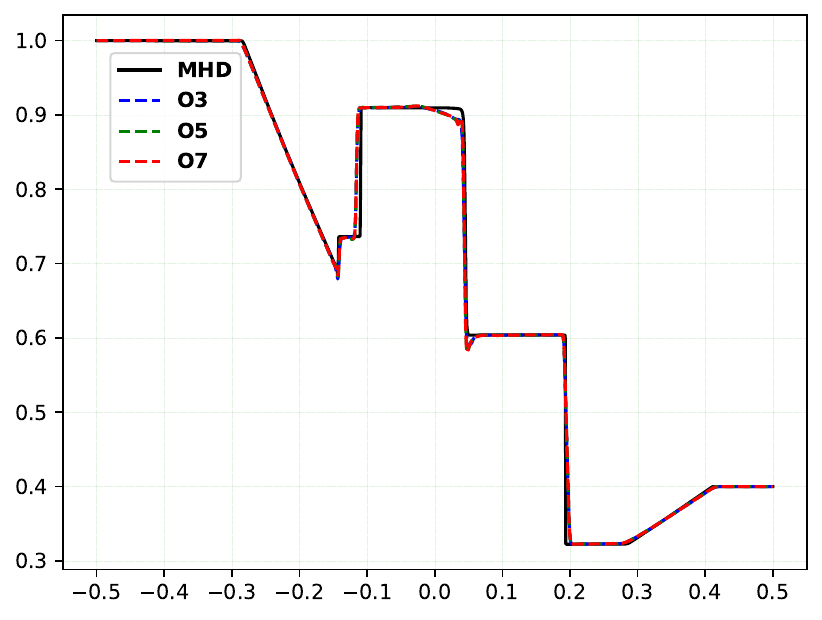}}\\
	\subfigure[$\pll$ (without source term)]{\includegraphics[width=2.0in, height=1.5in]{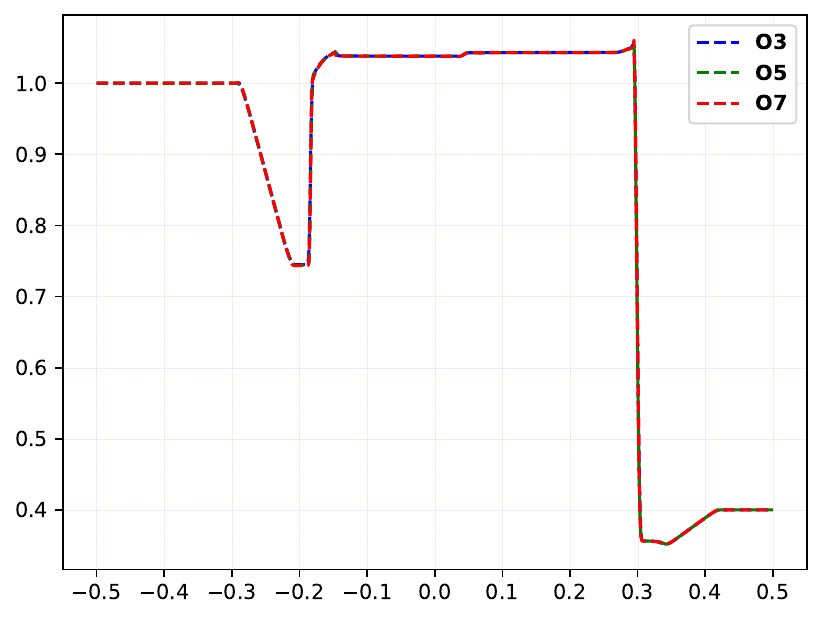}}
    \subfigure[$\pll$ (with source term)]{\includegraphics[width=2.0in, height=1.5in]{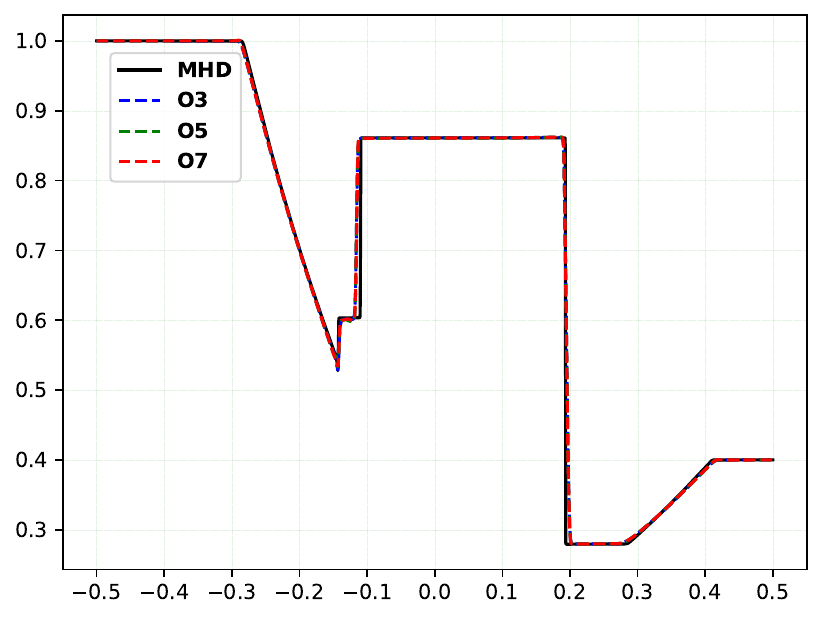}}\\
	\subfigure[$\per$ (without source term)]{\includegraphics[width=2.0in, height=1.5in]{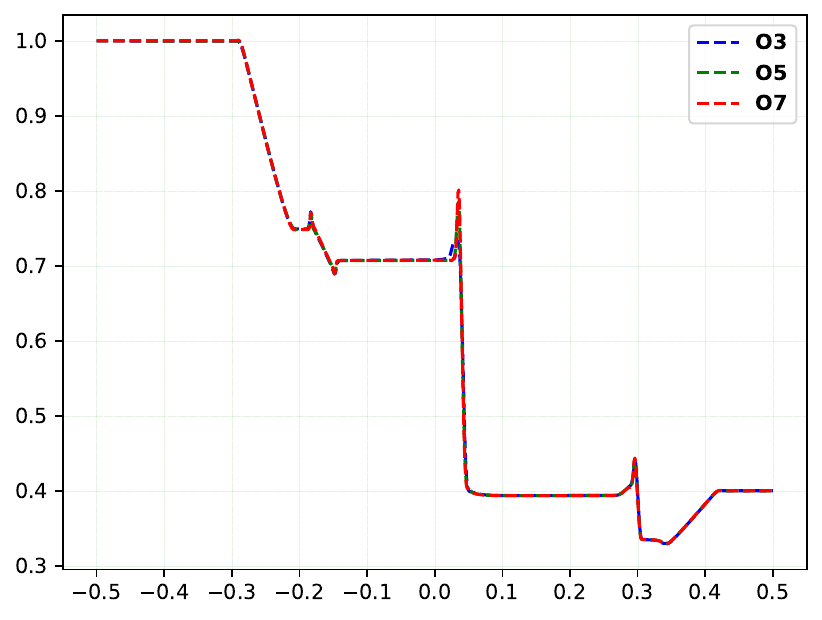}}
    \subfigure[$\per$ (with source term)]{\includegraphics[width=2.0in, height=1.5in]{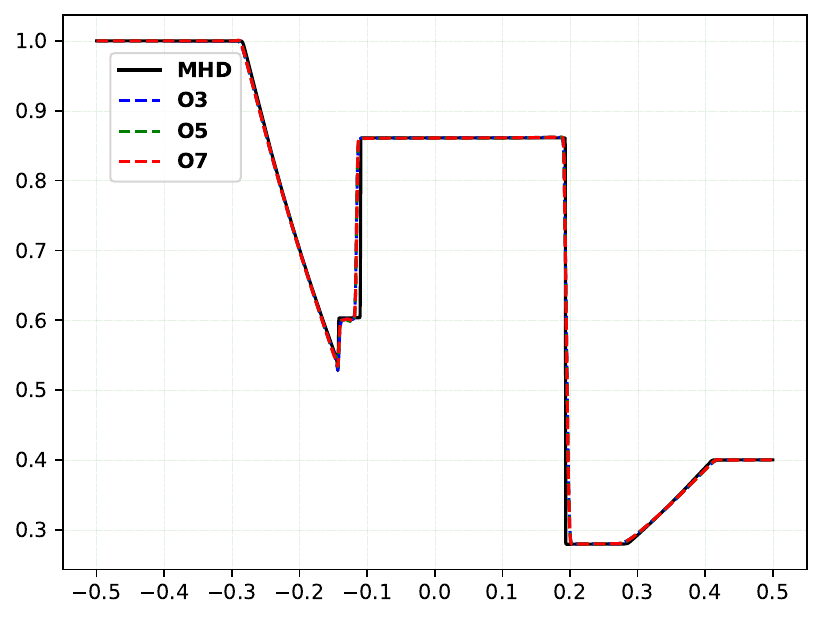}}\\
	\subfigure[$B_y$ (without source term)]{\includegraphics[width=2.0in, height=1.5in]{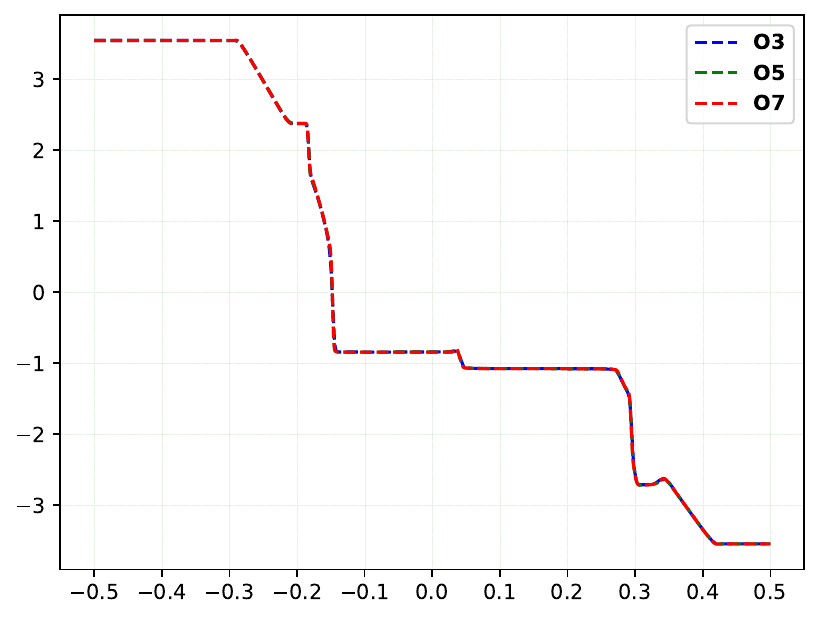}}
	\subfigure[$B_y$ (with source term)]{\includegraphics[width=2.0in, height=1.5in]{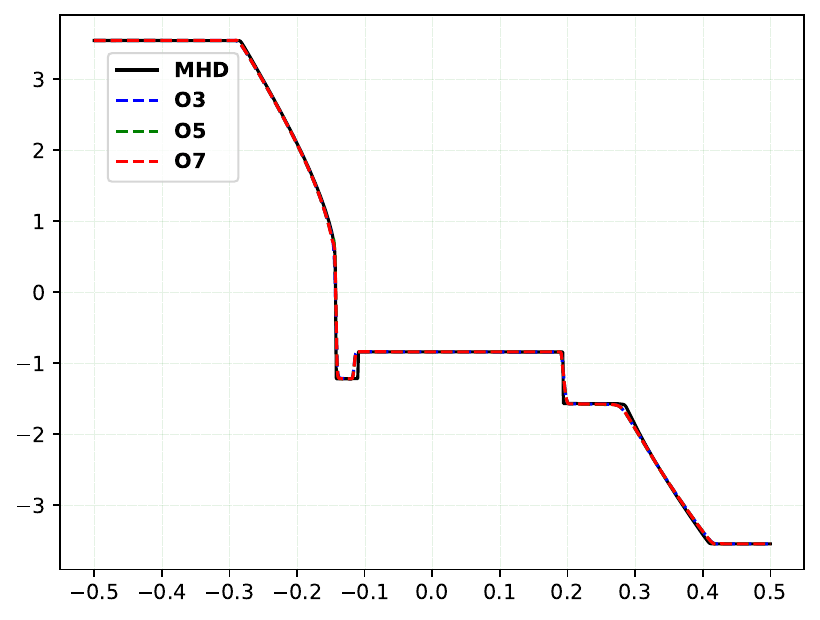}}
	\caption{\textbf{\nameref{test:rp4}:} Plots of density, parallel and perpendicular pressure components and magnetic field in $y$ direction for 3rd, 5th and 7th order numerical schemes without and with source term using the HLLI Riemann solver and 800 cells at final time t = 0.15.}
	\label{fig:6.1}
\end{center}
\end{figure}

\subsection{Riemann Problem 5}
\label{test:rp5}
In this test case, we again consider the CGL generalization (\cite{singh2024entropy}) of a MHD Riemann problem (\cite{Balsara2018efficient}). The computational domain is taken to be $[-0.5,0.5]$ with outflow boundary conditions. The initial conditions are taken to be,
\[(\rho, u_{x}, u_{y}, u_{z}, \pll, \per, B_{y}, B_{z}) = \begin{cases}
\left(\frac{1}{4\pi}, -1, 1, -1, 1, 1, -1, 1\right), & \textrm{if } x\leq 0\\
\left(\frac{1}{4\pi}, -1, -1, -1, 1, 1, 1, 1\right), & \textrm{otherwise}
\end{cases}\]
with $B_{x}=1$. We again consider $800$ cells and compute till the final time of $t=0.1$.
\begin{figure}[!htbp]
\begin{center}
	\subfigure[$u_y$ (without source term)]{\includegraphics[width=2.0in, height=1.5in]{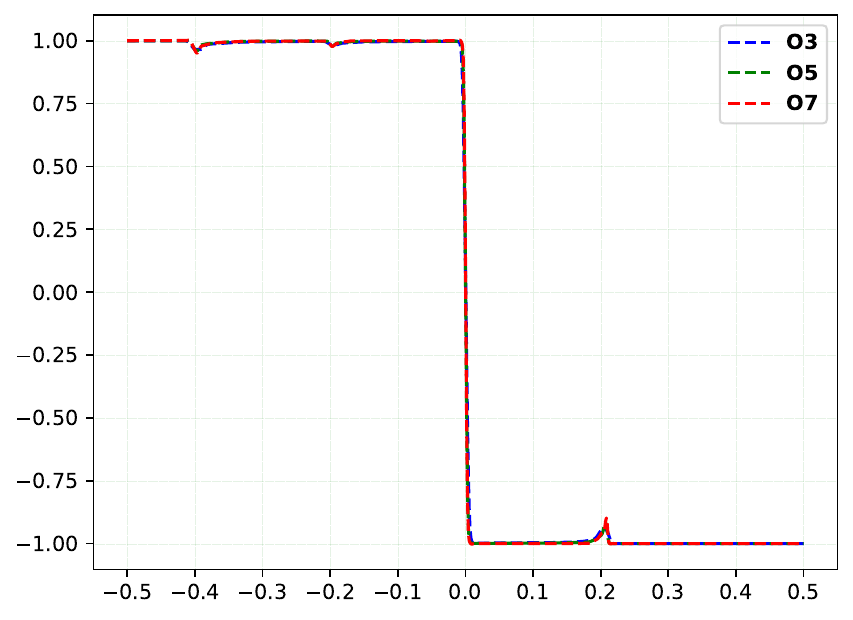}}
    \subfigure[$u_y$ (with source term)]{\includegraphics[width=2.0in, height=1.5in]{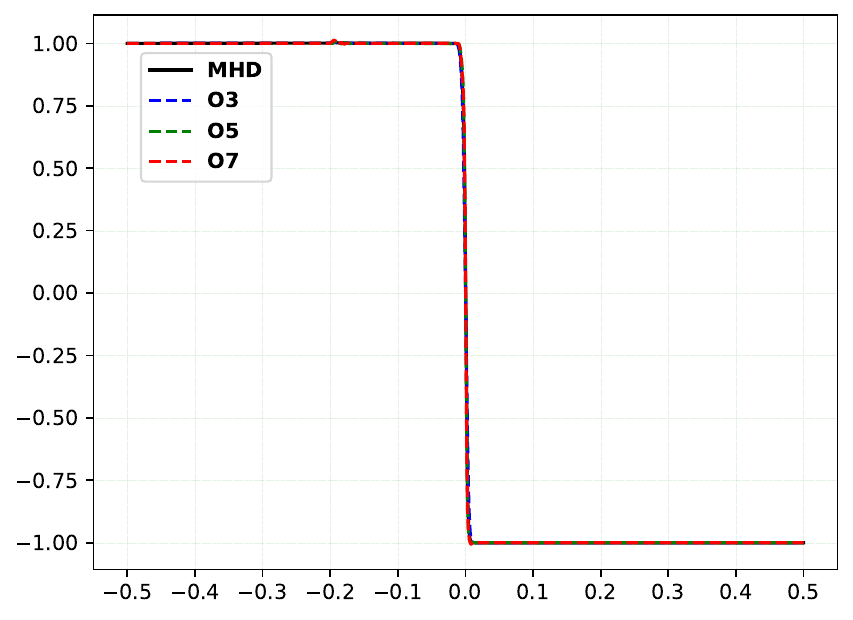}}\\
	\subfigure[$\pll$ (without source term)]{\includegraphics[width=2.0in, height=1.5in]{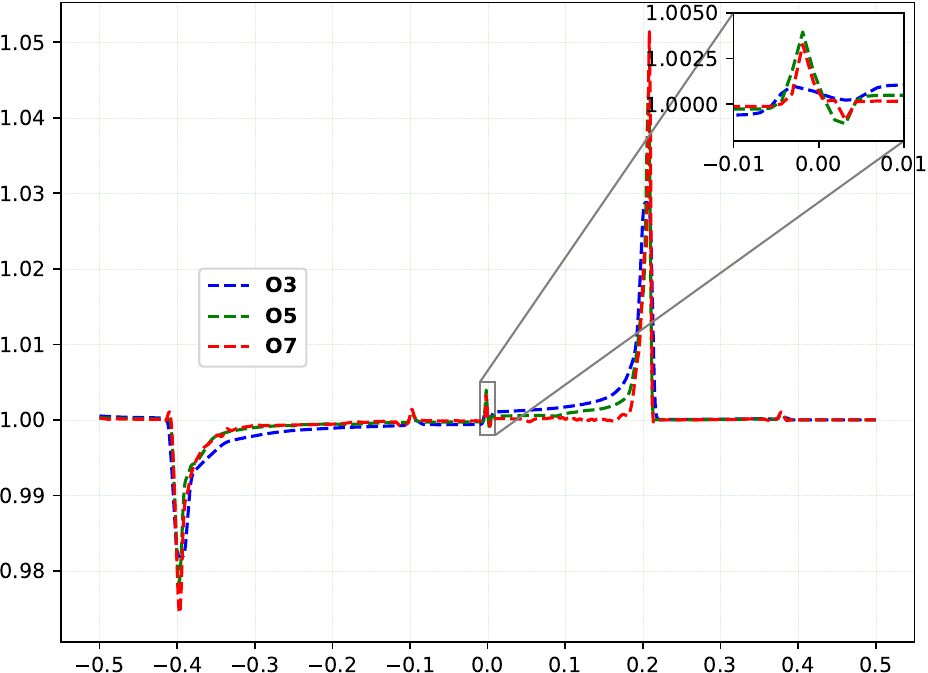}}
    \subfigure[$\pll$ (with source term)]{\includegraphics[width=2.0in, height=1.5in]{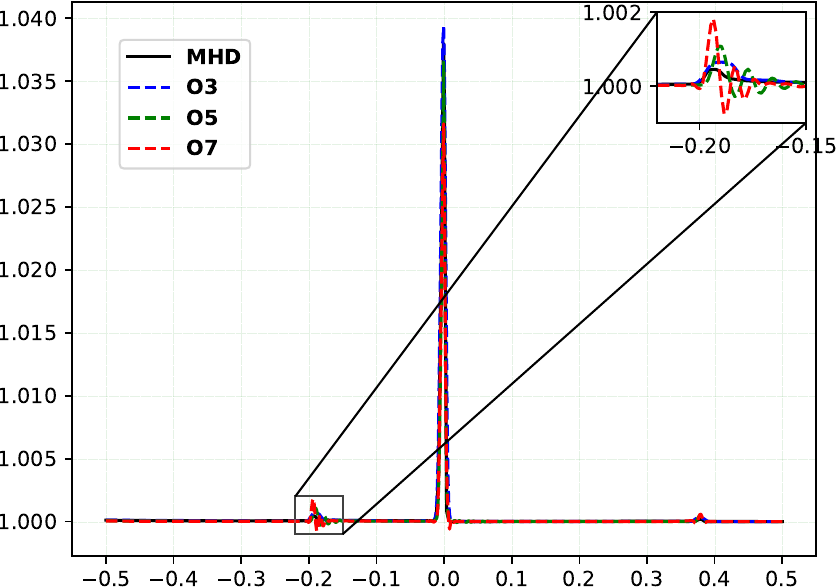}}\\
	\subfigure[$\per$ (without source term)]{\includegraphics[width=2.0in, height=1.5in]{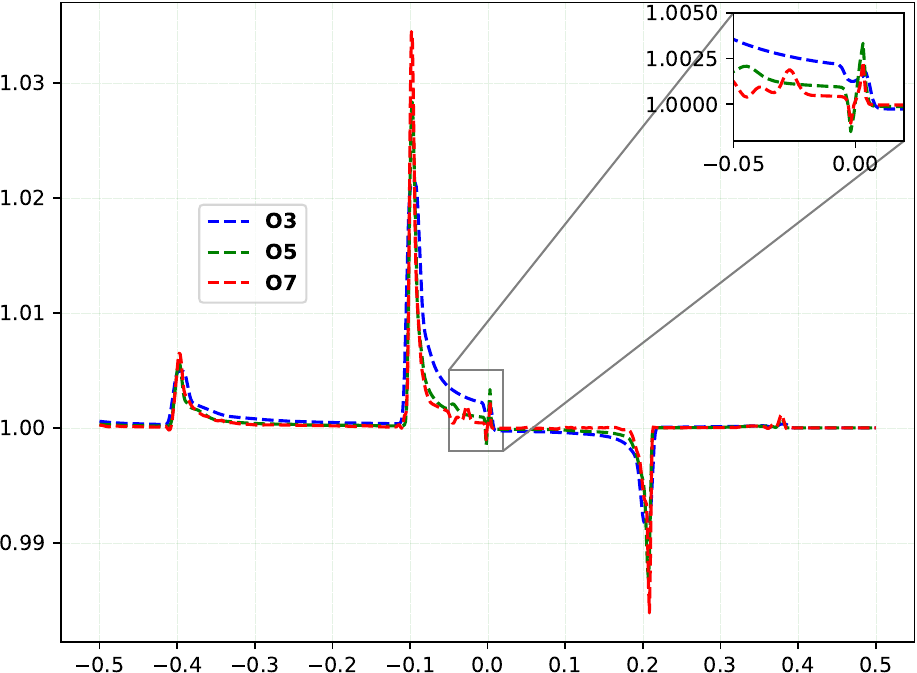}}
    \subfigure[$\per$ (with source term)]{\includegraphics[width=2.0in, height=1.5in]{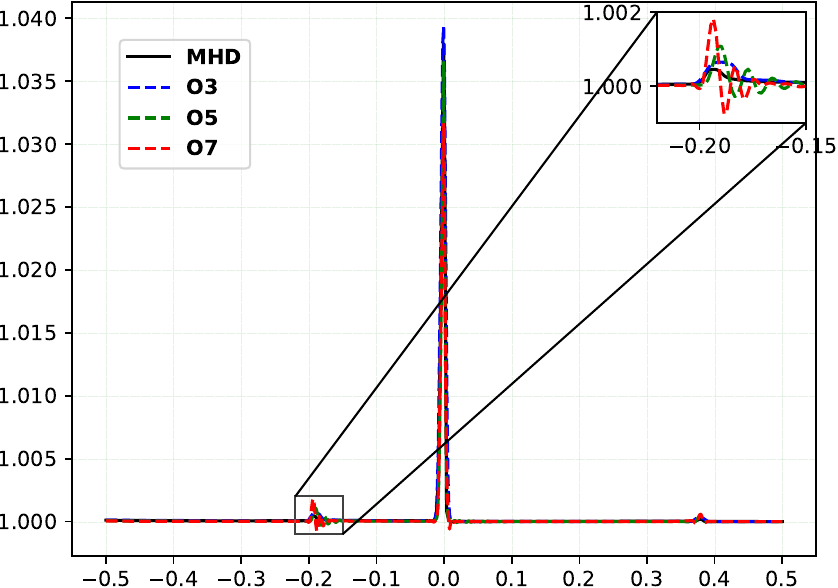}}\\
	\subfigure[$B_y$ (without source term)]{\includegraphics[width=2.0in, height=1.5in]{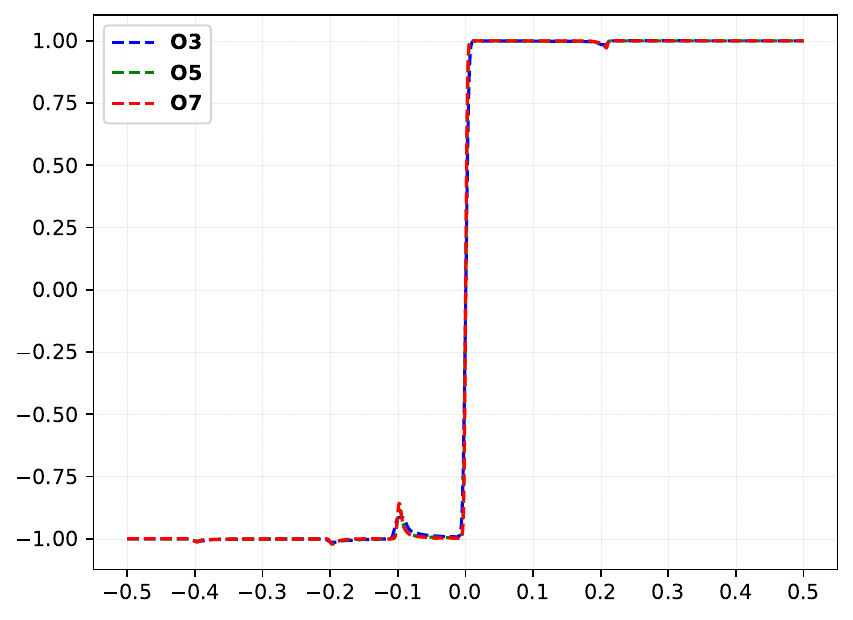}}
	\subfigure[$B_y$ (with source term)]{\includegraphics[width=2.0in, height=1.5in]{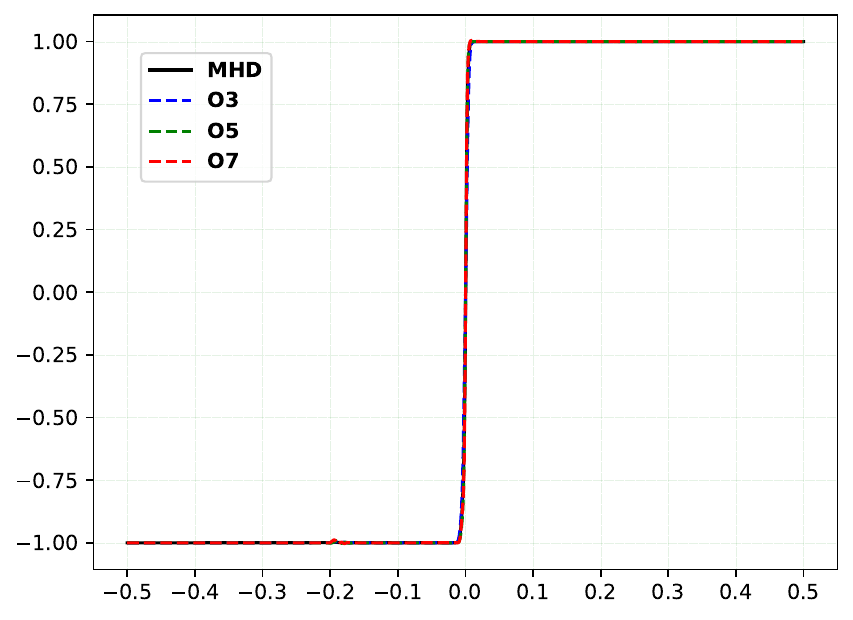}}
	\caption{\textbf{\nameref{test:rp5}:} Plots of velocity in $y$ direction, parallel and perpendicular pressure components and magnetic field in $y$ direction for 3rd, 5th and 7th order numerical schemes without and with source term using the HLL Riemann solver and 800 cells at final time t = 0.1.}
	\label{fig:7}
\end{center}
\end{figure}
\begin{figure}[!htbp]
\begin{center}
	\subfigure[$u_y$ (without source term)]{\includegraphics[width=2.0in, height=1.5in]{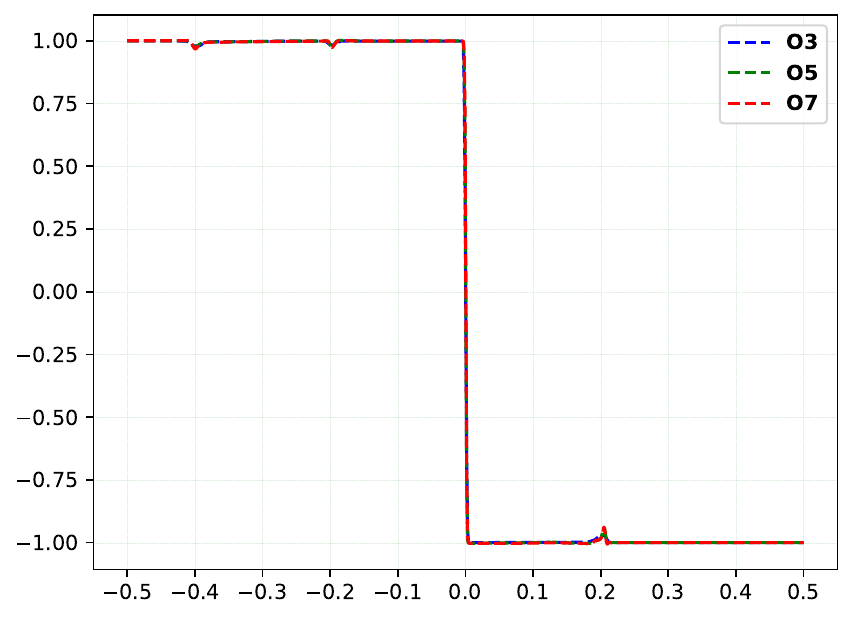}}
    \subfigure[$u_y$ (with source term)]{\includegraphics[width=2.0in, height=1.5in]{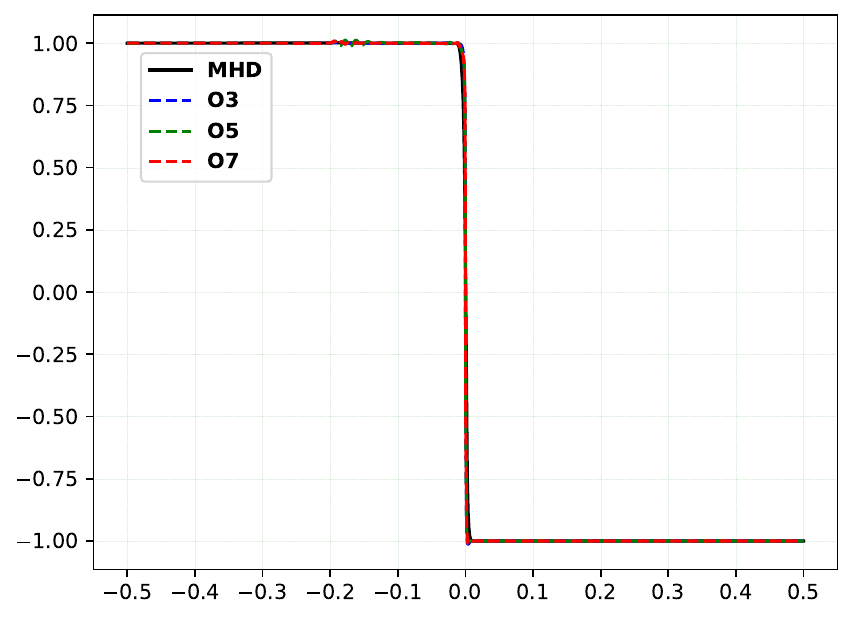}}\\
	\subfigure[$\pll$ (without source term)]{\includegraphics[width=2.0in, height=1.5in]{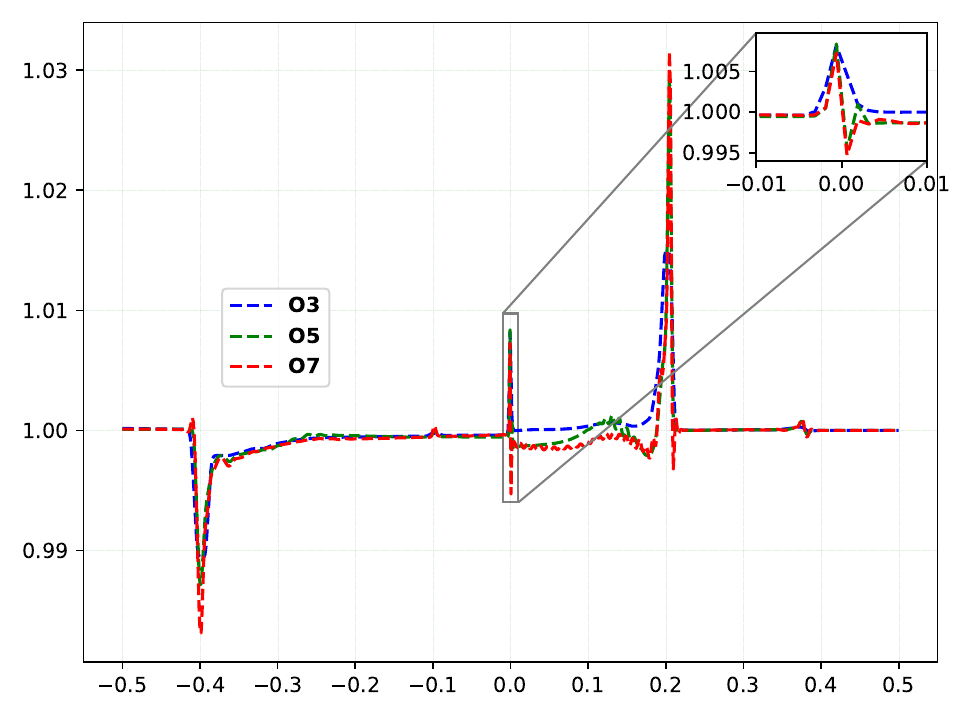}}
    \subfigure[$\pll$ (with source term)]{\includegraphics[width=2.3in, height=1.7in]{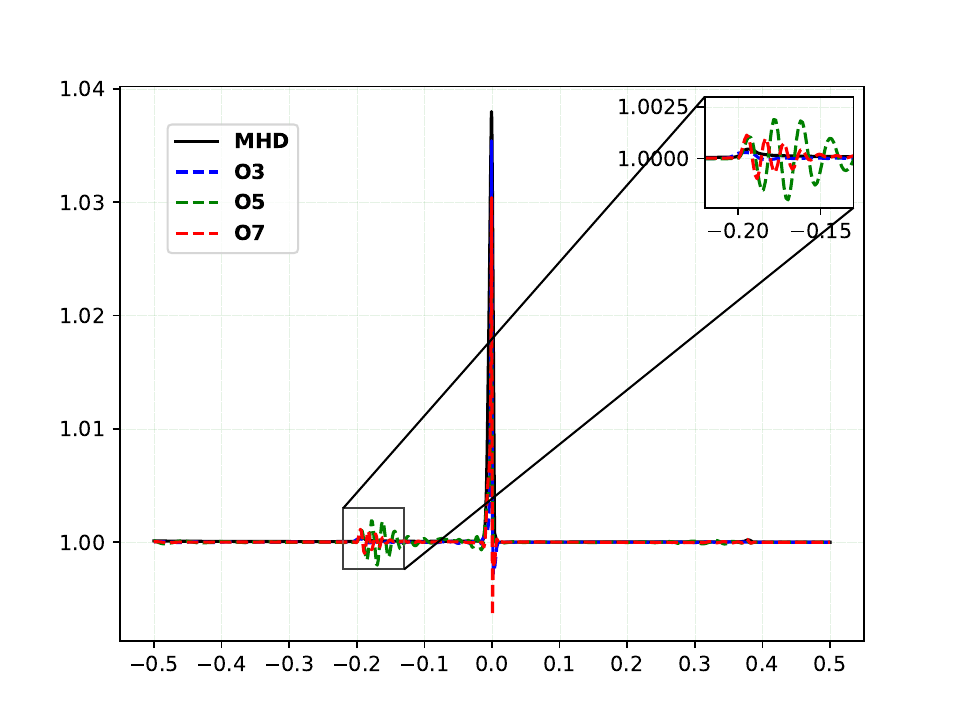}}\\
	\subfigure[$\per$ (without source term)]{\includegraphics[width=2.0in, height=1.5in]{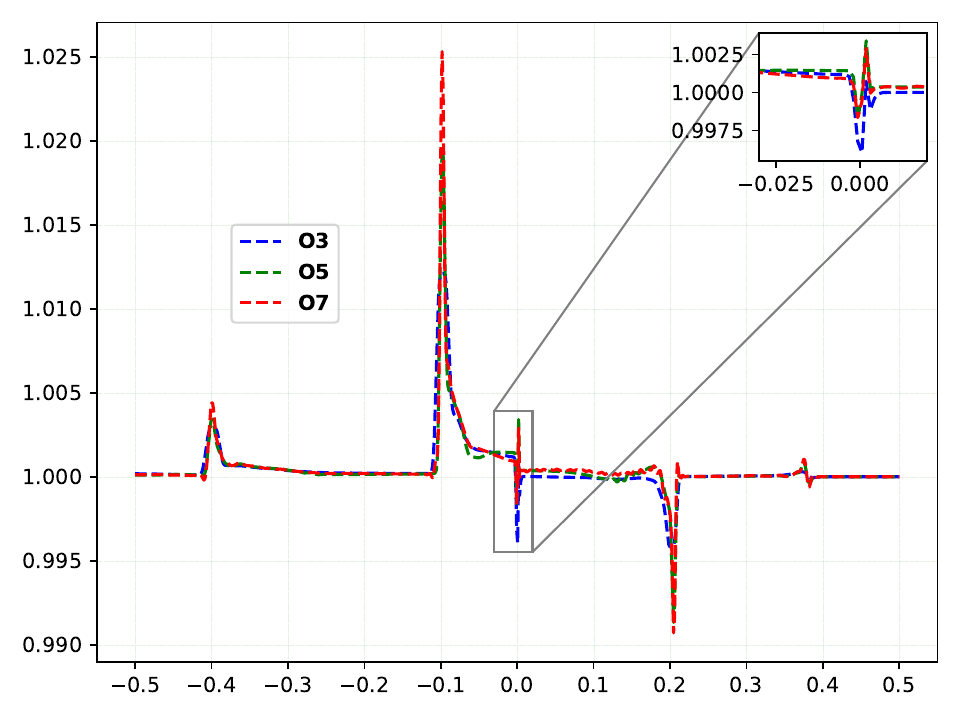}}
    \subfigure[$\per$ (with source term)]{\includegraphics[width=2.3in, height=1.7in]{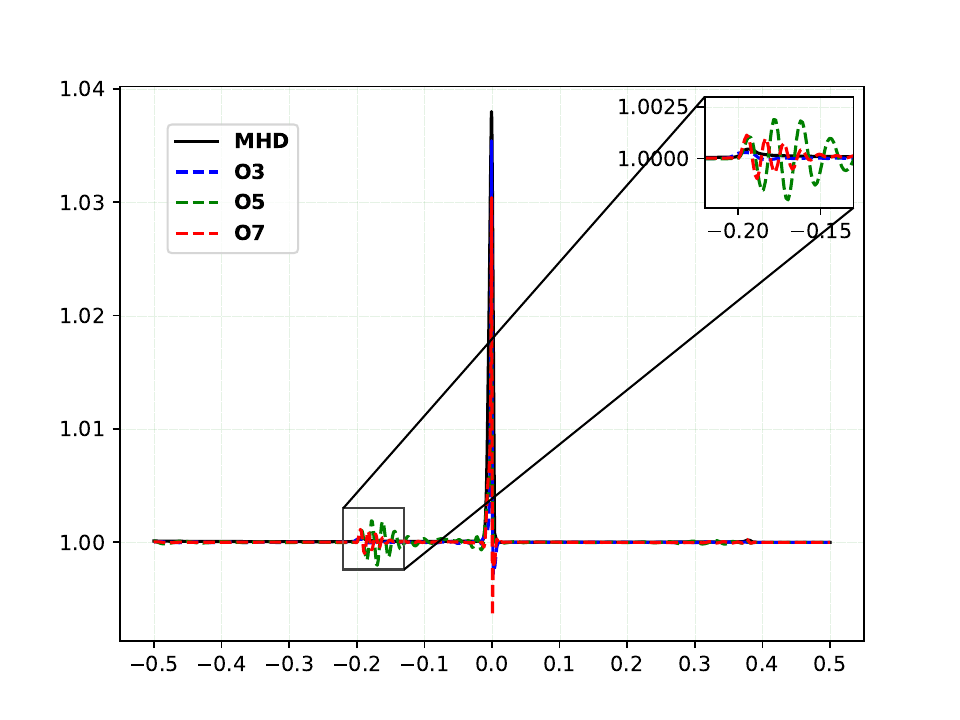}}\\
	\subfigure[$B_y$ (without source term)]{\includegraphics[width=2.0in, height=1.5in]{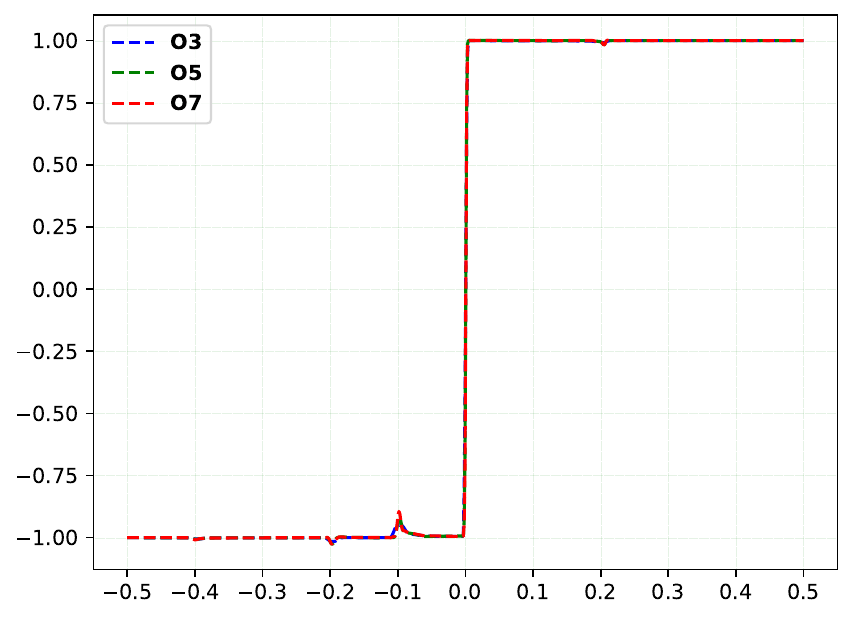}}
	\subfigure[$B_y$ (with source term)]{\includegraphics[width=2.0in, height=1.5in]{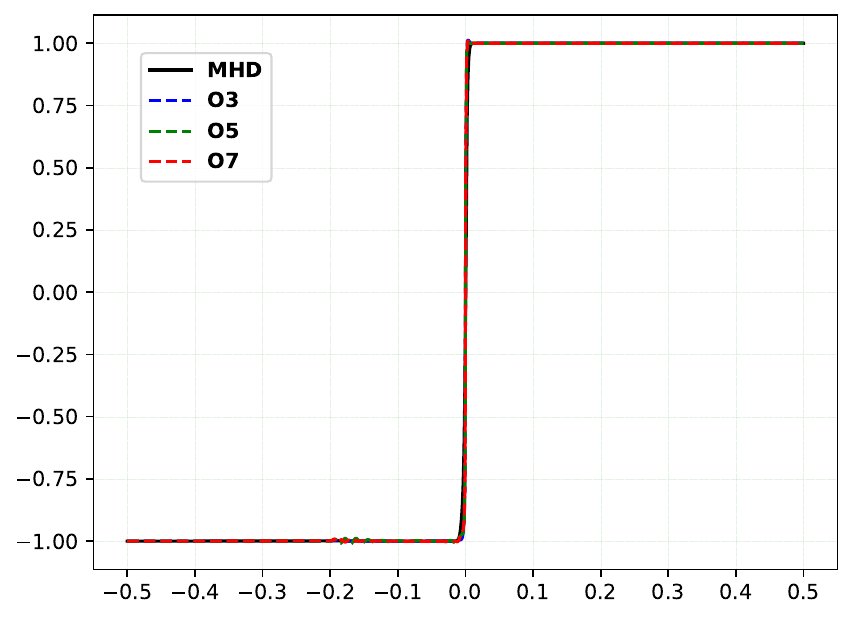}}
	\caption{\textbf{\nameref{test:rp5}:} Plots of velocity in $y$ direction, parallel and perpendicular pressure components and magnetic field in $y$ direction for 3rd, 5th and 7th order numerical schemes without and with source term using the HLLI Riemann solver and 800 cells at final time t = 0.1.}
	\label{fig:7.1}
\end{center}
\end{figure}
The numerical results are presented in Figures $\eqref{fig:7}$ and $\eqref{fig:7.1}$. We observe that the schemes using both solvers can resolve small waves moving outward. We also note that the 5th and 7th-order schemes are oscillatory. Furthermore, in the isotropic case, both pressure components match with MHD pressure. We do not observe any significant difference in HLL and HLLI solvers in this test case. The results are similar to those presented in \cite{singh2024entropy}.

\section{Conclusion}
\label{sec:con}
CGL equations, also known as the double-adiabatic model, are used to model plasma flows in the presence of a strong magnetic field. These equations are a set of hyperbolic PDEs with non-conservative products. In this article, we have presented the eigenvalues and the complete set of associated eigenvectors. We have also proved the linear degeneracy of some of the characteristic fields. \rev{Using the eigenvalues and eigenvectors, we also design the HLL and HLLI Riemann solver for the CGL equations}. Finally, using the eigensystem, we design very high-order AFD-WENO schemes in one dimension. We then test the proposed schemes on an extensive set of test cases. Several of these cases are motivated by MHD test cases. To capture the isotropic (MHD)limit, we also consider a stiff source term, which is treated using IMEX schemes. The extension of these schemes to higher dimensions, however, needs a proper treatment of divergence-free constraint on the magnetic field. Furthermore, the eigenvectors of the CGL equations have several degeneracies that need to be adequately treated for stable numerical schemes in higher dimensions. We plan to address these issues in future work.

\section*{Acknowledgements}
The work of Harish Kumar is supported in parts by VAJRA grant No. VJR/2018/000129 by the Dept. of Science and Technology, Govt. of India. Harish Kumar and Chetan Singh acknowledge the support of FIST Grant Ref No. SR/FST/MS-1/2019/45 by the Dept. of Science and Technology, Govt. of India.

\printcredits

\bibliographystyle{abbrv}

\bibliography{cas-refs}

\end{document}